\documentclass[12pt]{article}
\usepackage{amsfonts}
\usepackage{pst-all} 
\usepackage{amssymb}
\usepackage{latexsym}
\usepackage{amsmath}
\usepackage{amscd}
\usepackage{amsthm}
\usepackage{graphicx}
\usepackage{grffile}
\usepackage{epsfig}
\usepackage{hyperref}
\usepackage{psfrag}
\usepackage{url}
\newtheorem{theorem}{Theorem}[section]
\newtheorem{definition}[theorem]{Definition}

\newtheorem{proposition}[theorem]{Proposition}
\newtheorem{notation}[theorem]{Notation}
\newtheorem{claim}[theorem]{Claim}
\newtheorem{lemma}[theorem]{Lemma}
\newtheorem{remark}[theorem]{Remark}
\newtheorem{corollary}[theorem]{Corollary}

\newtheorem{prop}[theorem]{Proposition}
\newtheorem{corol}[theorem]{Corollary}
\newcommand{\E}{{\mathbb E}}
\newcommand{\N}{{\mathbb N}}
\newcommand{\Z}{{\mathbb Z}}
\renewcommand{\P}{{\mathbb P}}
\newcommand{\Q}{{\mathbb Q}}
\newcommand{\R}{{\mathbb R}}

\newcommand{\x}{{\mathbf x}}
\newcommand{\y}{{\mathbf y}}
\renewcommand{\u}{{\mathbf u}}
\renewcommand{\v}{{\mathbf v}}
\newcommand{\w}{{\mathbf w}}

\newcommand{\WW}{{\mathcal W}}

\renewcommand{\qed}{$\blacksquare$}

\usepackage[normalem]{ulem}
\usepackage{color}
\title{Transmission and navigation on disordered lattice networks, directed spanning forests and Brownian web}
\author{
	\begin{tabular}{c}
		{Subhroshekhar Ghosh}  
		 \footnote{Department of Mathematics, National University of Singapore, Singapore 119076.}$^{\hspace{5pt}, }$\footnote{Dept. of Statistics and Applied Probability, National University of Singapore, Singapore 117546.}
		 \\ subhrowork@gmail.com
	\end{tabular}
	\and
	\begin{tabular}{c}
		{Kumarjit Saha}
		  \footnote{Dept. of Mathematics, Ashoka University, N.C.R., Sonepat, Haryana 131029, India.}
		  \\ kumarjit.saha@ashoka.edu.in
	\end{tabular}
}
\date{}

\begin{document}

\maketitle

\begin{abstract}
Stochastic networks based on random point sets as nodes have attracted considerable interest in many applications, particularly in communication networks, including wireless sensor networks,  peer-to-peer networks and so on. The study of such networks generally requires the nodes to be independently and uniformly distributed as a Poisson point process. In this work, we venture beyond this standard paradigm and investigate  the stochastic geometry of  networks obtained from  \textit{directed spanning forests} (DSF) based on randomly perturbed lattices, which have desirable statistical properties as a models of spatially dependent point fields. In the regime of low disorder, we show in 2D and 3D that the DSF almost surely consists of a single tree. In 2D, we further establish that the DSF, as a collection of paths, converges under diffusive scaling to the Brownian web.

%We consider a perturbed point process $V$ given by $\{\v + U_\v : \v \in \Z^d\}$, where 
%$\{U_\v : \v \in \Z^d\}$ denotes an i.i.d. collection of random variables such that each 
%random variable is uniformly distributed over the region $[-\delta, \delta]^d$. The \textit{directed spanning forest} (DSF) along the direction  $e_d$ on the perturbed point process $V$ 
%gives a graph where each point $\x \in V$ connects to the nearest point with strictly 
%higher $d$ coordinate. In this paper we show that for small regime of perturbation, more specifically for 
%$\delta < 1/8$, the perturbed DSF consists of a single tree with probability $1$. We also show that for $d=2$, as a collection of paths, under diffusive scaling the perturbed DSF converges to the 
%Brownian web.
\end{abstract}

\section{Introduction and main results}
Spatial networks have long been an important class of models for understanding the large scale behaviour of systems in a wide array of applications. These include, but are not limited to, transport networks, power grids, various kinds of social networks, different types of  communication networks including wireless sensor networks, multicast communication networks, peer-to-peer networks and drainage networks, to name a few. In the mathematical study of such networks, an important modelling hypothesis is the random distribution of their nodes. This often serves to capture the macroscopic properties of highly complex networks, in addition to facilitating theoretical analysis. For a partial overview of the literature, we refer the reader to \cite{BB09}, \cite{BB10}, \cite{B11}, \cite{P03} and the references therein.

In the context of communication networks, the study of particular structures like \textit{radial spanning trees} and \textit{directed spanning forests} have gained considerable attention (see, e.g., \cite{BB07}, \cite{B08}, \cite{BKR99}, \cite{EGHK99}, and the references contained therein). These structures often represent broadly related concepts, and in fact, DSFs can be seen as the limit of radial spanning trees far away from the origin. Such structures help in the study of localized co-ordination protocols in networks which are also aimed to be scalable with network size. These have applications to a wide variety of problems, including small world phenomena, computational geometry, decentralised navigation in networks, to mention a few (for details, we refer the interested reader to \cite{K00}, \cite{KSU99}, \cite{PRR99}, and the references therein). In summary, network structures such as the DSF are important theoretical models to study fundamental questions of transmission and navigation on real-world networks.  
%For greater details on the specific applications in the literature, we refer the interested reader to .... and the references therein. 

Generally speaking, the distribution of the random nodes in stochastic networks is taken to be the independent and uniform over space, in other words, the Poisson distribution and its variants (see, e.g., \cite{MR96}, \cite{P03}). The Poisson model is highly amenable to rigorous mathematical treatment, but is often limited in its effectiveness as a model - e.g., on a global scale the homogeneous Poisson process exhibits clusters of points interspersed with vacant spaces, whereas a more \textit{spatially uniform} distribution might be a closer representation of ground realities (see, e.g., \cite{GL17}). However, little is understood about the stochastic geometry of networks arising from such strongly correlated point processes, principally because the tools and techniques for studying the Poisson model overwhelmingly rely on  its exact spatial independence. 
%For negatively associated processes, an important class of models in the strongly correlated domain, even arguments involving FKG-type ideas are generally not applicable

In this work, we examine spatial network models, specifically directed spanning forests, on random point sets that are obtained as disordered lattices on Euclidean spaces. Such point processes exhibit much greater measure of \textit{spatial homogeneity} compared to the Poisson process on one hand, while still retaining a measure of analytical tractability on the other. A powerful manifestation of their relative orderliness is the fact that they are \textit{ hyperuniform}. Hyperuniformity of point processes have attracted a lot of interest in recent years, especially in the statistical physics literature (see, e.g., \cite{T02}, \cite{TS03}, \cite{GL17} and the references therein). A point process is said to be hyperuniform if the variance of the number of points in an expanding domain scales like its surface area (or slower), rather than its volume, which is the case for Poisson or any other \textit{extensive} system that exhibits FKG-type properties. In fact, hyperuniformity is closely related to negative association at the spatial level, which precludes the application of many arguments that are ordinarily staple in stochastic geometry.  In the subsequent  paragraphs, we lay out the details of the model and give an account of our principal results.

%Thus, in this paper we explore the stochastic geometry of strongly correlated spatial networks exhibiting negative spatial dependence, and answer some natural questions that arise in that setting.

We consider a disordered, or perturbed, version of the $d$ dimensional Euclidean lattice $\Z^d$. 
Consider the $d$-dimensional (closed) box centred at the origin $[-1, + 1]^d$.
Let $\{U_\w : \w \in \Z^d \}$ denote a collection of i.i.d. random variables 
(r.v.) such that each r.v. is uniformly distributed over the region $[-1, + 1]^d$.
For $\x  := (\x(1),\x(2),\dotsc,\x(d)) \in \R^d$ let $\x(i)$ denote the $i$-th coordinate of $\x$.
For a lattice vertex $\u \in \Z^d$, the corresponding 
perturbed point is given by $\u^\prime := \u + U_\u $.

\noindent \textbf{The model : }
The set  of (randomly) perturbed points, referred to as the vertex set $V$, is defined as
$V := \{ \w + U_\w : \w \in \Z^d \}$.

Here and henceforth, the quantity $  ||\x ||_{p} $ for $p \geq 1$ denotes  the $l_p$ norm of $\x$ in $ \R^d$.  For $\x \in \R^d$, the notation $h(\x)$ denotes the closest 
point in $V$ with respect to the $||\quad||_1$ distance
 with strictly higher $d$-th coordinate. Formally 
\begin{align}
\label{def:h-step}
h(\x) := \text{argmin} \{ ||\y-\x||_1 : \y \in V, \y(d) > \x(d) \}.
\end{align}
We highlight the fact that for any $\x \in \R^d$, the point $h(\x)$ is defined and 
it is unique almost surely (a.s.). 
We call the point $h(\x)$ as the next step starting from the point $\x$.
For $\x \in V$, the edge joining $\x $ and $h(\x)$ is denoted by $\langle \x, h(\x)\rangle$ and $E := \{\langle \x, h(\x)\rangle : \x \in V \}$ denotes the set of all edges. 
Let $\{ e_1, \cdots , e_d\}$ denote the standard orthonormal basis for $\R^d$.
The directed spanning forest (DSF) on $V$ with direction $e_d$ is the random graph $G := (V, E)$ consisting of the vertex set $V$ and the edge set $E := \{\langle \x, h(\x)\rangle : \x \in V \}$. By construction, each vertex $\x$ has exactly one outgoing edge, viz., $\langle \x, h(\x)\rangle$ and hence the random graph $G$ does not have a loop or cycle a.s. 

In this paper, we study the random graph $ G:= (V, E)$, 
which we will refer to as the {\em perturbed DSF}.  
We will study it only for dimensions $2$ and $3$.   
Our first main result shows that the random graph $G$ is connected a.s.

\begin{theorem}
\label{thm:TreeDSF}
For $d=2$ and $d=3$ the random graph $G$ is connected  and consists of a single tree a.s.
\end{theorem}

The study of the directed spanning forest (DSF) on the Poisson point processes was initiated in \cite{BB07}. The intricate dependencies caused by the construction of edges based on Euclidean distances, makes this model hard to study. In fact, the question posed by Baccelli and Bordenave \cite{BB07}
regarding connectivity of the Poisson DSF  
remained open for quite some time and finally, Coupier and Tran \cite{CT13} 
  proved that for $d=2$ the DSF is a tree almost surely. 
Their argument is based on a  Burton-Keane type argument and crucially depends on the planarity structure of $\R^2$ and can not be applied for higher dimensions. In this context, it is useful 
to mention here that a discrete directed spanning forest created on a random subset of $\Z^d$
was studied in \cite{RSS16} for all $d \geq 2$. It was proved that for $d=2, 3$, the discrete 
DSF is connected a.s.   and for $d\geq 4$ it consists of infinitely many disjoint trees.
The discrete DSF model also enjoys a Poisson point process type behaviour, viz., outside the 
explored region, vertices are independently distributed, a crucial property which we don't have 
for the perturbed lattice points considered here.

There are other random directed graphs studied in the literature 
for which the dichotomy in dimensions of having a single connected tree vis-a-vis a forest has been studied (see  \cite{FLT04},  \cite{GRS04},  \cite{ARS08}). However, the mechanisms
used for construction of  edges for all these models mentioned above, incorporate much 
more independence than that is available for the
DSF. For these models it has been  proved that the random graph is a connected tree in dimensions $2$ and
$3$, and a forest with infinitely many tree components in dimensions $4$ and more.
It is important to observe that for all the above mentioned models, the vertices are independently distributed over disjoint regions - at small as well as large mutual separations -
and this property was crucially used in the analysis of these models.
 On the other hand, for the disordered lattice models  this property no longer holds true in general,  even in the regime of low disorder, and we require new stochastic geometric techniques to overcome the difficulties posed by the long-ranged dependencies arising there from.
It is useful to mention here that though the perturbed DSF is constructed based on $||\quad||_1$ 
distance only, because of small regime of perturbations, similar arguments hold for $||\quad||_p$ for any $p \geq 1$.

Our next main result is  that for $d=2$ the random graph $G$ observed as a collection of paths,
converges  to the Brownian web under a suitable diffusive scaling.
The standard Brownian web originated in
the work of Arratia \cite{A79}, \cite{A81}
as the scaling limit of the voter model on $\Z$. It arises naturally as the diffusive
scaling limit of the coalescing simple random walk paths starting from every point
on the oriented lattice $\Z^2_{\text{even}} := \{(m,n) : m + n 
\text{ even} \}$. Intuitively, the Brownian web can be thought of as a collection
of one-dimensional coalescing Brownian motions starting from every point in the space time
plane $\R^2$. Later  Fontes \textit{et. al.} \cite{FINR04} provided a framework in
which the Brownian web is realized as a random variable taking values in a Polish space.
In the next section we present the relevant topological details from \cite{FINR04}.

For any $\x \in \R^d$, we set $h^0(\x) = \x$ and for $k \geq 1$, let 
$h^{k}(\x) := h(h^{k-1}(\x))$. We consider the random graph $G$ for $d=2$. 
For $\x \in V$, the path starting from $\x$ denoted by $\pi^\x$ is obtained by 
linearly joining the successive steps $h^k(\x)$ for $k \geq 1$ and hence 
we have $\pi^{\u}(h^{k}(\u)(2)) = h^{k}(\u)(1)$ for every $k \geq 0$.
Let ${\cal X} := \{\pi^{\u}:\u \in V\}$ denote the collection of all DSF paths.
For given  $\gamma , \sigma > 0$ and for any $n \in \N$, the $n$-th scaled version of 
path $\pi$ is given by $\pi_n(\gamma,\sigma) (t)= \pi(n^2 \gamma t)/n \sigma$ and 
 ${\cal X}_n(\gamma,\sigma) :=
\{\pi_n^{\u}(\gamma,\sigma):\u \in V\}$ denotes the collection of all the scaled paths.
Let $\overline{{\cal X}_n}(\gamma,\sigma)$ denote the closure of ${\cal X}_n(\gamma,\sigma)$
w.r.t. certain metric $d_\Pi$ as in \cite{FINR04}. In Section \ref{sec:BwebIntro}
we will explain this metric $d_\Pi$ in detail. 
Now we are ready to state our second theorem regarding convergence to the Brownian web. 
\begin{theorem}
\label{thm:BW}
There exist $\sigma > 0$ and $\gamma > 0$ such
that as $n\rightarrow \infty$,  $\overline{{\cal X}}_n(\gamma,\sigma)$ converges weakly to the
standard Brownian Web ${\cal W}$.
\end{theorem}
In the next section we will describe the topology of convergence in detail.
Note that one can combine both $\sigma$ and $\gamma$ to obtain a single normalization constant. 
But we prefer to keep both as we believe that it is easier to interpret these constants that way.
 
Ferrari \textit{et. al.} \cite{FFW05}  showed that, for $d=2$,  
the random graph on the Poisson points introduced by \cite{FLT04},
converges to the Brownian web under a suitable diffusive scaling. 
Coletti \textit{et. al.} \cite{CFD09}
proved a similar result for the discrete random graph studied in \cite{GRS04}.
For the discrete DSF considered in \cite{RSS16}, Roy \textit{et. al.}
proved that suitably normalized diffusively scaled paths converges to the 
Brownian web. For the discrete DSF considered in \cite{RSS16}, the paths are
 non-crossing. Another related discrete DSF model where paths can cross each other,
was studied in \cite{VZ17} and showed to converge to the Brownian web. 
In \cite{BB07} it has been shown that scaled paths of the
successive ancestors in the DSF converges weakly to the Brownian motion
and also conjectured that the scaling limit of the DSF on the Poisson points 
is the  Brownian web.  This conjecture remained open for a long time
 and very recently it has been proved in \cite{CSST19}.
The scaling limit of the collection of all paths is a much harder question, as one
 has to deal with the dependencies between different paths.
In this paper we show that the perturbed DSF, which is created on a dependent 
random environment due to the perturbed lattice points, 
belongs to the basin of attraction of the Brownian web. We 
actually prove a stronger result in the sense that we define a dual for the perturbed DSF and show that, under diffusive scaling the perturbed DSF and it's dual jointly converge in  distribution to the Brownian web and its dual. This joint convergence further allows us to show that for $d=2$, there is no bi-infinite path in the perturbed DSF a.s.   

The paper is organized in the following way. In Section \ref{sec:BwebIntro}
we recall the topological 
details  for convergence to the Brownian web from \cite{FINR04}.
In Section \ref{sec:MarkovProp}, a discrete time joint exploration process is introduced to describe the joint evolution of the DSF paths and we define random renewal steps for the joint 
exploration process. In Section \ref{sec:ProofTauTail} we prove that the number of steps between any two 
consecutive renewal steps decay exponentially.  
In Section \ref{sec:RW_Martingale} we show that at these renewal steps,
 the exploration process can be restarted
in some sense and for the single DSF path process these renewal steps allow 
us to break the trajectory into i.i.d. blocks of increments.  
For multiple paths the difference between the restarted paths observed at these renewal steps 
gives a Markov chain which behaves like a random walk 
away from the origin. Using this in Section \ref{sec:Trees}, 
we prove Theorem \ref{thm:TreeDSF}. An important ingredient is Proposition \ref{prop:DSF_CoalescingTimeTail}, which gives an estimate for coalescing time tail 
decay for two DSF paths for $d=2$. This estimate is crucially used in Section \ref{sec:cvBW} to prove
Theorem \ref{thm:BW}, i.e., the scaled DSF converges to the Brownian web. In Section \ref{sec:cvBW},
 we construct a dual for the perturbed DSF and prove that the DSF and its dual jointly converge to the Brownian web and its dual (Theorem \ref{thm:BW_Joint}).  

\section{The Brownian web}
\label{sec:BwebIntro}

 Fontes \textit{et. al.} \cite{FINR04} provided a suitable framework so that the Brownian 
web can be regarded as a random variable taking values in a Polish space.
In this section, we recall the relevant topological details from \cite{FINR04}.

Let $\R^{2}_c$ denote the completion of the space time plane $\R^2$ with
respect to the metric
\begin{equation*}
\rho((x_1,t_1),(x_2,t_2))=|\tanh(t_1)-\tanh(t_2)|\vee \Bigl| \frac{\tanh(x_1)}{1+|t_1|}
-\frac{\tanh(x_2)}{1+|t_2|} \Bigr|.
\end{equation*}
% I think it is good to include this line,
As a topological space $\R^{2}_c$ can be identified with the
continuous image of $[-\infty,\infty]^2$ under a map that identifies the line
$[-\infty,\infty]\times\{\infty\}$ with the point $(\ast,\infty)$, and the line
$[-\infty,\infty]\times\{-\infty\}$ with the point $(\ast,-\infty)$.
A path $\pi$ in $\R^{2}_c$ with starting time $\sigma_{\pi}\in [-\infty,\infty]$
is a mapping $\pi :[\sigma_{\pi},\infty]\rightarrow [-\infty,\infty]$ such that
$\pi(\infty)= \pi(-\infty)= \ast$ and $t\rightarrow (\pi(t),t)$ is a continuous
map from $[\sigma_{\pi},\infty]$ to $(\R^{2}_c,\rho)$.
We then define $\Pi$ to be the space of all paths in $\R^{2}_c$ with all possible starting times in $[-\infty,\infty]$.
The following metric, for $\pi_1,\pi_2\in \Pi$
\begin{equation*}
d_{\Pi} (\pi_1,\pi_2)= |\tanh(\sigma_{\pi_1})-\tanh(\sigma_{\pi_2})|\vee\sup_{t\geq
\sigma_{\pi_1}\wedge
\sigma_{\pi_2}} \Bigl|\frac{\tanh(\pi_1(t\vee\sigma_{\pi_1}))}{1+|t|}-\frac{
\tanh(\pi_2(t\vee\sigma_{\pi_2}))}{1+|t|}\Bigr|
\end{equation*}
makes $\Pi$ a complete, separable metric space. Convergence in this
metric can be described as locally uniform convergence of paths as
well as convergence of starting times.
Let ${\cal H}$ be the space of compact subsets of $(\Pi,d_{\Pi})$ equipped with
the Hausdorff metric $d_{{\cal H}}$ given by,
\begin{equation*}
d_{{\cal H}}(K_1,K_2)= \sup_{\pi_1 \in K_1} \inf_{\pi_2 \in
K_2}d_{ \Pi} (\pi_1,\pi_2)\vee
\sup_{\pi_2 \in K_2} \inf_{\pi_1 \in K_1} d_{\Pi} (\pi_1,\pi_2).
\end{equation*}
The space $({\cal H},d_{{\cal H}})$ is a complete separable metric space. Let
$B_{{\cal H}}$ be the Borel  $\sigma-$algebra on the metric space $({\cal H},d_{{\cal H}})$.
The Brownian web ${\cal W}$ is an $({\cal H},B_{{\cal H}})$ valued random
variable.

Consider the collection of DSF paths ${\cal X} = \{\pi^\x : \x \in V\} \subset \Pi$.
 Two paths $\pi_1 , \pi_2 \in \Pi$
are said to be non-crossing if there does not exist any $s_1, s_2 \in [\sigma_{\pi_1}\vee \sigma_{\pi_2}, 
\infty)$ such that 
\begin{align}
\label{eq:NonCrossing}
(\pi_1(s_1) - \pi_2(s_1))(\pi_1(s_2) - \pi_2(s_2)) < 0.
\end{align}
It follows that DSF paths are non-crossing and further for each $n \geq 1$,  
${\cal X}_n(\gamma, \sigma)$ a.s. forms a non-crossing subset of $\Pi$.  
Finally the closure of ${\cal X}_n(\gamma, \sigma)$ in $(\Pi, d_\Pi)$
denoted by $\overline{{\cal X}_n}(\gamma, \sigma)$ gives an $({\cal H}, {\cal B}_{\cal H})$-valued random variable a.s. We will show that as $n\to \infty$, as $({\cal H}, {\cal B}_{\cal H})$-valued random variables, $\overline{{\cal X}_n}(\gamma, \sigma)$ converges in distribution to the Brownian web ${\cal W}$.

\section{ Joint exploration process}
\label{sec:MarkovProp}

We consider the DSF paths starting from two points $\u, \v \in \Z^d$
with $\u(d) = \v(d)$. Note that the construction with two points automatically takes 
care of the construction of a single DSF path. As shown in Figure \ref{fig:DSFPaths}, given the past 
movements we have the information that interior of the shaded region can not have any perturbed lattice 
point in it and because of that, the process of a single DSF path
$\{ h^n(\u) : n \geq 0 \}$ is \textbf{not} Markov. We need to introduce some notations to
define a joint exploration process of two DSF paths starting from $\u$ and $\v$
so that both the paths move in tandem. 
Later we show that there are random times when this exploration process exhibits renewal properties. 
%WLOG we take $\v_1(d) = 0$. For ease of notation we denote the set of points $\{\v_1 - \delta e_d, \v_2 -\delta e_d, \cdots, \v_k -\delta e_d\}$ as $\{\u_1, \u_2 , \cdots, \u_k \}$. By the choice of $\delta$ it follows that $h(\u_i)  = \v^\prime_i$ for all $1\leq i \leq k$.
   Before we proceed further, it 
 is important to mention that several qualitative results of this paper involve constants.
  For the sake of clarity, we will use $C_0$ and $C_1$ to denote two positive constants, whose exact values may change from one line to the other. The important thing is that both $C_0$ and $C_1$ are universal constants whose values depend only on the dimension $d$. 
  
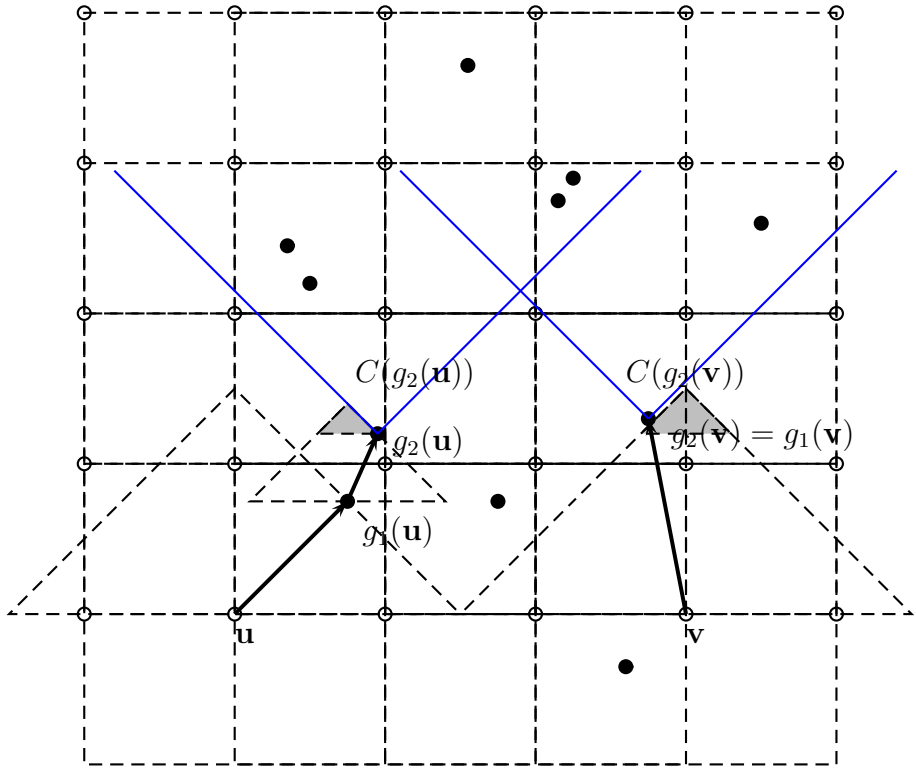
\begin{figure}[ht!]
\label{fig:DSFPaths}
\begin{center}

\begin{pspicture}(0,-2)(10,8.5)
%\psgrid[subgriddiv=0,griddots=5,gridlabels=0]

\pspolygon[linestyle=dashed, fillstyle = solid , fillcolor = lightgray](3.9,2.4) (3.5,2.8)(3.1,2.4)
\pspolygon[linestyle=dashed, fillstyle = solid , fillcolor = lightgray](7.4,2.4) (8.6,2.4)(8,3)

\pscircle[fillcolor=black,fillstyle=none](0,0){.1}
\pscircle[fillcolor=black,fillstyle=none](2,0){.1}
\pscircle[fillcolor=black,fillstyle=none](4,0){.1}
\pscircle[fillcolor=black,fillstyle=none](6,0){.1}
\pscircle[fillcolor=black,fillstyle=none](8,0){.1}
\pscircle[fillcolor=black,fillstyle=none](10,0){.1}
\pscircle[fillcolor=black,fillstyle=none](0,2){.1}
\pscircle[fillcolor=black,fillstyle=none](2,2){.1}
\pscircle[fillcolor=black,fillstyle=none](4,2){.1}
\pscircle[fillcolor=black,fillstyle=none](6,2){.1}
\pscircle[fillcolor=black,fillstyle=none](8,2){.1}
\pscircle[fillcolor=black,fillstyle=none](10,2){.1}
\pscircle[fillcolor=black,fillstyle=none](0,4){.1}
\pscircle[fillcolor=black,fillstyle=none](2,4){.1}
\pscircle[fillcolor=black,fillstyle=none](4,4){.1}
\pscircle[fillcolor=black,fillstyle=none](6,4){.1}
\pscircle[fillcolor=black,fillstyle=none](8,4){.1}
\pscircle[fillcolor=black,fillstyle=none](10,4){.1}
\pscircle[fillcolor=black,fillstyle=none](0,6){.1}
\pscircle[fillcolor=black,fillstyle=none](2,6){.1}
\pscircle[fillcolor=black,fillstyle=none](4,6){.1}
\pscircle[fillcolor=black,fillstyle=none](6,6){.1}
\pscircle[fillcolor=black,fillstyle=none](8,6){.1}
\pscircle[fillcolor=black,fillstyle=none](10,6){.1}
\pscircle[fillcolor=black,fillstyle=none](0,8){.1}
\pscircle[fillcolor=black,fillstyle=none](2,8){.1}
\pscircle[fillcolor=black,fillstyle=none](4,8){.1}
\pscircle[fillcolor=black,fillstyle=none](6,8){.1}
\pscircle[fillcolor=black,fillstyle=none](8,8){.1}
\pscircle[fillcolor=black,fillstyle=none](10,8){.1}

\pspolygon[linestyle=dashed](4,2)(4,-2)(0,-2)(0,2)
\pspolygon[linestyle=dashed](6,2)(6,-2)(2,-2)(2,2)
\pspolygon[linestyle=dashed](8,2)(8,-2)(4,-2)(4,2)
\pspolygon[linestyle=dashed](10,2)(10,-2)(6,-2)(6,2)

\pspolygon[linestyle=dashed](4,4)(4,0)(0,0)(0,4)
\pspolygon[linestyle=dashed](6,4)(6,0)(2,0)(2,4)
\pspolygon[linestyle=dashed](8,4)(8,0)(4,0)(4,4)
\pspolygon[linestyle=dashed](10,4)(10,0)(6,0)(6,4)

\pspolygon[linestyle=dashed](4,6)(4,2)(0,2)(0,6)
\pspolygon[linestyle=dashed](6,6)(6,2)(2,2)(2,6)
\pspolygon[linestyle=dashed](8,6)(8,2)(4,2)(4,6)
\pspolygon[linestyle=dashed](10,6)(10,2)(6,2)(6,6)

\pspolygon[linestyle=dashed](4,8)(4,4)(0,4)(0,8)
\pspolygon[linestyle=dashed](6,8)(6,4)(2,4)(2,8)
\pspolygon[linestyle=dashed](8,8)(8,4)(4,4)(4,8)
\pspolygon[linestyle=dashed](10,8)(10,4)(6,4)(6,8)

\pscircle[fillcolor=black,fillstyle=solid](3.5,1.5){.1}
\pscircle[fillcolor=black,fillstyle=solid](7.5,2.6){.1}
\pscircle[fillcolor=black,fillstyle=solid](7.2,-.7){.1}

\pscircle[fillcolor=black,fillstyle=solid](3.9, 2.4){.1}

\pscircle[fillcolor=black,fillstyle=solid](5.1,7.3){.1}
\pscircle[fillcolor=black,fillstyle=solid](6.3,5.5){.1}
\pscircle[fillcolor=black,fillstyle=solid](3,4.4){.1}
\pscircle[fillcolor=black,fillstyle=solid](2.7,4.9){.1}
\pscircle[fillcolor=black,fillstyle=solid](9,5.2){.1}
\pscircle[fillcolor=black,fillstyle=solid](6.5,5.8){.1}
\pscircle[fillcolor=black,fillstyle=solid](5.5,1.5){.1}

\psline[linewidth=1.5pt]{->}(2,0)(3.5,1.5)
\psline[linewidth=1.5pt]{->}(3.5,1.5)(3.9, 2.4)
\psline[linewidth=1.5pt]{->}(8,0)(7.5,2.6)

\pspolygon[linestyle=dashed](5,0) (-1,0)(2,3)
\pspolygon[linestyle=dashed](5,0) (11,0)(8,3)
\pspolygon[linestyle=dashed](4.8,1.5) (2.2,1.5)(3.5,2.8)

\rput[tl](2,-.2){$\mathbf{u}$}
\rput[tl](8,-.2){$\mathbf{v}$}
\rput[tl](3.7,1.3){$g_1(\mathbf{u})$}
\rput[tl](7.8,2.6){$g_2(\mathbf{v}) = g_1(\mathbf{v})$}
\rput[tl](4.1, 2.5){$g_2(\mathbf{u})$}

\psline[linecolor=blue](3.9, 2.4)(0.4,5.9)
\psline[linecolor=blue](3.9, 2.4)(7.4,5.9)
\rput[tl](3.6, 3.4) {$C(g_2(\mathbf{u}))$}

\psline[linecolor=blue](7.5,2.6)(4.2,5.9)
\psline[linecolor=blue](7.5,2.6)(10.8,5.9)
\rput[tl](7.2, 3.4) {$C(g_2(\mathbf{v}))$}

\end{pspicture}
\caption{For $d=2$, first two steps of the joint exploration process starting from $\mathbf{u}, \mathbf{v} \in \mathbb{Z}^2$ are represented. The black dots represent perturbed points. Both the points move at the first step whereas only the point $g_1(\mathbf{u})$ moves at the second step. The shaded region represents explored region $H_2$ which can not have a perturbed point in it's interior. Both the $\pi/2$ cones $C(g_2(\mathbf{u}))$ and 
$C(g_2(\mathbf{v}))$ (represented within blue lines) are unexplored, i.e., $(C(g_2(\mathbf{u})) \cup C(g_2(\mathbf{v}))) \cap H_2 = \emptyset$.}
\end{center}
\end{figure}

%In order to define the joint exploration process we need to introduce some further notations.
%For $\u \in \Z^d$, we define the neighbouring set or neighbourhood of $\u$ of interest for taking the next 
%step as 
%$$
%N(\u) := \{\v \in \Z^d : ||\v - \u||_1 = 1, \v(d) \geq \u(d)\}.
%$$ 
%We will extend this notion of neighbourhood naturally for a general non-lattice point 
%$\x \in \Gamma$ as $N(\x) := N(\widehat{\x})$.
%Observe that, the choice of $\delta$ ensures that for all $\x \in \Gamma$, we must have that 
%$$h(\x) \in \cup_{\v \in N(\x)} S_\delta(\v).$$

Before describing our movement algorithm, we observe that while starting from the lattice points $\u$ and $\v$, we don't have any information about the perturbed points in the vertex set $V$.   
Later we will define renewal steps for this process (see Definition \ref{def:Tau1Step})
and while restarting the process from a lattice point at the renewal step,
 it turns out that for a finite subset 
 $A \subset \Z^d $ of lattice points,
we have the information that for each $\w \in A$, the associated 
perturbed point $\w + U_\w$ is uniformly 
distributed over certain region.
In Section \ref{sec:ProofTauTail} and in Section \ref{sec:RW_Martingale}, 
we will mention about this in more detail.

The joint exploration process starting from the vertices $\u$ and $\v$ is denoted 
by $\{(g_n(\u), g_n(\v)): n \geq 0\}$ and  inductively defined  as follows.  
Set $(g_0(\u), g_0(\v)) := (\u, \v)$ and $H_0 := \emptyset$. 
Take $(g_1(\u), g_1(\v)) := (h(g_0(\u)), h(g_0(\v)))$.
 While starting from $(g_1(\u), g_1(\v))$ we need not have 
 $\lfloor g_1(\u)(d) \rfloor = \lfloor g_1(\v)(d) \rfloor$. Our movement algorithm ensures that if we have 
$\lfloor g_1(\u)(d) \rfloor \neq \lfloor g_1(\v)(d) \rfloor$ then vertex with lower $d$ co-ordinate 
takes step and the other one stays put. Otherwise both the vertices take step. When both the vertices take step then it is possible that we may have $h(g_1(\u)) = g_1(\v)$ or $h(g_1(\v)) = g_1(\u)$. This means that the two trajectories have coalesced already. In order to reflect this 
 we make the following modification: 
\begin{align*}
(g_{2}(\u), g_{2}(\v)) :=
\begin{cases}
(h(g_{1}(\u)), g_{1}(\v))  & \text{ if } \lfloor g_{1}(\u)(d) \rfloor < \lfloor g_{1}(\v)(d) \rfloor \\
(g_{1}(\u), h(g_{1 }(\v)))  & \text{ if } \lfloor g_{1}(\u)(d) \rfloor > \lfloor g_{1}(\v)(d) \rfloor \\
(h^2(g_1(\u)), h(g_1(\v))) & \text{ if }  \lfloor g_{1}(\u)(d) \rfloor = \lfloor g_{1}(\v)(d) \rfloor \text{ and }h(g_1(\u)) = g_1(\v) \\
(h(g_1(\u)), h^2(g_1(\v))) & \text{ if }  \lfloor g_{1}(\u)(d) \rfloor = \lfloor g_{1}(\v)(d) \rfloor \text{ and }h(g_1(\v)) = g_1(\u) \\
(h(g_1(\u)), h(g_1(\v))) & \text{ otherwise}.
\end{cases}
\end{align*}
More generally for any $n \geq 1$, we define
\begin{notation}
\begin{align*}
(g_{n + 1}(\u), g_{n + 1}(\v)) :=
\begin{cases}
(h(g_{n}(\u)), g_{n }(\v))  & \text{ if } \lfloor g_{n}(\u)(d) \rfloor < \lfloor g_{n}(\v)(d) \rfloor \\
(g_{n}(\u), h(g_{n }(\v)))  & \text{ if } \lfloor g_{n}(\u)(d) \rfloor > \lfloor g_{n}(\v)(d) \rfloor \\
(h^2(g_n(\u)), h(g_n(\v))) & \text{ if }  \lfloor g_{n}(\u)(d) \rfloor = \lfloor g_{n}(\v)(d) \rfloor \text{ and }h(g_n(\u)) = g_n(\v) \\
(h(g_n(\u)), h^2(g_n(\v))) & \text{ if }  \lfloor g_{n}(\u)(d) \rfloor = \lfloor g_{n}(\v)(d) \rfloor \text{ and }h(g_n(\v)) = g_n(\u) \\
(h(g_n(\u)), h(g_n(\v))) & \text{ otherwise}.
\end{cases}
\end{align*}
\end{notation}
In words, if both the points $g_n(\u)$ and $g_n(\v)$ are at the same lattice level, i.e., if $  \lfloor g_{n}(\u)(d) \rfloor = \lfloor g_{n}(\v)(d) \rfloor$ then both the vertices move. Otherwise only the lower level vertex moves and the remaining one stays put.
Further, when both the vertices take step, if one of them takes step to the other one, it takes one more step so that 
we have $g_{n+k}(\u) = g_{n+k}(\v)$ for all $k \geq 1$ reflecting the fact that both the paths have coalesced. In what follows, the $n+1$-th step of the joint exploration process is referred to as
$\langle (g_n(\u), g_n(\v)), (g_{n+1}(\u), g_{n+1}(\v))\rangle$. Similarly for a single 
DSF path process $\{g_n(\u) : n \geq 0\}$, the $n + 1$-th step refers to $\langle g_n(\u), g_{n+1}(\u)\rangle = \langle h^n(\u), h^{n+1}(\u)\rangle$.

We earlier commented that because of the information generated due to the 
 past movements, the process $\{g_n(\u) : n \geq 0 \}$ is not Markov. Because of the same argument, 
 the process $\{(g_n(\u), g_n(\v)) : n \geq 0\}$ is \textbf{not} Markov as well.
 For $l \in \R$, the upper half-plane (closed) is given by  
$\mathbb{H}^+(l) := \{\x \in \R^d : \x(d) \geq l \}$. 
  We set $r_n := g_n(\u)(d) \wedge g_n(\v)(d)$
 and we try to describe the explored region in the upper half-plane $\mathbb{H}^+(r_n)$ 
 about which we have the information that there is no perturbed lattice points in it's interior.
 For $\x \in \R^d$ and for $r \geq 0$ let 
 \begin{align}
 \label{def:UpperBall_LowerBall}
 B^+(\x, r) := \{\y \in \R^d : ||\y - \x||_1 \leq r, \y(d) \geq \x(d)\} \text{ and }
B^-(\x, r) := \{\y \in \R^d : ||\y - \x||_1 \leq r, \y(d) \leq \x(d)\}.
 \end{align}
denote the upper and lower half of the $||\quad||_1$ ball $B(\x, r)$  of radius $r$ centred at $\x$
respectively. For $r = 0$, both the sets $B^+(\x, 0)$ and $B^+(\x, 0)$ are taken to be empty. 

For $n \geq 1$, let $H_n$ denote the \textit{explored region} or the \textit{history region}
 in the upper half-plane $\mathbb{H}^+(r_n)$ due to $n$ steps of the process.
  Observe that during the $n$-th step for any moving vertex, i.e., for any $\w \in \{\u, \v\}$ with $g_{n-1}(\w) \neq g_{n}(\w)$
we must have that the interior of the region $B^+(g_{n-1}(\w),||g_{n-1}(\w) - h(g_{n-1}(\w))||_1)$ can not contain points from $V$. Among this, the part of the explored region in the upper half-plane 
$\mathbb{H}^+(r_n)$ actually affects the distribution of the $(n+1)$-th step.
For $n \geq 1$, the explored region
or history region in the upper half-plane $\mathbb{H}^+(r_n)$,
which can not contain points from $V$ in it's interior, is denoted as $H_n$ and defined as: 
\begin{notation}
\begin{align}
\label{eq:HistorySet_NoPoint_n}
H_n := 
\begin{cases}
\bigl ( B^+(g_{n-1}(\v), || h(g_{n-1}(\v)) - g_{n-1}(\v)||_1) \cup H_{n-1} \bigr ) \cap \mathbb{H}^+(r_n) & \text{ if }\lfloor g_{n-1}(\u)(d)\rfloor < \lfloor g_{n-1}(\v)(d)\rfloor \\
\bigl ( B^+(g_{n-1}(\v), || h(g_{n-1}(\u)) - g_{n-1}(\u)||_1) \cup H_{n-1} \bigr ) \cap \mathbb{H}^+(r_n) & \text{ if }\lfloor g_{n-1}(\v)(d)\rfloor < \lfloor g_{n-1}(\u)(d)\rfloor \\ 
\bigl ( B^+(g_{n-1}(\u), || h(g_{n-1}(\u)) - g_{n-1}(\u)||_1) \cup \\
 \quad  B^+(g_{n-1}(\v), || h(g_{n-1}(\v)) - g_{n-1}(\v)||_1) \cup H_{n-1} \bigr ) \cap \mathbb{H}^+(r_n) &\text{ otherwise}.
\end{cases}   
\end{align}
\end{notation}
In Figure \ref{fig:DSFPaths} for $d=2$, we have an illustration of the history region $H_2$ for the joint 
exploration process starting from $\u$ and $\v$.   
We observe that the formation of history set $H_n$ 
depends on the previous steps. Given ${\cal F}_n$, the explored region $H_n$
is of the form $\bigl ( \cup_{i=1}^k B^+(\x_i, l_i) \bigr ) \cap \mathbb{H}^+(r_n)$
 for some $k\geq 1$ with $\x_i(d) < r_n$
 and $l_i \geq 0$ for all $1 \leq i \leq k$. 

For $\x \in V$ let $\widehat{\x} \in \Z^d$ denote the lattice point such 
that $\x = \widehat{\x} + U_\x$.
Note that for each $\x \in V$, the point $\widehat{\x} \in \Z^d$ is a.s. uniquely defined. For the given choice of 
$\u$ and $\v$, we set $\Gamma_0 = \emptyset$ and for $n \geq 1$, we define 
\begin{notation}
\begin{align}
\label{eq:Gamma_Lattice_n}
\Gamma_n := \{\widehat{g_j}(\u), \widehat{g_j}(\v) : 0 \leq j \leq n \}  \subset \Z^d.
\end{align} 
\end{notation}
In other words the set $\Gamma_n$ denotes the set of lattice points 
whose perturbations were used already by $n$ steps $\{
(g_j(\u),  g_j(\v)): 0 \leq j \leq n\}$. Given $\{
(g_j(\u),  g_j(\v)): 0 \leq j \leq n\}$, a lattice point 
$\w \in \Z^d \setminus \Gamma_n$ as well as the associated perturbed point $\w + U_\w$
is said to be \textit{unexplored}.   
\begin{notation}
\label{eq:FFiltration_n}
Let
\begin{align*}
{\mathcal F}_n = {\mathcal F}_n(\u, \v) := \sigma \bigl( \{(g_j(\u), g_j(\v)) ,
:  0 \leq j \leq n\} , \Gamma_n \bigr)
\end{align*}
denote the natural filtration. 
\end{notation}
Clearly ${\mathcal F}_n$, has all information about $n$ steps of the joint process 
and the lattice points, whose perturbations were used in these steps, as well.
It is important to observe that for $0 \leq j \leq n$, including the history set $H_j$ in (\ref{eq:FFiltration_n}) is redundant, because
 the steps of the DSF paths $(g_j(\u), g_j(\v))$ for $0 \leq j \leq n$ 
 give us complete information about the 
 explored regions $H_j$  for $0 \leq j \leq n$ as well.

As observed earlier that the process $\{(g_n(\u), g_n(\v)) : n \geq 0\}$ is \textit{ not } Markov. 
Given ${\cal F}_n$ we have the information that the history region $H_n$ can not have a perturbed  point in its interior and  the explored region $H_n$ clearly depends on the past movements.
Further we also have the information that the perturbed points associated to the lattice points in $\Gamma_n$ have been used up. If we carry all these informations, then the 
 following proposition gives a Markovian structure. 
%It is good to mention here that, 
%unlike for the discrete DSF (constructed on a random subset of $\Z^d$ and studied in \cite{RSS16}),  or for the Poisson DSF (studied in \cite{BB07}, \cite{CSST19}), 
%the process $\{(g_n(\u_1),\cdots, g_n(\u_k), H_n(\u_1, \cdots, \u_k)) : n \geq 0\}$ is \textit{no} longer
% a Markov chain. We need to also carry the information about the set $\Gamma_n \subset \Z^d$. 

 \begin{proposition}
\label{prop:DSFMarkov}
The process $\bigl\{ \bigl( g_n(\u),g_n(\v), H_n, \Gamma_n \bigr) : n \geq 0 \bigr \}$ forms a Markov chain.  
\end{proposition}

\noindent \textbf{Proof: } We first consider a sequence 
of i.i.d. collections of $\{ U^{n}_\w : \w \in \Z^d \}_{n \geq 1}$ of i.i.d. random variables uniformly distributed over $[-1, +1]^d$ and independent of the i.i.d. 
family $\{ U_\w : \w \in \Z^d \}$ that we have started with. Given ${\cal F}_n$, 
we consider $\bigl( g_n(\u),g_n(\v), H_n , \Gamma_n \bigr) = (\x_1, \x_2,  h_n, A_n)$ where $A_n$
is a finite subset of $\Z^d$. 
We define a non-negative integer valued random variable $\eta$ as follows:
\begin{align*}
\eta := \inf\{ n \geq 1 : \w + U^n_\w \notin h_n \text{ for all }\w \in \Z^d \setminus A_n \}.
\end{align*}
Since $h_n$ is a bounded region and $A_n$ also finite, $\eta$ is a.s. finite.
Set $r = \x_1(d) \wedge \x_2(d)$ 
and based on $\{U^\eta_\w : \w \in \Z^d\}$ we can define 
perturbed lattice points in $\mathbb{H}^+(\lfloor r \rfloor)$ as 
$$
V^\eta := \{ \w + U^{\eta}_\w : \w \in \Z^d \setminus A_n, \w(d) \geq \lfloor r \rfloor \}.
$$
Note that $V^\eta$ does not use the lattice points in $A_n$.
Given $\bigl( g_n(\u), g_n(\v), H_n , \Gamma_n \bigr) = (\x_1, \x_2, h_n, A_n)$,
the process will evolve according to the point process $V^\eta$.
In other words the conditional distribution of $\bigl(  g_{n+1}(\u),g_{n+1}(\v), H_{n+1} , \Gamma_{n+1} \bigr )$ given $\bigl( g_{n}(\u),g_{n}(\v), H_{n} , \Gamma_{n} \bigr) = (\x_1, \x_2, h_n, A_n)$, can be expressed as
\begin{align*}
\bigl(&  g_{n+1}(\u),g_{n+1}(\v), H_{n+1} , \Gamma_{n+1} \bigr )
\mid \bigl \{ \bigl( g_{n}(\u),g_{n}(\v), H_{n} , \Gamma_{n} \bigr) = (\x_1, \x_2, h_n, A_n), \cdots
\bigr \} \\
& \qquad  \qquad \qquad  = f \bigl ( (\x_1, \x_2, h_n, A_n), \{U^n_\w  : \w \in \Z^d \}_{n \geq 1}
\bigr )
\end{align*}
for some measurable $f$. Hence the Markov property follows by random mapping representation (see \cite{LPW09}). 
\qed

\begin{remark}
\label{rem:ConstructionPerturbedPoints}
We presented a construction starting from $\u, \v \in \Z^d$ with $\u(d) = \v(d)$. 
One can also do a similar construction starting from $\x, \y \in V$. In that case, 
our starting $\Gamma_0$ should be taken as $\{\widehat{\x}, \widehat{\y}\}$ 
and distribution of the joint exploration process will depend upon $\Gamma_0$.  
\end{remark}

Next we will show that there are random times such that the 
joint exploration process observed  at these random times exhibits renewal properties
in some sense. 
 In order to describe such a sequence of random times 
few more notations are required. 
For $\x \in V$, we define $\x^\uparrow, \x^\downarrow \in \Z^d$ as  
\begin{notation}
\label{def:UpArrowVertex}
\begin{align*}
\x^\uparrow & := \text{argmin} \{||\x - \w ||_1 : \w \in \Z^d, \w(d) = \lfloor \x(d)\rfloor + 1\},\\
\x^\downarrow & := \text{argmin} \{||\x - \w ||_1 : \w \in \Z^d, \w(d) = \lfloor \x(d)\rfloor\}.
\end{align*}
\end{notation}
In other words, $\x^\uparrow$ denotes the closest lattice point with $\x^\uparrow(d) = \lfloor \x(d)\rfloor + 1$. Similarly $\x^\downarrow$ denotes the closest lattice point with $\x^\downarrow(d)
 = \lfloor \x(d)\rfloor$. Fix $\delta$ as a small positive constant and consider the regions:  
$$
\Delta_n^\u := B^+(g^\uparrow_n(\u), \delta) \text{ and }
\Delta_n^\v := B^+(g^\uparrow_n(\v), \delta).
$$
Now we define certain kind of favourable steps referred as `up' steps
(see Figure \ref{fig:Up_Step}). 
We will also explain why these steps are favourable.
\begin{definition}
\label{def:nUpStep}
For any  $n \in \N$ given ${\cal F}_n$, we say that the $n+1$-th step 
is an \textbf{`up'} step if the following conditions are 
satisfied:
\begin{itemize}
\item[(i)] $\lfloor g_n(\u)(d) \rfloor = \lfloor g_n(\v)(d) \rfloor$ and 
\item[(ii)] $ g_{n+1}(\u) \in \Delta_n^\u$ and $g_{n+1}(\v) \in \Delta_n^\v$.
\end{itemize} 
\end{definition}
It is important to observe that the `up' steps are good for us. Because of upward movements,
 they do {\it not} generate large history regions. 
 More precisely, the maximum height of the \textit{new} history regions produced during an up step is 
  $(1/2 + \delta)(d-1)$. We give a brief justification for the marginal process $\{g_n(\u) : n \geq 0\}$ for $d=2$. 
Note that $|g_n(\u)(1) - g^\uparrow_n(\u)(1)| \leq 1/2$. If $(n+1)$-th step is an up step, then we must have $g_{n+1}(\u)
\in B^+(g^\uparrow_n, \delta)$ and hence the upper bound for the height of the new history triangle follows.
 On the other hand if we have consecutive two up steps, then the height of the \textit{newly} created history region at the end of consecutive two ups steps is of height at most $2\delta (d-1)$. This follows since the $||\quad||_1$ distance w.r.t. first $(d-1)$ coordinates between the last position and the penultimate position is bounded by $2\delta(d-1)$.
Further, given $g_n(\u), g_n(\v)$ with $\lfloor g_n(\u)(d)\rfloor = \lfloor g_n(\v)(d)\rfloor$, after two consecutive up steps we must have $g_{n+2}(\u)(d) \wedge g_{n+2}(\u)(d) 
  \geq g_{n}(\u)(d) \wedge g_{n}(\u)(d) + 1$. 
  Hence, if $L(H_n) = l > 1$
   then after two consecutive up steps the height of the history region 
   becomes smaller than $2\delta(d-1)\vee (l-1)$.  This observation motivates us to use up steps in defining our renewal steps.      
%More formally, given ${\cal F}_n$, if the $n$-th step $\langle W^{\text{move}}_n , h(W^{\text{move}}_n)\rangle$ is an `up' step then it does not generate 
%any additional history triangle over the region $\Gamma(r_n) \setminus S_\delta()$ 
\begin{figure}[ht!]
\label{fig:Up_Step}
\begin{center}

\begin{pspicture}(0,-2)(10,8.5)
%\psgrid[subgriddiv=0,griddots=5,gridlabels=0]

\pspolygon[linestyle=dashed, fillstyle = solid , fillcolor = lightgray](3.9,2.4) (3.5,2.8)(3.1,2.4)
\pspolygon[linestyle=dashed, fillstyle = solid , fillcolor = lightgray](7.4,2.4) (8.6,2.4)(8,3)
\pspolygon[fillcolor = gray, fillstyle = solid](8,4.5)(8.5,4)(7.5,4)
\pspolygon[fillcolor = gray, fillstyle = solid](4,4.5)(4.5,4)(3.5,4)

\pscircle[fillcolor=black,fillstyle=none](0,0){.1}
\pscircle[fillcolor=black,fillstyle=none](2,0){.1}
\pscircle[fillcolor=black,fillstyle=none](4,0){.1}
\pscircle[fillcolor=black,fillstyle=none](6,0){.1}
\pscircle[fillcolor=black,fillstyle=none](8,0){.1}
\pscircle[fillcolor=black,fillstyle=none](10,0){.1}
\pscircle[fillcolor=black,fillstyle=none](0,2){.1}
\pscircle[fillcolor=black,fillstyle=none](2,2){.1}
\pscircle[fillcolor=black,fillstyle=none](4,2){.1}
\pscircle[fillcolor=black,fillstyle=none](6,2){.1}
\pscircle[fillcolor=black,fillstyle=none](8,2){.1}
\pscircle[fillcolor=black,fillstyle=none](10,2){.1}
\pscircle[fillcolor=black,fillstyle=none](0,4){.1}
\pscircle[fillcolor=black,fillstyle=none](2,4){.1}
\pscircle[fillcolor=red,fillstyle=solid](4,4){.1}
\pscircle[fillcolor=black,fillstyle=none](6,4){.1}
\pscircle[fillcolor=red,fillstyle=solid](8,4){.1}
\pscircle[fillcolor=black,fillstyle=none](10,4){.1}
\pscircle[fillcolor=black,fillstyle=none](0,6){.1}
\pscircle[fillcolor=black,fillstyle=none](2,6){.1}
\pscircle[fillcolor=black,fillstyle=none](4,6){.1}
\pscircle[fillcolor=black,fillstyle=none](6,6){.1}
\pscircle[fillcolor=black,fillstyle=none](8,6){.1}
\pscircle[fillcolor=black,fillstyle=none](10,6){.1}
\pscircle[fillcolor=black,fillstyle=none](0,8){.1}
\pscircle[fillcolor=black,fillstyle=none](2,8){.1}
\pscircle[fillcolor=black,fillstyle=none](4,8){.1}
\pscircle[fillcolor=black,fillstyle=none](6,8){.1}
\pscircle[fillcolor=black,fillstyle=none](8,8){.1}
\pscircle[fillcolor=black,fillstyle=none](10,8){.1}

\pspolygon[linestyle=dashed](4,2)(4,-2)(0,-2)(0,2)
\pspolygon[linestyle=dashed](6,2)(6,-2)(2,-2)(2,2)
\pspolygon[linestyle=dashed](8,2)(8,-2)(4,-2)(4,2)
\pspolygon[linestyle=dashed](10,2)(10,-2)(6,-2)(6,2)

\pspolygon[linestyle=dashed](4,4)(4,0)(0,0)(0,4)
\pspolygon[linestyle=dashed](6,4)(6,0)(2,0)(2,4)
\pspolygon[linestyle=dashed](8,4)(8,0)(4,0)(4,4)
\pspolygon[linestyle=dashed](10,4)(10,0)(6,0)(6,4)

\pspolygon[linestyle=dashed](4,6)(4,2)(0,2)(0,6)
\pspolygon[linestyle=dashed](6,6)(6,2)(2,2)(2,6)
\pspolygon[linestyle=dashed](8,6)(8,2)(4,2)(4,6)
\pspolygon[linestyle=dashed](10,6)(10,2)(6,2)(6,6)

\pspolygon[linestyle=dashed](4,8)(4,4)(0,4)(0,8)
\pspolygon[linestyle=dashed](6,8)(6,4)(2,4)(2,8)
\pspolygon[linestyle=dashed](8,8)(8,4)(4,4)(4,8)
\pspolygon[linestyle=dashed](10,8)(10,4)(6,4)(6,8)

\pscircle[fillcolor=black,fillstyle=solid](3.5,1.5){.1}
\pscircle[fillcolor=black,fillstyle=solid](7.5,2.6){.1}

\pscircle[fillcolor=black,fillstyle=solid](3.9, 2.4){.1}

\pscircle[fillcolor=black,fillstyle=solid](5.5,1.5){.1}

\psline[linewidth=1.5pt]{->}(2,0)(3.5,1.5)
\psline[linewidth=1.5pt]{->}(3.5,1.5)(3.9, 2.4)
\psline[linewidth=1.5pt]{->}(3.9, 2.4)(4.3,4.2)
\psline[linewidth=1.5pt]{->}(8,0)(7.5,2.6)
\psline[linewidth=1.5pt]{->}(7.5,2.6)(8.2,4.3)

\pspolygon[linestyle=dashed](5,0) (-1,0)(2,3)
\pspolygon[linestyle=dashed](5,0) (11,0)(8,3)
\pspolygon[linestyle=dashed](4.8,1.5) (2.2,1.5)(3.5,2.8)

\pscircle[fillcolor=black,fillstyle=solid](4.3,4.2){.1}
\pscircle[fillcolor=black,fillstyle=solid](8.2,4.3){.1}

\pscircle[fillcolor=black,fillstyle=solid](3.5,1.5){.1}
\pscircle[fillcolor=black,fillstyle=solid](7.5,2.6){.1}
\pscircle[fillcolor=black,fillstyle=solid](1.8,2.2){.1}
\pscircle[fillcolor=black,fillstyle=solid](1.8,4.18){.1}
\pscircle[fillcolor=black,fillstyle=solid](2.2,5.7){.1}
\pscircle[fillcolor=black,fillstyle=solid](3.85,6.18){.1}
\pscircle[fillcolor=black,fillstyle=solid](4.5,6.13){.1}
\pscircle[fillcolor=black,fillstyle=solid](7.8,5.08){.1}
\pscircle[fillcolor=black,fillstyle=solid](9.15,5.22){.1}
\pscircle[fillcolor=black,fillstyle=solid](8.9,5.1){.1}
\pscircle[fillcolor=black,fillstyle=solid](9.12,7.13){.1}
\pscircle[fillcolor=black,fillstyle=solid](7.7,-.8){.1}

\rput[tl](2,-.2){$\mathbf{u}$}
\rput[tl](8,-.2){$\mathbf{v}$}
%\rput[tl](3.7,1.3){$g_1(\mathbf{u})$}
%\rput[tl](5.2,2.6){$g_2(\mathbf{v}) = g_1(\mathbf{v})$}
%\rput[tl](4.1, 2.5){$g_2(\mathbf{u})$}

%\psline[linecolor=blue](3.9, 2.4)(0.4,5.9)
%\psline[linecolor=blue](3.9, 2.4)(7.4,5.9)
\rput[tl](3.6, 3.8) {$g^\uparrow_2(\mathbf{u})$}

%\psline[linecolor=blue](7.5,2.6)(4.2,5.9)
%\psline[linecolor=blue](7.5,2.6)(10.8,5.9)
\rput[tl](7.4, 3.8) {$g^\uparrow_2(\mathbf{v})$}

\end{pspicture}
\end{center}
\caption{For $d=2$, an `up step' is represented for the joint exploration process. The lattice points $g^\uparrow_2(\mathbf{u})$ and $g^\uparrow_2(\mathbf{v})$ are represented as red circles and the 
gray shaded regions $\Delta^\u_n$ and $\Delta^\v_n$ must contain points from $V$.}

\end{figure}
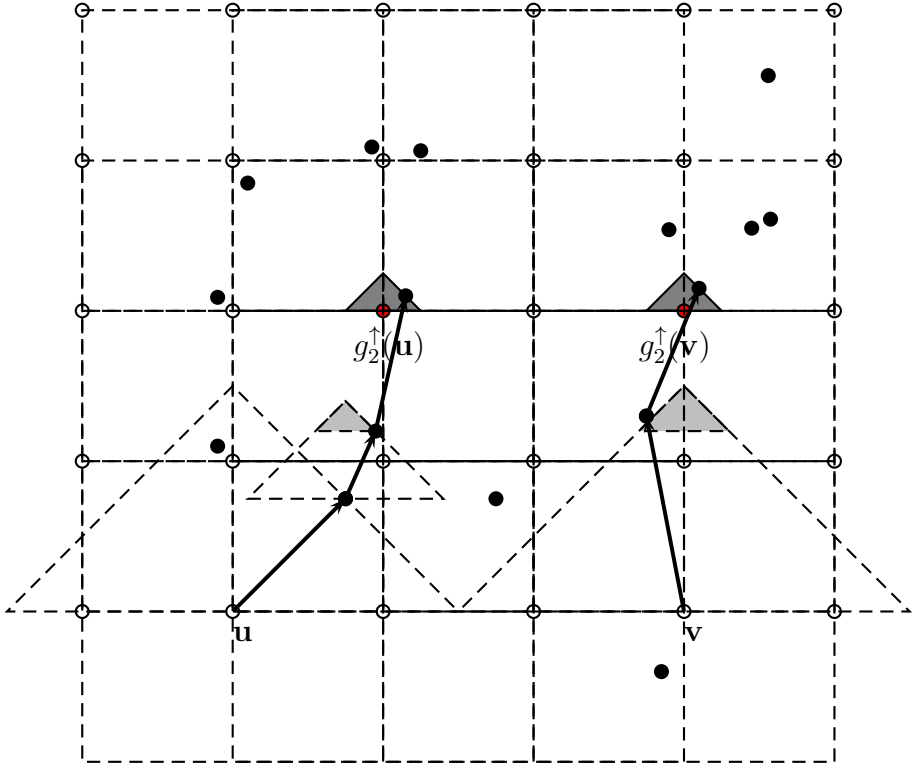

Next we put further restrictions  on `up' steps and define `special up' steps.
\begin{definition}
\label{def:SpecialUpStep}
For $\w \in \Z^d$, we define 
\begin{align}
\label{def:InftyNbd}
N(\w) := \{ \y \in \Z^d : \y(d) \geq \w(d), ||\y - \w||_\infty \leq 1\},
\end{align}
as the $|| \quad ||_\infty$ neighbourhood of $\w$ in the upper half-plane $\mathbb{H}^+(\w (d))$.

Next we define the event $A(\w)$ as 
\begin{align*}
A(\w) := \{ \y + U_\y \in B^+(\y, \delta) \text{ for all } \y \in N(\w)\} ,
\end{align*}
i.e., for each $\y \in N(\w)$, the associated perturbed point belongs to $B^+(\y, \delta)$.
Given ${\cal F}_n$, the $n+1$-th step 
is called a \textit{`special up'} step if is an up step and the event $A(g_n(\u)) \cap A(g_n(\v))$
also occurs (see Figure \ref{fig:A_n}).
\end{definition}
It is important to observe that the set $N(\w)$ is defined w.r.t. $||\quad||_\infty$ norm
and on the event $A(\w)$, the region $ B^+(\w, \delta)$ contains a single point from $V$, viz., 
$\w + U_\w$. Hence if the $n+1$-th step is a `special up step' then by definition we have
$$ g_{n+1}(\u) = \bigl ( g_n(\u)^\uparrow + U_{g_n(\u)^\uparrow} \bigr ) \text{ as well as }
g_{n+1}(\v) = \bigl ( g_n(\v)^\uparrow + U_{g_n(\v)^\uparrow} \bigr ).
$$  
Instead of the joint process if we consider the marginal process $\{ g_n(\u) =h^n(\u) : n \geq 0\}$,
then up steps and special up steps are defined similarly.

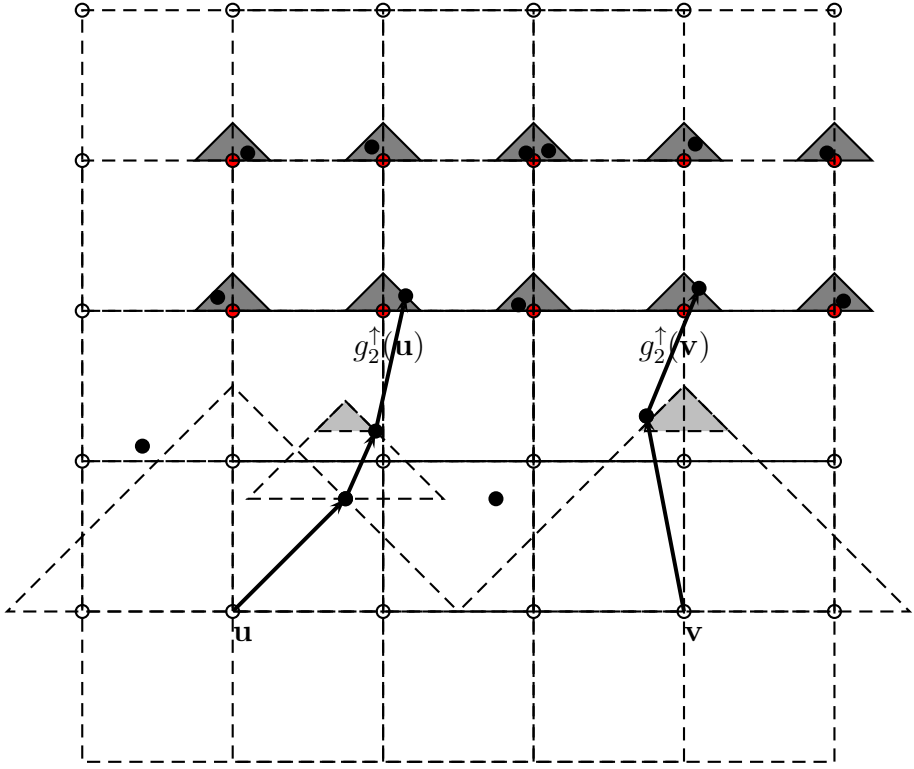
\begin{figure}[ht!]
\label{fig:A_n}
\begin{center}

\begin{pspicture}(0,-2)(10,8.5)
%\psgrid[subgriddiv=0,griddots=5,gridlabels=0]

\pspolygon[linestyle=dashed, fillstyle = solid , fillcolor = lightgray](3.9,2.4) (3.5,2.8)(3.1,2.4)
\pspolygon[linestyle=dashed, fillstyle = solid , fillcolor = lightgray](7.4,2.4) (8.6,2.4)(8,3)

\pspolygon[fillcolor = gray, fillstyle = solid](4,4.5)(4.5,4)(3.5,4)
\pspolygon[fillcolor = gray, fillstyle = solid](2,4.5)(2.5,4)(1.5,4)
\pspolygon[fillcolor = gray, fillstyle = solid](6,4.5)(6.5,4)(5.5,4)
\pspolygon[fillcolor = gray, fillstyle = solid](4,6.5)(4.5,6)(3.5,6)
\pspolygon[fillcolor = gray, fillstyle = solid](2,6.5)(2.5,6)(1.5,6)
\pspolygon[fillcolor = gray, fillstyle = solid](6,6.5)(6.5,6)(5.5,6)
\pspolygon[fillcolor = gray, fillstyle = solid](8,4.5)(8.5,4)(7.5,4)
\pspolygon[fillcolor = gray, fillstyle = solid](10,4.5)(10.5,4)(9.5,4)
\pspolygon[fillcolor = gray, fillstyle = solid](8,6.5)(8.5,6)(7.5,6)
\pspolygon[fillcolor = gray, fillstyle = solid](10,6.5)(10.5,6)(9.5,6)

\pscircle[fillcolor=black,fillstyle=none](0,0){.1}
\pscircle[fillcolor=black,fillstyle=none](2,0){.1}
\pscircle[fillcolor=black,fillstyle=none](4,0){.1}
\pscircle[fillcolor=black,fillstyle=none](6,0){.1}
\pscircle[fillcolor=black,fillstyle=none](8,0){.1}
\pscircle[fillcolor=black,fillstyle=none](10,0){.1}
\pscircle[fillcolor=black,fillstyle=none](0,2){.1}
\pscircle[fillcolor=black,fillstyle=none](2,2){.1}
\pscircle[fillcolor=black,fillstyle=none](4,2){.1}
\pscircle[fillcolor=black,fillstyle=none](6,2){.1}
\pscircle[fillcolor=black,fillstyle=none](8,2){.1}
\pscircle[fillcolor=black,fillstyle=none](10,2){.1}
\pscircle[fillcolor=black,fillstyle=none](0,4){.1}
\pscircle[fillcolor=red,fillstyle=solid](2,4){.1}
\pscircle[fillcolor=red,fillstyle=solid](4,4){.1}
\pscircle[fillcolor=red,fillstyle=solid](6,4){.1}
\pscircle[fillcolor=red,fillstyle=solid](8,4){.1}
\pscircle[fillcolor=red,fillstyle=solid](10,4){.1}
\pscircle[fillcolor=black,fillstyle=none](0,6){.1}
\pscircle[fillcolor=red,fillstyle=solid](2,6){.1}
\pscircle[fillcolor=red,fillstyle=solid](4,6){.1}
\pscircle[fillcolor=red,fillstyle=solid](6,6){.1}
\pscircle[fillcolor=red,fillstyle=solid](8,6){.1}
\pscircle[fillcolor=red,fillstyle=solid](10,6){.1}
\pscircle[fillcolor=black,fillstyle=none](0,8){.1}
\pscircle[fillcolor=black,fillstyle=none](2,8){.1}
\pscircle[fillcolor=black,fillstyle=none](4,8){.1}
\pscircle[fillcolor=black,fillstyle=none](6,8){.1}
\pscircle[fillcolor=black,fillstyle=none](8,8){.1}
\pscircle[fillcolor=black,fillstyle=none](10,8){.1}

\pspolygon[linestyle=dashed](4,2)(4,-2)(0,-2)(0,2)
\pspolygon[linestyle=dashed](6,2)(6,-2)(2,-2)(2,2)
\pspolygon[linestyle=dashed](8,2)(8,-2)(4,-2)(4,2)
\pspolygon[linestyle=dashed](10,2)(10,-2)(6,-2)(6,2)

\pspolygon[linestyle=dashed](4,4)(4,0)(0,0)(0,4)
\pspolygon[linestyle=dashed](6,4)(6,0)(2,0)(2,4)
\pspolygon[linestyle=dashed](8,4)(8,0)(4,0)(4,4)
\pspolygon[linestyle=dashed](10,4)(10,0)(6,0)(6,4)

\pspolygon[linestyle=dashed](4,6)(4,2)(0,2)(0,6)
\pspolygon[linestyle=dashed](6,6)(6,2)(2,2)(2,6)
\pspolygon[linestyle=dashed](8,6)(8,2)(4,2)(4,6)
\pspolygon[linestyle=dashed](10,6)(10,2)(6,2)(6,6)

\pspolygon[linestyle=dashed](4,8)(4,4)(0,4)(0,8)
\pspolygon[linestyle=dashed](6,8)(6,4)(2,4)(2,8)
\pspolygon[linestyle=dashed](8,8)(8,4)(4,4)(4,8)
\pspolygon[linestyle=dashed](10,8)(10,4)(6,4)(6,8)

\pscircle[fillcolor=black,fillstyle=solid](3.5,1.5){.1}
\pscircle[fillcolor=black,fillstyle=solid](7.5,2.6){.1}

\pscircle[fillcolor=black,fillstyle=solid](3.9, 2.4){.1}

\pscircle[fillcolor=black,fillstyle=solid](5.5,1.5){.1}

\psline[linewidth=1.5pt]{->}(2,0)(3.5,1.5)
\psline[linewidth=1.5pt]{->}(3.5,1.5)(3.9, 2.4)
\psline[linewidth=1.5pt]{->}(3.9, 2.4)(4.3,4.2)
\psline[linewidth=1.5pt]{->}(8,0)(7.5,2.6)
\psline[linewidth=1.5pt]{->}(7.5,2.6)(8.2,4.3)

\pspolygon[linestyle=dashed](5,0) (-1,0)(2,3)
\pspolygon[linestyle=dashed](5,0) (11,0)(8,3)
\pspolygon[linestyle=dashed](4.8,1.5) (2.2,1.5)(3.5,2.8)

\pscircle[fillcolor=black,fillstyle=solid](4.3,4.2){.1}
\pscircle[fillcolor=black,fillstyle=solid](8.2,4.3){.1}

\pscircle[fillcolor=black,fillstyle=solid](3.5,1.5){.1}
\pscircle[fillcolor=black,fillstyle=solid](7.5,2.6){.1}
\pscircle[fillcolor=black,fillstyle=solid](0.8,2.2){.1}
\pscircle[fillcolor=black,fillstyle=solid](1.8,4.18){.1}
\pscircle[fillcolor=black,fillstyle=solid](2.2,6.1){.1}
\pscircle[fillcolor=black,fillstyle=solid](3.85,6.18){.1}
\pscircle[fillcolor=black,fillstyle=solid](6.2,6.13){.1}
\pscircle[fillcolor=black,fillstyle=solid](5.9,6.1){.1}
\pscircle[fillcolor=black,fillstyle=solid](5.8,4.08){.1}
\pscircle[fillcolor=black,fillstyle=solid](8.15,6.22){.1}
\pscircle[fillcolor=black,fillstyle=solid](9.9,6.1){.1}
\pscircle[fillcolor=black,fillstyle=solid](10.12,4.13){.1}

\rput[tl](2,-.2){$\mathbf{u}$}
\rput[tl](8,-.2){$\mathbf{v}$}
%\rput[tl](3.7,1.3){$g_1(\mathbf{u})$}
%\rput[tl](5.2,2.6){$g_2(\mathbf{v}) = g_1(\mathbf{v})$}
%\rput[tl](4.1, 2.5){$g_2(\mathbf{u})$}

%\psline[linecolor=blue](3.9, 2.4)(0.4,5.9)
%\psline[linecolor=blue](3.9, 2.4)(7.4,5.9)
\rput[tl](3.6, 3.8) {$g^\uparrow_2(\mathbf{u})$}

%\psline[linecolor=blue](7.5,2.6)(4.2,5.9)
%\psline[linecolor=blue](7.5,2.6)(10.8,5.9)
\rput[tl](7.4, 3.8) {$g^\uparrow_2(\mathbf{v})$}

\end{pspicture}
\end{center}
\caption{For $d=2$, this figure represents special up step for the joint exploration process. The points 
in the set $N(g^\uparrow_2(\u)) \cup N(g^\uparrow_2(\v))$ are marked as red circles.
 Together with the up step,
each of the gray regions must have the associated perturbed point in it. It follows that both the regions 
$\Delta^\u_2$ and $\Delta^\v_2$ have exactly one point from $V$.}

\end{figure}

Now we are ready to define the following sequence of random steps, which we will call as renewal steps.
Let $m_d \in \N $ is as in Lemma \ref{lem:HistoryHeightBound}.
Set $\tau_0 = 0$ and for $j \geq 1$ define
\begin{notation}
\begin{align}
\label{def:Tau1Step}
\tau_j = \tau_j(\u , \v) := \inf\{& n > \tau_{j - 1} + m_d : \text{ the last } m_d \text{ steps are `up' steps and} \nonumber \\
& \qquad \text{ the last one is a `special up' step}  \},
\end{align}
\end{notation} 
In order to make sure that the definition makes sense we need to  show that for all $j \geq 1$, 
the r.v. $\tau_j$ is a.s. finite. We will prove a much stronger result. In Section \ref{sec:ProofTauTail}, we show that for all $j \geq 0$, the difference r.v. $(\tau_{j+1} - \tau_j)$  
decays exponentially. 

We observe that at a renewal step we must have
$\lfloor g_{\tau_j}(\u)(d) \rfloor  = \lfloor g_{\tau_{j}}(\v)(d)\rfloor$. 
Further the restriction $\tau_j > \tau_{j - 1} + m_d$ ensures 
that the steps between $\tau_{j-1}$  and $\tau_{j}$ are 
completely different from the steps between $\tau_{j-2}$  and $\tau_{j-1}$.
Before we proceed we comment that for the marginal process $\{g_n(\u) : n \geq 0\}$,
the sequence of renewal steps $\{\tau_j(\u) : j \geq 0\}$ is defined similarly.

In order to motivate the choice of $m_d$ in the definition of renewal step $\tau$, 
we prove one lemma which tells us that 
the `height' of the explored regions for our exploration process 
remains bounded throughout by $m_d  - 4$.
We first define our height function.

For any bounded subset $B$ of $\R^d$, we define the height of $B$ as 
$$
L(B) := \sup\{\x_1(d) - \x_2(d) : \x_1, \x_2 \in B \}.
$$ 
The explored region $H_n$ represents the dependency with previous steps and
it is good to have explored regions with lesser heights.
Our next lemma shows that the function $L(H_n)$ is bounded.

\begin{lemma}
\label{lem:HistoryHeightBound}
There exists $m_d \in \N$ depending on the dimension $d$ such that for all $n \geq 0$ we have  
\begin{equation}
\label{lem:HistoryHeightBound}
L(H_n) \leq m_d - 4 \text{ a.s.}
\end{equation}
\end{lemma} 
\noindent \textbf{Proof :}
We first prove it for the marginal process $\{g_n(\u) : n \geq 0\}$.
We recall that for each $n \geq 0$, the history region $H_n$ represents the 
explored region in the upper half-plane $\mathbb{H}^+(g_{n}(\u)(d))$.
Since we always have $g_{n+1}(\u)(d) > g_{n}(\u)(d)$, the `height' of the
regions explored due to earlier movements, must decrease. On the other hand, the height of the newly created history region  
$B^+(g_n(\u), ||g_n(\u) - g_{n+1}(\u)||_1)$ must be bounded by the increment 
$||g_n(\u) - g_{n+1}(\u)||_1$. Hence, the function $L(H_n)$ 
satisfies the following recursion relation 
\begin{equation}
\label{eq:L_nEvolution_Relation}
L(H_{n+1}) \leq \max \{ L(H_{n}), ||g_n(\u) - g_{n+1}(\u)||_1\} \text{ with probability }1.
\end{equation}
Now, for the marginal process given ${\cal F}_n$,
 we must have that the lattice point $g^\uparrow_n(\u) + e_d$
has not been used up, i.e.,  $(g^\uparrow_n(\u) + e_d) \notin \Gamma_n$.
Our model ensures that the corresponding perturbed point must 
belong to the box $(g^\uparrow_n(\u) + e_d) + [-1, +1]^d $. Clearly, this gives an upper bound for the increment $||g_n(\u) - g_{n+1}(\u)||_1$.
Together with (\ref{eq:L_nEvolution_Relation}), this completes the proof for the marginal process.

For the joint exploration process, $L(H_n)$ function satisfies a similar recursion relation:
$$
L(H_{n+1}) \leq \max \{ L(H_{n}), ||g_n(\u) - h(g_{n}(\u))||_1, ||g_n(\v) - h(g_{n}(\v))||_1\} \text{ with probability }1.
$$ 
Though sometimes for a step of the joint exploration process a point moves twice, during the
second step it just follows the trajectory of the other point and hence does not 
generate any new explored region. 
For the joint process, at most one of the lattice points $g^\uparrow_n(\u) + e_d$ and $g^\uparrow_n(\v) + e_d$  could be used up, but there must be unexplored neighbours of these points and hence, similar argument proves the lemma for the joint exploration process as well. 
\qed

It is important to note that for any $j \geq 1$, 
the r.v. $\tau_j$ is \textit{not} a stopping time
w.r.t. the filtration $\{{\cal F}_n : n \geq 0 \}$, defined as in (\ref{eq:FFiltration_n}).
Because of this, we define a new filtration $\{ {\cal G}_n : n \geq 0\}$.

\begin{notation}
\label{not:G_n}
For any event $A$, let $\mathbf{1}_A$ denote the indicator r.v. of the event $A$.
For $n \geq 1$ we define the event 
$E_n = E_n(\u, \v) := A(g^\uparrow_{n-1}(\u) ) \cap A(g^\uparrow_{n-1}(\v))$.
 Now, set ${\cal G}_0 = {\cal F}_0$ and for $n \geq 1$ define 
\begin{align}
%\label{def:GnFiltration}
{\cal G}_n := \sigma \bigl ( \{(g_j(\u), g_j(\v)) : 0 \leq j \leq n-1\}, 
\Gamma_n, \mathbf{1}_{E_1}, \mathbf{1}_{E_2}, \cdots, \mathbf{1}_{E_n} \bigr).
\end{align}
\end{notation}
It is important to observe that the  $\sigma$-field ${\cal G}_n$ has information only about the lattice 
points $\widehat{g_n}(\u)$ and $\widehat{g_n}(\v)$, but \textit{not} about the original points $g_n(\u)$ and $g_n(\v)$. In addition, it also has information about occurrence or non-occurrence of the 
event $A(g^\uparrow_{n-1}(\u) ) \cap A(g^\uparrow_{n-1}(\v))$.
We did not include the points $g_n(\u)$ and $g_n(\v)$ in this $\sigma$-field purposefully. The reason 
will be made clear in the proof of Proposition \ref{prop:SinglePt_Rwalk}. 

On the event $A(g^\uparrow_{n-1}(\u) ) \cap A(g^\uparrow_{n-1}(\v))$ if we have 
$\widehat{g_n}(\u) = g^\uparrow_n(\u)$ and $\widehat{g_n}(\v) = g^\uparrow_n(\v)$, then the $n$-th 
step $\langle (g_{n-1}(\u), g_{n-1}(\v)), (g_{n}(\u), g_{n}(\v))\rangle$ is 
an up step and consequently a special up step as well. Hence it 
follows that for any $j \geq 1$, the r.v. $\tau_j$ is a stopping time 
w.r.t. the filtration $\{{\cal G}_n : n \geq 0 \}$ and hence $(\widehat{g_{\tau_j}}(\u), \widehat{g_{\tau_j}}(\v))$ is ${\cal G}_{\tau_j}$ adapted
as well.  The next proposition tells us that the r.v. $\tau_j$ is a.s. finite for all $j \geq 1$.
\begin{proposition}
\label{prop:BetaExpTail}
Fix any $j \geq 0$. There exist $C_0, C_1$ positive constants which depend only on 
the dimension $d$ such that for all $n \geq 1$ we have
\begin{equation}
\label{eq:BetaExpTail}
\P(\tau_{j+1} - \tau_j \geq n \mid {\cal G}_{\tau_j}) \leq C_0 \exp{(- C_1 n)}.
\end{equation}
\end{proposition}
The proof of Proposition \ref{prop:BetaExpTail} is non-trivial mainly because of the dependency
of the point process considered here. In the next section, we will 
first explain the difficulty involved in proving this proposition and then prove this proposition. 
Before ending this section we just mention here briefly 
the significance of our renewal steps for the marginal process and the joint process.

Note that for any $j \geq 1$, the lattice points $\widehat{g_{\tau_j}}(\u)$ and $\widehat{g_{\tau_j}}(\v)$ are ${\cal G}_{\tau_j}$ measurable and for the marginal process $\{g_n(\u) : n \geq 0\}$, the
 renewal steps allow to restart the process from the lattice point $\widehat{g_{\tau_j}}(\u)$
 with the initial condition that for all $\w \in N(\widehat{g_{\tau_j}}(\u))$
 (where the set $N(\widehat{g_{\tau_j}}(\u))$ is defined as in (\ref{def:InftyNbd})), the associated 
 perturbed point $\w + U_\w$ belongs to $B^+(\w, \delta)$. In Proposition \ref{prop:SinglePt_Rwalk} we will show  that this allows us to divide the marginal process into blocks of i.i.d. increments.  
 
 For the joint exploration process, we can again restart the process at renewal steps   
 from the lattice points $\widehat{g_{\tau_j}}(\u)$ and $\widehat{g_{\tau_j}}(\v)$. 
In Proposition \ref{prop:TwoPt_Markovwalk} we show that the  difference between two  
restarted DSF paths observed at renewal steps is Markov as well. For $d=2$, we  
further show that far from the origin, this process behaves like a symmetric random walk.
 This gives us that the coalescing time between two 
DSF paths is an a.s. finite r.v. with suitable tail decay. This would be crucial for proving 
convergence to the Brownian web. For $d=3$, we use the Lyapunov function technique to conclude our result.
  In the next section we prove Proposition \ref{prop:BetaExpTail}. 

\section{ Proof of Proposition \ref{prop:BetaExpTail} }
\label{sec:ProofTauTail} 

\subsection{Proof of the main proposition}
\label{subsec:ProofTauTail}
Proposition \ref{prop:BetaExpTail}  will be proved through a sequence of lemmas. 
We first present a property of the joint exploration process 
that will be heavily used in the sequel. We need to introduce some notations first. 
Let $C(\mathbf{0}) := \{ (r,\theta) : r \geq 0, \theta \in (\pi/4, 3\pi/4)\}$ denote the $\pi/2$ 
angular (open) cone centred at the origin. 
For $\x \in \R^d$ the corresponding cone centred at $\x$ is given by $C(\x) := \x + C(\mathbf{0}) $. 

\begin{lemma} 
\label{lemma:UnexploredCone}
Fix any $n \in \N$. Given ${\cal F}_n$, the $\pi/2$ angular cones centred at $g_n(\u)$ and $g_n(\v)$ are unexplored with probability $1$, i.e., we have,
\begin{equation}
\label{eq:UnexploredCone}
(C(g_n(\u)) \cup C(g_n(\v))) \cap H_n = \emptyset  \text{ a.s. }
\end{equation}
In fact for any $l_1, l_2 \geq 0$ we have 
$$
(C(g_n(\u) + l_1 e_d) \cup C(g_n(\v)  + l_2 e_d)) \cap H_n = \emptyset  \text{ a.s. }
$$
\end{lemma}

\noindent \textbf{Proof: } The above lemma is a consequence of simple geometric properties. 
See Figure 1 for an illustration of this lemma.
We first recall the fact that for any $n \geq 1$, the explored region $H_n$ is of the 
form $\cup_{i=1}^k B^+(\x_i,r_i) \cap \mathbb{H}^+(g_n(\u)(d)\wedge g_n(\v)(d))$ for some $k \geq 1$, point $\x_i \in \mathbb{H}^-(g_n(\u)(d)\wedge g_n(\v)(d)) $ and $r_i \geq 0$ for all $1 \leq i\leq k$.
If any $||\quad ||_1$ triangle $B^+(\x , r)$ with $\x(d) < g_n(\u)(d)\wedge g_n(\v)(d)$
intersects with either of the two cones, $C(g_n(\u))$ and $C(g_n(\v))$, then either $g_n(\u)$
or $g_n(\v)$ must lie in the interior of $B^+(\x , r)$. This contradicts the fact 
that the interior of the explored region $H_n$ must be 
free from the perturbed points. 
Finally for any $l_1, l_2 \geq 0$ we have that  $C(g_n(\u) + l_1 e_d) \subseteq C(g_n(\u))$ as well as 
$C(g_n(\v) + l_2 e_d) \subseteq C(g_n(\v))$. This completes the proof.
\qed

The above lemma will be an important tool to prove Proposition \ref{prop:BetaExpTail}. 
Heuristically, since there are unexplored cones in the upward direction centred at $g_n(\u)$ and $g_n(\v)$, using them one can construct a  favourable configuration for an up step  such that the 
probability of having such a configuration is bounded away from zero irrespective of the 
history set. But because of the dependency of the perturbed lattice points, it is 
hard to implement this seemingly easy strategy for our model. 

In order to illustrate the difficulty involved, we need some more notations. 
Consider the marginal process $\{g_n(\u) : n \geq 0\}$
starting from $\u$. Recall that $g^\downarrow_n(\u)$ denotes the 
lattice point closest to $g_n(\u)$ on the 
hyperplane $\{\x \in \R^d : \x(d) = \lfloor g_n(\u)(d) \rfloor\}$. 
Now we define the following sets of lattice points:
\begin{notation}
\label{def:S^l_n} 
\begin{align*}   
S^\text{low}_n = S^{\text{low}}_n(\u) & := \{ \w \in \Z^d : \w(d) = \lfloor g_n(\u)(d) \rfloor
\text{ with } \w \notin \Gamma_n \text{ and }
||\w - g^\downarrow_n(\u)||_\infty \leq 1 \}. \\
S^\text{up}_n = S^{\text{up}}_n(\u) & := \{ \w \in \Z^d : \w(d) = \lfloor g_n(\u)(d) \rfloor + 2,  || \w - (g^\uparrow_n(\u) + e_d)||_\infty \leq 1 , \w \notin \Gamma_n   \} \text{ and }\\
S^\text{same}_n = S^\text{same}_n(\u) & := \{ \w \in \Z^d : \w(d) = \lfloor g_n(\u)(d) \rfloor + 1,  || \w - g^\uparrow_n(\u) ||_\infty \leq 1 , \w \notin \Gamma_n  \}.
\end{align*}
\end{notation}
$S^{\text{low}}_n$ gives the set of unexplored lattice points $\w$ at the lower level $\{\w \in \Z^d :\w(d) = \lfloor g_n(\u)(d) \rfloor\}$  with 
$||\w - g^\downarrow_n(\u)||_\infty \leq 1$.
Similarly the sets $S^{\text{same}}_n$ and $S^{\text{up}}_n$ are defined.

There could be situations when the feasible regions available for the perturbed points $\w + U_\w$ for $\w \in S^\text{low}_n$ 
are very very restrictive and this creates an obstacle to have a favourable configuration  for an up step at the next step.  
In Figure \ref{fig:PLatticeDifficulty} we describe one such situation. Hence we need to consider significant modifications
of the strategy described earlier.The main idea is to show that there are some `good' situations when the process 
has an uniform  strictly positive lower bound for the probability for taking an up step and 
the process will encounter such good situations often. The reason for considering $||\quad||_\infty$
norm instead of $||\quad||_1$ norm in the definition of the above sets is that, 
for $d=3$, the perturbed point $\w + U_\w$ associated to an unexplored lattice point $\w \in \Z^d \setminus \Gamma_n$ with $\w(d) = \lfloor g_n(\u)(d)\rfloor$ and $||\w - g^\downarrow_n(\u)||_\infty \leq 1$ may still create problem for an up step.  
Now we proceed to the proof of Proposition \ref{prop:BetaExpTail} which will be done 
through a sequence of lemmas. To keep the notations simple, we prove it for a single DSF path process, i.e., $\{g_n(\u) : n \geq 0\}$ and the argument is exactly the same for the joint exploration process starting from $\u$ and $\v$. We need to introduce some more notations.

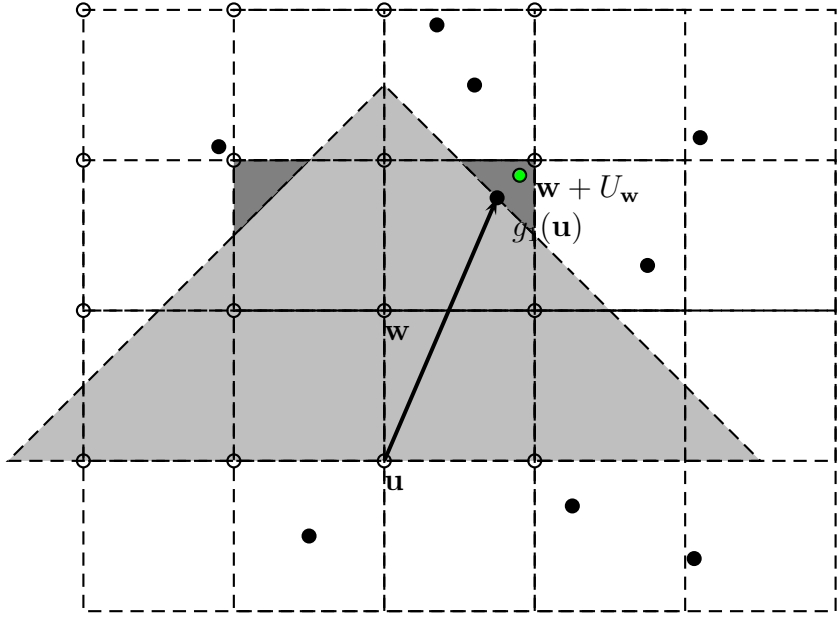
\begin{figure}[ht!]
\label{fig:PLatticeDifficulty}
\begin{center}

\begin{pspicture}(2,-2)(6,7)
%\psgrid[subgriddiv=0,griddots=5,gridlabels=0]
\pspolygon[linestyle=dashed, fillstyle = solid, fillcolor = lightgray](4,5) (-1,0)(9,0)
\pspolygon[linestyle=dashed, fillstyle = solid, fillcolor = gray](6,4)(5,4)(6, 3)
\pspolygon[linestyle=dashed, fillstyle = solid, fillcolor = gray](2,4)(3,4)(2, 3)

\pscircle[fillcolor=black,fillstyle=solid](5.5,3.5){.1}

\pscircle[fillcolor=black,fillstyle=none](0,0){.1}
\pscircle[fillcolor=black,fillstyle=none](2,0){.1}
\pscircle[fillcolor=black,fillstyle=none](4,0){.1}
\pscircle[fillcolor=black,fillstyle=none](6,0){.1}
\pscircle[fillcolor=black,fillstyle=none](0,2){.1}
\pscircle[fillcolor=black,fillstyle=none](2,2){.1}
\pscircle[fillcolor=black,fillstyle=none](4,2){.1}
\pscircle[fillcolor=black,fillstyle=none](6,2){.1}
\pscircle[fillcolor=black,fillstyle=none](0,4){.1}
\pscircle[fillcolor=black,fillstyle=none](2,4){.1}
\pscircle[fillcolor=black,fillstyle=none](4,4){.1}
\pscircle[fillcolor=black,fillstyle=none](6,4){.1}
\pscircle[fillcolor=black,fillstyle=none](0,6){.1}
\pscircle[fillcolor=black,fillstyle=none](2,6){.1}
\pscircle[fillcolor=black,fillstyle=none](4,6){.1}
\pscircle[fillcolor=black,fillstyle=none](6,6){.1}

\pscircle[fillcolor=black,fillstyle=solid](5.2,5){.1}
\pscircle[fillcolor=black,fillstyle=solid](4.7,5.8){.1}

\pspolygon[linestyle=dashed](4,2)(4,-2)(0,-2)(0,2)
\pspolygon[linestyle=dashed](6,2)(6,-2)(2,-2)(2,2)
\pspolygon[linestyle=dashed](8,2)(8,-2)(4,-2)(4,2)
\pspolygon[linestyle=dashed](10,2)(10,-2)(6,-2)(6,2)

\pspolygon[linestyle=dashed](4,4)(4,0)(0,0)(0,4)
\pspolygon[linestyle=dashed](6,4)(6,0)(2,0)(2,4)
\pspolygon[linestyle=dashed](8,4)(8,0)(4,0)(4,4)
\pspolygon[linestyle=dashed](10,4)(10,0)(6,0)(6,4)

\pspolygon[linestyle=dashed](4,6)(4,2)(0,2)(0,6)
\pspolygon[linestyle=dashed](6,6)(6,2)(2,2)(2,6)
\pspolygon[linestyle=dashed](8,6)(8,2)(4,2)(4,6)
\pspolygon[linestyle=dashed](10,6)(10,2)(6,2)(6,6)

\psline[linewidth=1.5pt]{->}(4,0)(5.5,3.5)

\pscircle[fillcolor=black,fillstyle=solid](8.2,4.3){.1}

\pscircle[fillcolor=black,fillstyle=solid](7.5,2.6){.1}

\pscircle[fillcolor=black,fillstyle=solid](1.8,4.18){.1}
\pscircle[fillcolor=green,fillstyle=solid](5.8,3.8){.1}
\pscircle[fillcolor=black,fillstyle=solid](8.12,-1.3){.1}
\pscircle[fillcolor=black,fillstyle=solid](3,-1){.1}

\pscircle[fillcolor=black,fillstyle=solid](6.5,-.6){.1}

\rput[tl](4,-.2){$\mathbf{u}$}
\rput[tl](4,1.8){$\mathbf{w}$}
\rput[tl](6,3.8){$\mathbf{w}  + U_\mathbf{w}$}
\rput[tl](5.7,3.3){$g_1(\mathbf{u})$}

\end{pspicture}
\caption{For $d=2$, the above figure explains the difficulty with this model. Given that $\mathbf{u}$ connects to the point
$g_1(\mathbf{u})$, the feasible region for the perturbed point
$\mathbf{w} + U_{\mathbf{w}}$  (shown as a green dot) is very restrictive (shown as 
the gray shaded region in the above figure). This causes problem for creating a favourable configuration for an up step as the point $g_1(\u)$ might be forced to take step to $\mathbf{w} + U_{\mathbf{w}}$.}
\end{center}
\end{figure}

We recall the definitions of $S^{\text{same}}_n$ and $S^{\text{up}}_n$
from Notation \ref{def:S^l_n}. For the joint exploration process 
$\{ (g_n(\u), g_n(\v)) : n \geq 0\}$ the sets $S^\text{lower}_n, S^\text{up}_n, S^\text{same}_n$ are defined only when we have $\lfloor g_n(\u)(d) \rfloor = \lfloor g_n(\v)(d) \rfloor$ and in that case these sets are defined in a similar way
 by considering both the vertices $g_n(\u)$ and $g_n(\v)$. In other words, they are simply union of the corresponding sets for  $g_n(\u)$ and $g_n(\v)$.

In what follows, for any set $A$, the notation $\# A$ denotes the cardinality of $A$.     
Our next corollary follows from simple geometric properties.
\begin{corollary}
\label{cor:JtExplorationProcessS_n_Properties}
Fix $n \in \N$. Given ${\cal F}_n$, there exist $c^1_d, c^2_d \in \N$ depending only on $d$
 such that we have
\begin{itemize}
\item[(i)] $\# (S^\text{low}_n) \leq c^1_d$;
\item[(ii)] $2 \leq \# (S^\text{up}_n \cup S^\text{same}_n) \leq c^2_d $.
\end{itemize}
\end{corollary}
\noindent \textbf{ Proof :} Item (i) follows from the observation that $\# (S^\text{low}_n) \leq 2d$. 
For item (ii) the same argument provides an upper bound for the set $S^\text{up}_n \cup S^\text{same}_n$.
It is straightforward to observe that for the marginal process 
$\{g_n(\u) : n \geq 0\}$, none of the vertices in  the set $\{\w \in \Z^d : \w(d) = \lfloor g_n(\u)(d) \rfloor + 2 \}$ can belong to the set $\Gamma_n$ and hence the number $2$ clearly provides an lower bound for the cardinality of the set $S^\text{up}_n \cup S^\text{same}_n$.    
For the joint exploration process, if there exists a lattice point $\w \in \Gamma_n $ with $
\w(d) = \lfloor g_n(\u)(d) \rfloor + 2 $ then we must have $g_j(\v) = \w + U_\w$ for some $j \leq n$. 
Further, our movement algorithm ensures that after reaching this perturbed point $\w + U_\w$, 
the other DSF path stays put there. The situation remains the same if the roles of $\u$ and $\v$
are interchanged. This ensures that while working with the joint exploration process,
 at most one point in the set $\{\w \in \Z^d : \w(d) = \lfloor g_n(\u)(d) \rfloor + 2 \}$
 belongs to $\Gamma_n$. This completes the proof.
\qed

Now we proceed with the proof of Proposition \ref{prop:BetaExpTail}. 
We will first prove this proposition for the marginal process $\{g_n(\u) : n \geq 0\}$ 
and then for the joint process.  It will 
be proved through a sequence of lemmas. First we will state some lemmas and assuming that 
these lemmas hold true we will prove Proposition \ref{prop:BetaExpTail}. After 
that in Section \ref{sec:lemmas}, we will prove each of these lemmas. 

From the above discussion it follows that the points in the set $S^{\text{low}}_n$ 
may create problem for an up step. The next lemma says that for $\w \in S^{\text{low}}_n$, 
the point $\w + U_\w$ could be placed either very 
close to the top face of the box $B(\w) := \w + [-1, +1]^d$ or in the lower half-plane 
$\mathbb{H}^-(g_n(\u)(d))$ with probability uniformly bounded away from zero. The intuition is that, 
if any such point lies too close to the top face of the box $B(\w)$
and if $g_n(\u)$ connects to that point, it is highly likely that soon after
it will take an up step.  

Recall the small positive constant $\delta$. For $\w \in \Z^d$, 
we consider the region $R(\w) := B(\w) \cap \mathbb{H}^+(\w(d) + (1 - \delta))$
 close to the top face of the box $B(\w)$. 
\begin{lemma}
\label{lem:TopFace_LowRegionBound}
Fix $n \in \N$. There exists $p_0 > 0$ which depends only on $\delta$ and dimensions $d$
such that
for any $\w \in S^{\text{low}}_n$ we have 
\begin{equation}
\label{eq:TopFace_LowRegionBound}
\P( \w + U_\w \in (R(\w ) \cup \mathbb{H}^-(g_n(\u)(d))  \mid  {\cal F}_n ) \geq p_0,
\end{equation}
\end{lemma}
We will present the proof of this lemma in Section \ref{sec:lemmas}. 
Based on the above lemma, the next lemma shows that given ${\cal F}_n$, the process takes
at most geometric many steps to take an up step.    
Given ${\cal F}_n$, let  
\begin{align}
\label{def:SpecialSingleUpStep}
\nu_n := \inf\{ m \geq 1: \text{ the step }\langle g_{n+m-1}(\u), g_{n+m}(\u) \rangle
\text{ is an up step} \}.
\end{align}
The next lemma shows that not only the random variable $\nu_n$ is finite, 
it's tail decays exponentially.
\begin{lemma} 
\label{lem:SpecialSingleUpStep} 
Fix $n \in \N$. Given ${\cal F}_n$ there exist positive constants $C_0, C_1$ 
and a positive integer $c_{d}$ which depend only on the dimension $d$,  
 such that for all $l \geq 1$ we have
\begin{equation}
\label{eq:SpecialSingleUpStepExpDecay}
\P( \nu_n \geq c_{d} l \mid {\cal F}_n) \leq C_0 \exp{(-C_1 l)},
\end{equation} 
where $c^1_{d}$ as in Corollary \ref{cor:JtExplorationProcessS_n_Properties}.
\end{lemma}
The next lemma shows that given ${\cal F}_n$, conditional on the event that the last 
step is an up step, the probability that the next step is also an up step is uniformly bounded 
away from zero.
\begin{lemma}
\label{lem:GivenUpStep_NextUpStepBound}
Given ${\cal F}_n$, there exists $p_1 > 0$ depending only on $\delta$ and
the dimension $d$ such that conditional to the $\sigma$-field ${\cal F}_n$ given the $n$-th (last)
step is an \textit{up step}, the probability that the $(n+1)$-th step is also an up step 
is uniformly bounded from below by $p_1$.
\end{lemma}
Now assuming the above three lemmas hold true, we first complete the proof of Proposition \ref{prop:BetaExpTail}. 

\noindent \textbf{Proof of Proposition \ref{prop:BetaExpTail}:} 
We first prove Proposition \ref{prop:BetaExpTail} for $j = 0$. We recall the definition of $m_d \in \N$ from Lemma \ref{lem:HistoryHeightBound}. 
Because of Lemma \ref{lem:TopFace_LowRegionBound} and Lemma \ref{lem:GivenUpStep_NextUpStepBound} 
we have that given ${\cal F}_n$, the minimum 
number of steps required to take $m_d$ many consecutive up steps 
has exponentially decaying tail. In order to complete the proof of Proposition \ref{prop:BetaExpTail}
we need to show that given ${\cal F}_n$, after $m_d - 1$ many up steps, the probability of taking a
`special' up step is uniformly bounded away from zero. 

To do that, consider the situation that 
the process takes $m_d - 2$ many consecutive up steps. We recall from 
Lemma \ref{lem:HistoryHeightBound} that 
$L(H_n) \leq m_d - 4$ a.s. This implies that after $m_d - 3$ many up steps the \textit{old} 
history region has to go away. Set $\x_1 = 
\widehat{g_{n + m_d - 2}}(\u), \x_2  = \x_1 + e_d$ and $\x_3 = \x_2 + e_d$.
For an illustration of this lemma see Figure \ref{fig:TauTail} where these points $\x_1, \x_2$
and $\x_3$ are marked as red dots. 
 
After $m_d - 3$ many consecutive 
up steps followed by another up step 
ensures that the history region $H_{n + (m_d -2)}$ is \textit{properly} contained in $B^+(\x_1, 2\delta)$. 

This ensures that the lattice points $\x_2,\x_3$ as well as the lattice points in $N(\x_3)$, defined as in (\ref{def:InftyNbd}) are unexplored.  We define an event 
\begin{align*}
A^3_n := \{& \w + U_\w \in B^+(\w, \delta)  \text{ for all }\w \in N(\x_3) \cup \{ \x_2\}
\text{ and one of the following happens: }\\
& \text{ either for all }\w \in \Z^d \setminus \Gamma_{n + m_d - 2}
\text{ with }\w(d) = \x_1(d) \text{ and }||\x_1 - \w ||_1 \leq 2 \\
& \text{ we have } \w  + U_\w \in B^-(\w + e_d, \delta);\\
& \text{ or for all }\w \in \Z^d \setminus \Gamma_{n + m_d - 2}
\text{ with }\w(d) = \x_2(d) \text{ and }||\x_2 - \w ||_1 \leq 2 \\
& \text{ we have } \w  + U_\w \in B^-(\w + e_d, \delta)
\}.
\end{align*}
For an illustration of the event $A^3_n$ refer to Figure \ref{fig:A^3_n}.
Note that $r = ||\x_1 - \x_2||_\infty + 2\delta$ gives the maximum distance possible between any two points in $B^+(\x_1, \delta)$ and $B^+(\x_2, \delta)$. Since $\delta > 0$ is chosen arbitrarily small, the event $A^3_n$ ensures that the perturbed point $\x_2 + U_{\x_2} $ is the closest point to $\x_1$ in $V$ and hence it must take an up step. Note that, after $m_d  - 2$ many up steps all the lattice points in $N(\x_2)$ may not be unexplored and hence we can not have a special up step there, whereas all the points in the set $N(\x_3)$  are unexplored. This is precisely the reason for taking $m_d \in \N$ so that $L(H_n) \leq m_d - 4$.   The same argument tells us that on the event $A^3_n$, the perturbed point in $B^+(\x_2, \delta)$ (which is unique) must take an up step to $B^+(\x_3, \delta)$. Further, distribution of the perturbed points for $\w \in N(\x_3)$ guarantees that this is a special up step.

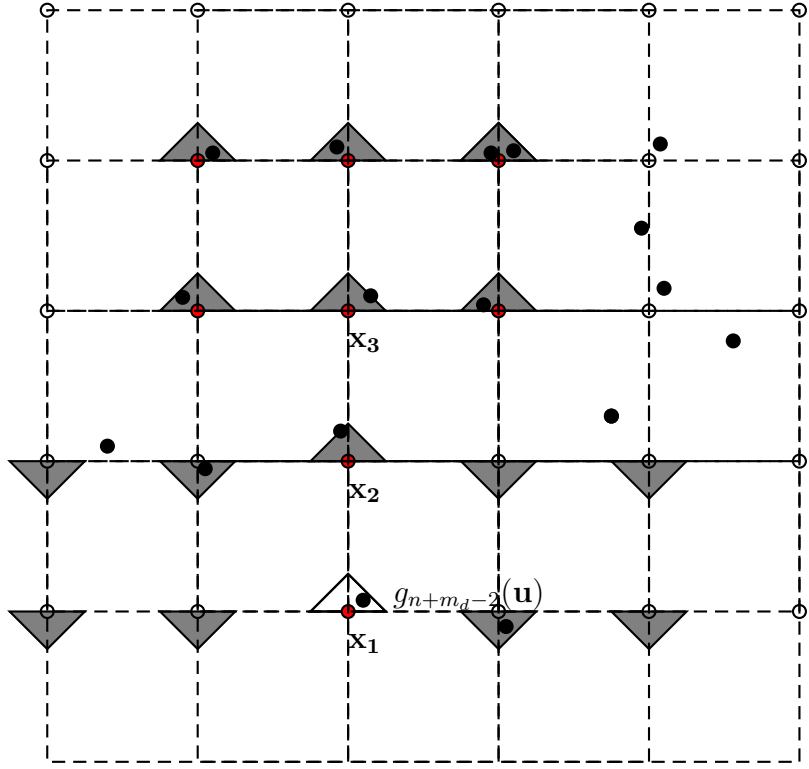
\begin{figure}[ht!]
\label{fig:A^3_n}
\begin{center}

\begin{pspicture}(0,-2)(10,8.5)
%\psgrid[subgriddiv=0,griddots=5,gridlabels=0]

%\pspolygon[linestyle=dashed, fillstyle = solid , fillcolor = lightgray](3.9,2.4) (3.5,2.8)(3.1,2.4)
%\pspolygon[linestyle=dashed, fillstyle = solid , fillcolor = lightgray](7.4,2.4) (8.6,2.4)(8,3)
%\pspolygon[fillcolor = lightgray, fillstyle = solid](4,5)(1,2)(7,2)
%\pspolygon[fillcolor = lightgray, fillstyle = solid](4,3)(1,0)(7,0)
\pspolygon[fillcolor = white, fillstyle = solid](4,0.5)(4.5,0)(3.5,0)
\pspolygon[fillcolor = gray, fillstyle = solid](2,1.5)(2.5,2)(1.5,2)
\pspolygon[fillcolor = gray, fillstyle = solid](0,1.5)(0.5,2)(-0.5,2)
\pspolygon[fillcolor = gray, fillstyle = solid](8,1.5)(8.5,2)(7.5,2)
\pspolygon[fillcolor = gray, fillstyle = solid](6,1.5)(6.5,2)(5.5,2)
\pspolygon[fillcolor = gray, fillstyle = solid](2,-0.5)(2.5,0)(1.5,0)
\pspolygon[fillcolor = gray, fillstyle = solid](0,-0.5)(0.5,0)(-0.5,0)
\pspolygon[fillcolor = gray, fillstyle = solid](8,-0.5)(8.5,0)(7.5,0)
\pspolygon[fillcolor = gray, fillstyle = solid](6,-0.5)(6.5,0)(5.5,0)

\pspolygon[fillcolor = gray, fillstyle = solid](4,4.5)(4.5,4)(3.5,4)
\pspolygon[fillcolor = gray, fillstyle = solid](2,4.5)(2.5,4)(1.5,4)
\pspolygon[fillcolor = gray, fillstyle = solid](6,4.5)(6.5,4)(5.5,4)
\pspolygon[fillcolor = gray, fillstyle = solid](4,6.5)(4.5,6)(3.5,6)
\pspolygon[fillcolor = gray, fillstyle = solid](2,6.5)(2.5,6)(1.5,6)
\pspolygon[fillcolor = gray, fillstyle = solid](6,6.5)(6.5,6)(5.5,6)
\pspolygon[fillcolor = gray, fillstyle = solid](6,6.5)(6.5,6)(5.5,6)
\pspolygon[fillcolor = gray, fillstyle = solid](4,2.5)(4.5,2)(3.5,2)
\pspolygon[fillcolor = gray, fillstyle = none](4,0.5)(4.5,0)(3.5,0)
%\pspolygon[fillcolor = gray, fillstyle = solid](8,4.5)(8.5,4)(7.5,4)
%\pspolygon[fillcolor = gray, fillstyle = solid](10,4.5)(10.5,4)(9.5,4)
%\pspolygon[fillcolor = gray, fillstyle = solid](8,6.5)(8.5,6)(7.5,6)
%\pspolygon[fillcolor = gray, fillstyle = solid](8,4.5)(8.5,4)(7.5,4)
%\pspolygon[fillcolor = gray, fillstyle = solid](10,6.5)(10.5,6)(9.5,6)

\pscircle[fillcolor=black,fillstyle=none](0,0){.1}
\pscircle[fillcolor=black,fillstyle=none](2,0){.1}
\pscircle[fillcolor=red,fillstyle=solid](4,0){.1}
\pscircle[fillcolor=black,fillstyle=none](6,0){.1}
\pscircle[fillcolor=black,fillstyle=none](8,0){.1}
\pscircle[fillcolor=black,fillstyle=none](10,0){.1}
\pscircle[fillcolor=black,fillstyle=none](0,2){.1}
\pscircle[fillcolor=black,fillstyle=none](2,2){.1}
\pscircle[fillcolor=red,fillstyle=solid](4,2){.1}
\pscircle[fillcolor=black,fillstyle=none](6,2){.1}
\pscircle[fillcolor=black,fillstyle=none](8,2){.1}
\pscircle[fillcolor=black,fillstyle=none](10,2){.1}
\pscircle[fillcolor=black,fillstyle=none](0,4){.1}
\pscircle[fillcolor=red,fillstyle=solid](2,4){.1}
\pscircle[fillcolor=red,fillstyle=solid](4,4){.1}
\pscircle[fillcolor=red,fillstyle=solid](6,4){.1}
\pscircle[fillcolor=red,fillstyle=none](8,4){.1}
\pscircle[fillcolor=red,fillstyle=none](10,4){.1}
\pscircle[fillcolor=black,fillstyle=none](0,6){.1}
\pscircle[fillcolor=red,fillstyle=solid](2,6){.1}
\pscircle[fillcolor=red,fillstyle=solid](4,6){.1}
\pscircle[fillcolor=red,fillstyle=solid](6,6){.1}
\pscircle[fillcolor=red,fillstyle=none](8,6){.1}
\pscircle[fillcolor=red,fillstyle=none](10,6){.1}
\pscircle[fillcolor=black,fillstyle=none](0,8){.1}
\pscircle[fillcolor=black,fillstyle=none](2,8){.1}
\pscircle[fillcolor=black,fillstyle=none](4,8){.1}
\pscircle[fillcolor=black,fillstyle=none](6,8){.1}
\pscircle[fillcolor=black,fillstyle=none](8,8){.1}
\pscircle[fillcolor=black,fillstyle=none](10,8){.1}

\pspolygon[linestyle=dashed](4,2)(4,-2)(0,-2)(0,2)
\pspolygon[linestyle=dashed](6,2)(6,-2)(2,-2)(2,2)
\pspolygon[linestyle=dashed](8,2)(8,-2)(4,-2)(4,2)
\pspolygon[linestyle=dashed](10,2)(10,-2)(6,-2)(6,2)

\pspolygon[linestyle=dashed](4,4)(4,0)(0,0)(0,4)
\pspolygon[linestyle=dashed](6,4)(6,0)(2,0)(2,4)
\pspolygon[linestyle=dashed](8,4)(8,0)(4,0)(4,4)
\pspolygon[linestyle=dashed](10,4)(10,0)(6,0)(6,4)

\pspolygon[linestyle=dashed](4,6)(4,2)(0,2)(0,6)
\pspolygon[linestyle=dashed](6,6)(6,2)(2,2)(2,6)
\pspolygon[linestyle=dashed](8,6)(8,2)(4,2)(4,6)
\pspolygon[linestyle=dashed](10,6)(10,2)(6,2)(6,6)

\pspolygon[linestyle=dashed](4,8)(4,4)(0,4)(0,8)
\pspolygon[linestyle=dashed](6,8)(6,4)(2,4)(2,8)
\pspolygon[linestyle=dashed](8,8)(8,4)(4,4)(4,8)
\pspolygon[linestyle=dashed](10,8)(10,4)(6,4)(6,8)

\pscircle[fillcolor=black,fillstyle=solid](2.1,1.9){.1}
\pscircle[fillcolor=black,fillstyle=solid](7.5,2.6){.1}

\pscircle[fillcolor=black,fillstyle=solid](3.9, 2.4){.1}

\pscircle[fillcolor=black,fillstyle=solid](4.2,0.15){.1}

%\pspolygon[linestyle=dashed](5,0) (-1,0)(2,3)
%\pspolygon[linestyle=dashed](5,0) (11,0)(8,3)
%\pspolygon[linestyle=dashed](4.8,1.5) (2.2,1.5)(3.5,2.8)

\pscircle[fillcolor=black,fillstyle=solid](4.3,4.2){.1}
\pscircle[fillcolor=black,fillstyle=solid](8.2,4.3){.1}

\pscircle[fillcolor=black,fillstyle=solid](6.1, -0.2){.1}
\pscircle[fillcolor=black,fillstyle=solid](7.5,2.6){.1}
\pscircle[fillcolor=black,fillstyle=solid](0.8,2.2){.1}
\pscircle[fillcolor=black,fillstyle=solid](1.8,4.18){.1}
\pscircle[fillcolor=black,fillstyle=solid](2.2,6.1){.1}
\pscircle[fillcolor=black,fillstyle=solid](3.85,6.18){.1}
\pscircle[fillcolor=black,fillstyle=solid](6.2,6.13){.1}
\pscircle[fillcolor=black,fillstyle=solid](5.9,6.1){.1}
\pscircle[fillcolor=black,fillstyle=solid](5.8,4.08){.1}
\pscircle[fillcolor=black,fillstyle=solid](8.15,6.22){.1}
\pscircle[fillcolor=black,fillstyle=solid](7.9,5.1){.1}
\pscircle[fillcolor=black,fillstyle=solid](9.12,3.6){.1}

\rput[tl](4,-.3){$\mathbf{x_1}$}
\rput[tl](4.6,.45){$g_{n + m_d - 2} (\mathbf{u})$}
\rput[tl](4,1.7){$\mathbf{x_2}$}
\rput[tl](4,3.7){$\mathbf{x_3}$}
%\rput[tl](3.7,1.3){$g_1(\mathbf{u})$}
%\rput[tl](5.2,2.6){$g_2(\mathbf{v}) = g_1(\mathbf{v})$}
%\rput[tl](4.1, 2.5){$g_2(\mathbf{u})$}

%\psline[linecolor=blue](3.9, 2.4)(0.4,5.9)
%\psline[linecolor=blue](3.9, 2.4)(7.4,5.9)

%\psline[linecolor=blue](7.5,2.6)(4.2,5.9)
%\psline[linecolor=blue](7.5,2.6)(10.8,5.9)

\end{pspicture}
\end{center}
\caption{For $d=2$, this figure represents the event $A^3_n$. The lattice points $\mathbf{x}_1, \mathbf{x}_2,\mathbf{x}_3$ and the points in $N(\mathbf{x}_3)$ are represented as red dots. 
Corresponding to each
of the gray shaded regions, if the corresponding lattice point is unexplored (some of them might be explored as well) 
then the associated perturbed point belongs to the gray region only. This forces the point $g_{n + m_d -2}(\mathbf{u})$ to take an up step followed by a special up step.}
\end{figure}

Since all the lattice points in  $N( \x_3)$ remain unexplored and because, after $m_d - 2$ many consecutive up steps, 
we have $H_{n + m_d -2} \subset 
B^+(\x_1, 2\delta)$, 
it follows that $\P(A^3_n \mid \text{ last }(m_d - 2)\text{ many up steps })$ is uniformly bounded from below. 
 This completes the proof for $j  = 0$. 

In order to prove Proposition \ref{prop:BetaExpTail} for $j = 1$ we observe that 
given ${\cal G}_{\tau_1}$, we can restart the process from the lattice point
$\widehat{g_{\tau_1}}(\u)$ and it is guaranteed that $h(\widehat{g_{\tau_1}}(\u)) = g_{\tau_1}(\u)$.
But restarting the process from $\widehat{g_{\tau_1}}(\u)$ is different from our initial condition as 
we have the information that $\w + U_\w \in B^+(\w, \delta)$ for all $\w \in N(\widehat{g_{\tau_1}}(\u))$.
It follows that in geometric many steps the restarted process reaches the upper half-plane $\mathbb{H}^+(\lfloor g_{\tau_1}(\u)(d)\rfloor + 2)$ and for any lattice point $\w$ in this upper half-plane 
we don't have any information about the distribution of the associated perturbed point. 
Hence thereafter we can proceed with the same argument and this completes the proof for $j = 1$. 
The proof for general $j \geq 1$ is exactly the same.  
\qed

\subsection{Proofs of the lemmas in Section \ref{subsec:ProofTauTail}}
\label{sec:lemmas}

In this section we prove Lemma \ref{lem:TopFace_LowRegionBound}, Lemma \ref{lem:SpecialSingleUpStep} and Lemma \ref{lem:GivenUpStep_NextUpStepBound} which were used to prove Proposition
\ref{prop:BetaExpTail}. 

\noindent \textbf{Proof of Lemma \ref{lem:TopFace_LowRegionBound}:} 
Fix $\w \in S^{\text{low}}_n$ and we claim that:
\begin{claim}
\label{claim:LebesgueMeasureRatio}
 In order to prove Lemma \ref{lem:TopFace_LowRegionBound},
 it suffices to show that 
\begin{equation}
\label{eq:LebesgueMeasureRatio}
\frac{\ell( R(\w) \setminus H_n)}{\ell \bigl( \bigl( B(\w)
\cap \mathbb{H}^+(g_n(\u)(d)) \bigr) \setminus H_n \bigr)}
\geq \hat{c},
\end{equation}
where for any Borel $A \subset \R^d$, the number $\ell(A)$ 
denotes the Lebesgue measure of $A$.
\end{claim}

Equation (\ref{eq:LebesgueMeasureRatio}) suggests that
 there exists a positive constant $\hat{c}$ depending on $\delta$ and $d$
such that the ratio of the Lebesgue measure of the unexplored part in $B(\w)\cap \mathbb{H}^+(g_n(\u)(d))$ and the Lebesgue measure of the unexplored part in $R(\w)$
is uniformly bounded from below by $\hat{c}$. Note that 
the perturbed point $\w + U_\w$ belongs inside the box $B(\w) = \w + [-1, +1]^d$
 either in the upper half-plane 
$\mathbb{H}^+(g_n(\u)(d))$ or in the lower half-plane $\mathbb{H}^-(g_n(\u)(d))$.
Define the event $B = \{\w + U_\w \in \mathbb{H}^-(g_n(\u)(d))\}$. 
We obtain 
\begin{align*}
& \P(\w + U_\w \in (R(\w) \cup \mathbb{H}^-(g_n(\u)(d)))\mid {\cal F}_n)\\
= & \P(\w + U_\w \in (R(\w) \cup \mathbb{H}^-(g_n(\u)(d))) \cap B \mid {\cal F}_n) + 
\P(\w + U_\w \in (R(\w) \cup \mathbb{H}^-(g_n(\u)(d))) \cap B^c \mid {\cal F}_n)\\
\geq & \P(\w + U_\w \in (R(\w) \cup \mathbb{H}^-(g_n(\u)(d))) \cap B^c \mid {\cal F}_n)\\
= & \frac{\ell(R(\w) \setminus H_n)}{\ell \bigl( B(\w)\cap \mathbb{H}^+(g_n(\u)(d))) \setminus H_n \bigr )} \geq \hat{c}.
\end{align*} 
The last inequality follows from (\ref{eq:LebesgueMeasureRatio})
and the penultimate equality follows from the fact that the perturbed point $\w + U_\w$ is uniformly distributed over the unexplored  part in $B(\w)$. 
This justifies our claim \ref{claim:LebesgueMeasureRatio}.

\begin{figure}[ht!]
\label{fig:TopFace}
\begin{center}

\begin{pspicture}(-2, -.5)(12,9)
\pspolygon[fillcolor=gray,fillstyle=solid](2.9,4)(3.2,4)(3.2, 3.7)
\pspolygon[fillcolor=gray,fillstyle=solid](3.2,4)(3.5,4)(3.2, 3.7)
\pspolygon[fillcolor=gray,fillstyle=solid](10.9,4)(11.2,4)(11.2, 3.7)
\pspolygon[fillcolor=gray,fillstyle=solid](11.2,4)(11.5,4)(11.2, 3.7)
\pscircle[fillcolor=red,fillstyle=solid](3.2,4){.07}
\pscircle[fillcolor=red,fillstyle=solid](11.2,4){.07}

\pspolygon[fillcolor=lightgray,fillstyle=solid](5.4,1.5)(-2.4,1.5)(1.5, 5.4)
\pspolygon[fillcolor=lightgray,fillstyle=solid](2.9,3.5)(4.5, 3.5)(3.7, 4.3)

\pscircle[fillcolor=black,fillstyle=none](2,2){.07}
\pscircle[fillcolor=black,fillstyle=none](2,0){.07}
\pscircle[fillcolor=black,fillstyle=none](2,4){.07}
\pscircle[fillcolor=black,fillstyle=none](0,2){.07}
\pscircle[fillcolor=black,fillstyle=none](0,4){.07}
\pscircle[fillcolor=black,fillstyle=none](0,0){.07}
\pscircle[fillcolor=black,fillstyle=none](4,2){.07}
\pscircle[fillcolor=black,fillstyle=none](4,4){.07}
\pscircle[fillcolor=black,fillstyle=none](4,0){.07}

\psline[linecolor = green](-0.5,3.74)(4.7,3.74)

\pspolygon[linestyle=dashed](4,4)(4,0)(0,0)(0,4)
\pspolygon[fillcolor=black,fillstyle=none](4,4)(4,3.7)(0,3.7)(0,4)

\pspolygon[fillcolor=lightgray,fillstyle=solid](12,1.5)(6,1.5)(9, 4.5)

\pscircle[fillcolor=black,fillstyle=none](10,2){.07}
\pscircle[fillcolor=black,fillstyle=none](10,0){.07}
\pscircle[fillcolor=black,fillstyle=none](10,4){.07}
\pscircle[fillcolor=black,fillstyle=none](8,2){.07}
\pscircle[fillcolor=black,fillstyle=none](8,4){.07}
\pscircle[fillcolor=black,fillstyle=none](8,0){.07}
\pscircle[fillcolor=black,fillstyle=none](12,2){.07}
\pscircle[fillcolor=black,fillstyle=none](12,4){.07}
\pscircle[fillcolor=black,fillstyle=none](12,0){.07}

\pspolygon[linestyle=dashed](12,4)(12,0)(8,0)(8,4)
\pspolygon[fillcolor=black,fillstyle=none](12,4)(12,3.7)(8,3.7)(8,4)

\psline{->}(3.8,3.8)(6,3.8)
\psline{<-}(7.3,3.8)(10.2,3.8)

\rput[tl](2,-.7){Case (ii)}
\rput[tl](10,-.7){Case (i)}
\rput[tl](2,1.8){$\mathbf{w}$}
\rput[tl](10,1.8){$\mathbf{w}$}
\rput[tl](4.9,3.7){$y = g_n(\u)(2)$}
\rput[tl](2,.9){$B(\mathbf{w})$}
\rput[tl](10,.9){$B(\mathbf{w})$}
\rput[tl](6,4.2){$R(\mathbf{w})$}
\rput[tl](3.2,4.5){$\mathbf{x}$}
\rput[tl](11.2,4.5){$\mathbf{x}$}

\end{pspicture}

\end{center}
\caption{Case (ii) and case (i) illustrated for $d=2$. The top rectangles represent the region $R(\w)$. The lightgray regions represent the history sets. A choice of $\x \in B(\w)$ with $\x(2) = \w(2) + 1$
is indicated in the figure as red dots. The corresponding left triangle $\Delta^\x_l$ and $\Delta^\x_r$
are also mentioned as gray shaded regions. In the first picture (representing case (ii)) there is no such $\x$ with an unexplored left or right triangle whereas in the second picture for the given choice 
of $\x$, both the triangles are unexplored. In the first picture, the green line represents the line $y = g_n(\u)(2)$.}

\end{figure}
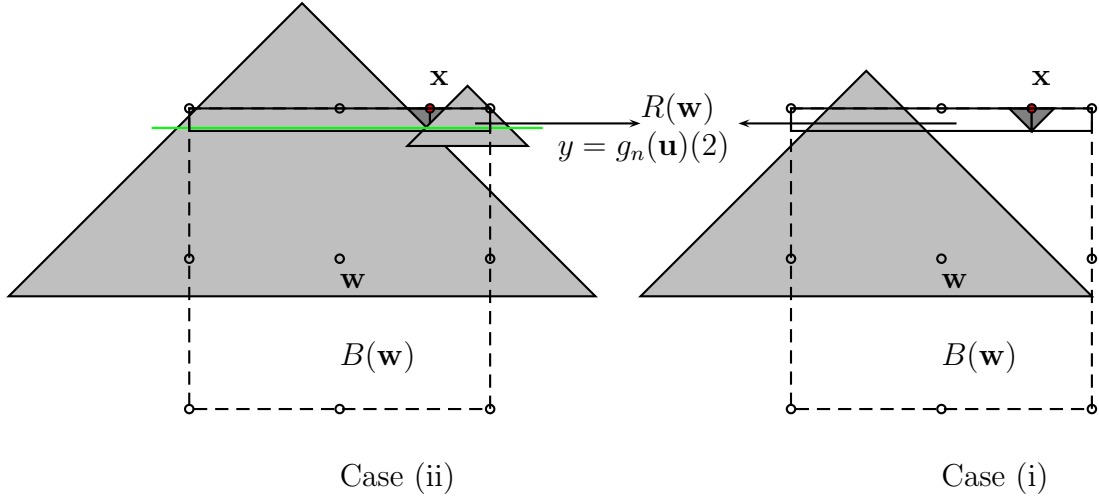

The proof of (\ref{eq:LebesgueMeasureRatio}) is based on a geometric argument.  
In order to keep the notations simple we present the proof for $d=2$. 
The same argument applies for $d=3$ as well. 
For any $\x \in B(\w)$ with $\x(2) = \lfloor g_n(\u)(2)\rfloor + 1$, 
we consider the $||\quad ||_1$ triangle 
$B^-(\x, \delta) = B(\x, \delta) \cap \mathbb{H}^-( \x(2) )$. 
Clearly $B^-(\x, \delta)$ is union of the left triangle $\Delta^\x_l := 
B^-(\x, \delta) \cap \{\y \in \R^2 : \y(1) \leq \x(1)\}$ and the right triangle 
$\Delta^\x_r := B^-(\x, \delta) \cap \{\y \in \R^2 : \y(1) \geq \x(1)\}$ (See Figure \ref{fig:TopFace}).  
We first consider the situation when there exists at least one $\x \in B(\w)$ with 
$\x(2) = \lfloor g_n(\u)(2)\rfloor + 1$ such that either the left triangle 
or the right triangle $\Delta^\x_2$ remain unexplored, i.e., either 
 $\Delta^\x_l \cap H_n = \emptyset$  or $\Delta^\x_r \cap H_n = \emptyset$ (See Figure \ref{fig:TopFace} case (i)).  Assuming that the left triangle $\Delta^\x_l$ is unexplored, 
we must have  
\begin{equation*}
\frac{\ell( R(\w) \setminus H_n)}{\ell \bigl( \bigl( B(\w)
\cap \mathbb{H}^+(g_n(\u)(2)) \bigr) \setminus H_n \bigr)}
\geq \frac{\ell(\Delta^\x_l)}{\ell(B(\w))} = \delta^2/8.
\end{equation*}
When the right triangle $\Delta^\x_r$ remains unexplored, 
the argument is the same and this proves (\ref{eq:LebesgueMeasureRatio}) for this situation, i.e., case (i).

Next in case (ii) we consider the situation that for all $\x \in B(\w) \setminus H_n$ 
with $\x(2) = \w(2) + 1$, we must have both the left triangle \text{and} the right triangle partially explored, i.e.
(see case (ii) in Figure \ref{fig:TopFace}). For more detail illustration of case (ii) see Figure \ref{fig:TopFace_caseII}. 

Since $H_n$ is of the form $\cup_{i=1}^k B^+(\x_i, r_i) 
\cap \mathbb{H}^+(g_n(\u)(2))$
for some finite $k \geq 1$ and $r_i \geq 0, \x_i \in \mathbb{H}^-(g_n(\u)(2))$ 
for all $1 \leq i \leq k$, it follows that the history triangles intersecting with $\Delta^\x_l$ and $\Delta^\x_r$
intersect before crossing the line $y = \w(2) + (1 - \delta)$. Since this is true for all $\x \in B(\w) \setminus H_n$ with $\w(d) = \lfloor g_n(\u)(d) \rfloor + 1$, we may deduce that the unexplored part in 
$B(\w) \in \mathbb{H}^+(g_n(\u)(2))$ is actually contained in $R(\w)$. This ensures that  
\begin{equation*}
\frac{\ell( R(\w) \setminus H_n)}{\ell \bigl( \bigl( B(\w)
\cap \mathbb{H}^+(g_n(\u)(2)) \bigr) \setminus H_n \bigr)}
= 1.  
\end{equation*}
This completes the proof of (\ref{eq:LebesgueMeasureRatio}) and thereby proves
Lemma \ref{lem:TopFace_LowRegionBound}.
 \qed
 
\begin{figure}[ht!]
\label{fig:TopFace_caseII}
\begin{center}

\begin{pspicture}(0,0)(8,8)
%\psgrid[subgriddiv=0,griddots=5,gridlabels=0]

\pspolygon[fillcolor = lightgray, fillstyle = solid](4.15,5.4)(2.35,5.4)(3.25,6.3)
\pspolygon[fillcolor = lightgray, fillstyle = solid](4.15,5.4)(8.15,5.4)(6.15,7.4)
\pspolygon[fillcolor = lightgray, fillstyle = solid](2.35,5.4)(-0.65,5.4)(0.85,6.9)

\pspolygon[linestyle= dashed, fillcolor = gray, fillstyle = none](0,0)(6,0)(6,6)(0,6)
\pspolygon[fillcolor = gray, fillstyle = none](0,5.2)(6,5.2)(6,6)(0,6)
\pspolygon(3.2,6)(4.8,6)(4,5.2)
\psline(4,6)(4,5.2)
\pscircle[fillcolor = black, fillstyle = solid](2.35,5.4){.1}
\pscircle[fillcolor = black, fillstyle = none](3,3){.1}

\psline[linecolor = green](0,5.4)(7,5.4)

\rput[tl](4,6.5){$\mathbf{x}$}
\rput[tl](3.5,3){$\mathbf{w}$}
\rput[tl](2.35,5.3){$g_n(\mathbf{u})$}
\end{pspicture}
\end{center}
\caption{For $d=2$, the situation in case (ii) is described in more detail here. The shaded regions represent the history region with $g_n(\u)$ represented as a black dot. The lattice point $\w$ belongs to $S^\text{low}_n$ and the square represents $B(\w)$.}

\end{figure}
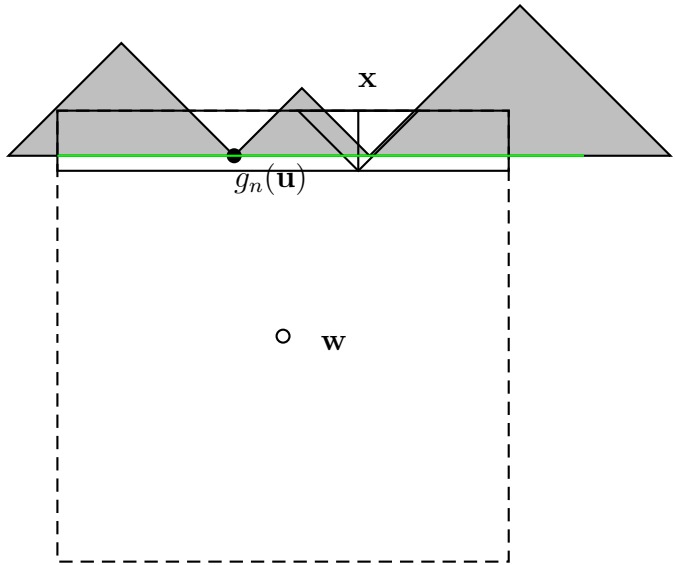

Next we prove Lemma  \ref{lem:SpecialSingleUpStep}.

 \noindent \textbf{Proof of Lemma \ref{lem:SpecialSingleUpStep}:} 
Given ${\cal F}_n$ we first define the event $A^1_n$ which 
ensures that for each $\w$ in $S^{\text{low}}_n$, the unexplored perturbed point
 $\w + U_\w$ belongs either close to their respective top faces
or in the lower half-plane $\mathbb{H}^-(g_n(\u)(d))$. Formally, 
\begin{align}
\label{def:EventA1n}
A^1_n := \{ \w + U_\w \in R(\w)\cup\mathbb{H}^-(g_n(\u)(d)) \text{ for all }\w \in S^{\text{low}}_n \} .
\end{align}
In other words, event $A^1_n$ says that for any $\w \in S^{\text{low}}_n$, the associated perturbed point $\w + U_\w$
must belong close to the top face of $B(\w)$ or in the lower half-plane $\mathbb{H}^-(g_n(\u)(d))$ (and thereby does not affect the next step starting from $g_n(\u)$).
Because of Lemma \ref{lem:TopFace_LowRegionBound} together with Corollary \ref{cor:JtExplorationProcessS_n_Properties},
 we have $\P(A^1_n | {\cal F}_n) \geq (p_0)^{c^1_d}$.
 
For $\x \in \R^d$ let $\overline{\x} := (\x(1), \cdots, \x(d-1))\in \R^{d-1}$ denote 
the projection of $\x$ in the first $d-1$ many co-ordinates. 
Set $\x_1 := (\overline{g_n(\u)}, \lfloor g_n(\u)(d) \rfloor + 1)$ as the projection of 
the point $g_n(\u)$ on the hyperplane $\{\y \in \R^d : \y(d) = \lfloor g_n(\u)(d) \rfloor + 1\}$
and let $\x_2$ denote the lattice point $\x_2 := g^\uparrow_n(\u) + e_d$. 
Because of Lemma \ref{lemma:UnexploredCone}, 
we have that the $\pi/2$ cone $C(\x_1)$ centred  at $\x_1$ does not intersect with $H_n$ and hence unexplored. Take $r := ||\x_1 - \x_2||_1 + 2\delta $ and consider the $||\quad ||_1$ triangle 
$B^+(\x_1, r)$. Now we define the event $A^2_n$ where 
\begin{align}
\label{def:EventA2n}
A^2_n := \{ & \#  ((B^+(\x_2, \delta) \setminus H_n)\cap V) = \# (B^+(\x_1, \delta) \cap V) = 1, 
\nonumber \\
 & \bigl ( B^+(\x_1, r) \setminus (B^+(\x_2, \delta)\cup B^+(\x_1, \delta)\cup H_n) \bigr) \cap V 
= \emptyset \}.
\end{align}
See Figure \ref{fig:TauTail} describing the events $A^1_n$ and $A^2_n$. Note that $r$ is the maximum possible distance  between any two points in $B^+(\x_1, \delta)$ and $B^+(\x_2, \delta)$ and from the choice of $r$ it follows that $B^+(\x_2, \delta) \subset B^+(\x_1, r)$. 
Hence the event $A^2_n$ ensures that the perturbed point in $B^+(\x_1, \delta)$ connects to the perturbed point in 
$B^+(\x_2, \delta)$ as it is the closest one. 

Now let us try to show that on the event $A^1_n \cap A^2_n$ the DSF path must take 
an up step after bounded many steps. For $\x \in \R^d$ let $\overline{\x} := (\x(1), \cdots, \x(d-1)) \in \R^{d-1}$ denote the projection of $\x$ on the first $d-1$ coordinates. 
Because of the presence of the perturbed point in $B^+(\x_1, \delta)$, 
the DSF path must connect to it unless there is a closer point in the lower half-plane $\mathbb{H}^-(\lfloor g_n(\u)(d) \rfloor + 1)$. 
 By Corollary \ref{cor:JtExplorationProcessS_n_Properties},  
there can be at most $c^1_d$ many such closer points in $\mathbb{H}^-(\lfloor g_n(\u)(d) \rfloor + 1)$. Coupled with the presence of a perturbed point in $B^+(\x_1, \delta)$, the fact that we are on the event $A^1_n$ and the fact that $R(\w)$ has height $\delta$, the statement in the last line ensure that 
the DSF path starting from $g_n(\u)$ will never visit a perturbed point $\y$ in  $\mathbb{H}^-(\lfloor g_n(\u)(d) \rfloor + 1)$ with $||\overline{\y} - \overline{g_n(\u)}||_1 > M_d \delta$ for some constant $M_d>0$.
%, i.e., 
%for \textit{no} $k \in \N$ we can have $g_{n+k}(\u) = \y$. 
By choosing $\delta$ sufficiently small, we can always ensure that the DSF path starting from $g_n(\u)$ after taking random number of steps (bounded by $c^1_d$) in $\mathbb{H}^-(\lfloor g_n(\u)(d) \rfloor + 1) $ must connect to the unique perturbed point in $B^+(\x_1, \delta)$.
As discussed earlier, after this, on the event $A^2_n$, the process must take an up step 
to the perturbed point in $B^+(\x_2, \delta)$.
  We observe that the events $A^1_n$ and $A^2_n$ are independent as they depend on perturbations of
disjoint set of lattice points. Hence in order to complete the proof all we just 
need to show that for all $n\geq 0$, 
the probability $\P(A^2_n \mid {\cal F}_n)$ is uniformly bounded away from zero.

Now for the process $\{g_j(\u) : j \geq 0\}$, any $\w \in \Z^d$ with $\w(d) = \lfloor g_n(\u)(d) \rfloor + 2$ can not belong to $\Gamma_n$. By Lemma \ref{lemma:UnexploredCone} we have that the cone 
$C(\x_1)$ remain unexplored which contains the region $B^+(\x_2, \delta)$ as well (see Figure \ref{fig:TauTail}). Hence we can find two distinct lattice points 
$\w_1, \w_2 \in S^{\text{up}}_n$ with $||\w_1 - \x_1||_1 \vee ||\w_2 - \x_2||_1
\leq 1$ which allows us to have 
\begin{align*}
& \P( B^+(\x_1, \delta) \cap V \neq \emptyset, B^+(\x_2, \delta) \cap V \neq \emptyset \mid {\cal F}_n)\\
\geq & \P( \w_1 + U_{\w_1} \in C(\x_1)\cap B^+(\x_1, \delta), \w_2 + U_{\w_2} \in B^+(\x_2, \delta) \mid {\cal F}_n ) \\
= & \delta^2/2^{d+1} \times \delta^2/2^d.
\end{align*}
Finally Corollary \ref{cor:JtExplorationProcessS_n_Properties} gives us that $\#(S^{\text{same}}_n \cup S^{\text{up}}_n)$ remains bounded by $c^2_d$.
The unexplored cone $C(\x_1)$ provides sufficient unexplored spaces
outside the region $B^+(\x_1, r)$ to place the perturbations of the remaining points in 
$S^{\text{same}}_n \cup S^{\text{up}}_n$ excluding the points $\w_1, \w_2$. 
So for any $\w \in S^{\text{same}}_n \cup S^{\text{up}}_n \setminus \{\w_1, \w_2\}$, 
Lebesgue measure of the feasible region $(B(\w) \cap C(\x_1)) \setminus B(\x_1, r)$
is uniformly bounded away from zero $\ell_d$ (say). This follows from the fact that for any $\w \in S^{\text{same}}_n \cup S^{\text{up}}_n \setminus \{\w_1, \w_2\}$, 
Lebesgue measure of the feasible region $(B(\w) \cap C(\x_1)) \setminus B(\x_1, r)$
is a continuous function in $||\w - \x_1 ||_1$ which is positive throughout. 
Hence on a compact domain it must have a strictly positive minima. 
Therefore we have 
$$
\P(A^2_n \mid {\cal F}_n) \geq (\delta^2/2^{d+1} \times \delta^2/2^d) (\ell_d/2^d)^{c^2_d - 2}.
$$ 
 This completes the proof.
 \qed
 
 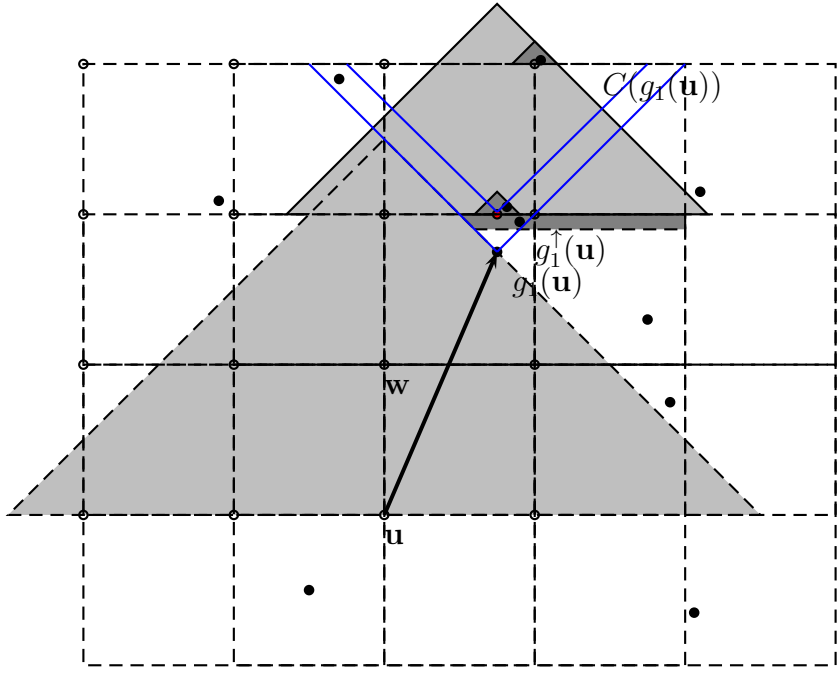
\begin{figure}[ht!]
 \label{fig:TauTail}
\begin{center}

\begin{pspicture}(2,-2)(6,8)

%\pspolygon[fillcolor=gray,fillstyle=none](5.5,3.5)( 8, 6)
\pspolygon[linestyle=dashed, fillstyle = solid, fillcolor = gray](4,4)(8, 4)(8,3.8)(4,3.8)

\pspolygon[fillcolor=lightgray,fillstyle=solid](5.5, 6.8)(8.3,4)(2.7, 4)
\pspolygon[linestyle=dashed, fillstyle = solid, fillcolor = lightgray](4,5) (-1,0)(9,0)

\pspolygon[fillcolor=gray,fillstyle=solid](5.5,4.3)(5.8,4)(5.2, 4)
\pspolygon[fillcolor=gray,fillstyle=solid](6,6.3)(5.7,6)(6.3, 6)

\pscircle[fillcolor=black,fillstyle=solid](5.5,3.5){.07}

\pscircle[fillcolor=black,fillstyle=none](0,0){.07}
\pscircle[fillcolor=black,fillstyle=none](2,2){.07}
\pscircle[fillcolor=black,fillstyle=none](4,2){.07}
\pscircle[fillcolor=black,fillstyle=none](6,2){.07}
\pscircle[fillcolor=black,fillstyle=none](0,2){.07}
\pscircle[fillcolor=black,fillstyle=none](2,0){.07}
\pscircle[fillcolor=black,fillstyle=none](4,0){.07}
\pscircle[fillcolor=black,fillstyle=none](6,0){.07}
\pscircle[fillcolor=black,fillstyle=none](0,4){.07}
\pscircle[fillcolor=black,fillstyle=none](2,4){.07}
\pscircle[fillcolor=black,fillstyle=none](4,4){.07}
\pscircle[fillcolor=black,fillstyle=none](6,4){.07}
\pscircle[fillcolor=black,fillstyle=none](0,6){.07}
\pscircle[fillcolor=black,fillstyle=none](2,6){.07}
\pscircle[fillcolor=black,fillstyle=none](4,6){.07}
\pscircle[fillcolor=red,fillstyle=none](6,6){.07}
\pscircle[fillcolor=black,fillstyle=solid](5.63,4.1){.07}
\pscircle[fillcolor=black,fillstyle=solid](7.8,1.5){.07}

\pscircle[fillcolor=black,fillstyle=solid](3.4,5.8){.07}
\pscircle[fillcolor=red,fillstyle=solid](5.5,4){.07}
\pscircle[fillcolor=black,fillstyle=solid](6.08,6.05){.07}

\pspolygon[linestyle=dashed](4,2)(4,-2)(0,-2)(0,2)
\pspolygon[linestyle=dashed](6,2)(6,-2)(2,-2)(2,2)
\pspolygon[linestyle=dashed](8,2)(8,-2)(4,-2)(4,2)
\pspolygon[linestyle=dashed](10,2)(10,-2)(6,-2)(6,2)

\pspolygon[linestyle=dashed](4,4)(4,0)(0,0)(0,4)
\pspolygon[linestyle=dashed](6,4)(6,0)(2,0)(2,4)
\pspolygon[linestyle=dashed](8,4)(8,0)(4,0)(4,4)
\pspolygon[linestyle=dashed](10,4)(10,0)(6,0)(6,4)

\pspolygon[linestyle=dashed](4,6)(4,2)(0,2)(0,6)
\pspolygon[linestyle=dashed](6,6)(6,2)(2,2)(2,6)
\pspolygon[linestyle=dashed](8,6)(8,2)(4,2)(4,6)
\pspolygon[linestyle=dashed](10,6)(10,2)(6,2)(6,6)

\psline[linewidth=1.5pt]{->}(4,0)(5.5,3.5)

\pscircle[fillcolor=black,fillstyle=solid](8.2,4.3){.07}
\pscircle[fillcolor=black,fillstyle=solid](7.5,2.6){.07}

\pscircle[fillcolor=black,fillstyle=solid](1.8,4.18){.07}
\pscircle[fillcolor=black,fillstyle=solid](5.8,3.9){.07}
\pscircle[fillcolor=black,fillstyle=solid](8.12,-1.3){.07}
\pscircle[fillcolor=black,fillstyle=solid](3,-1){.07}

\psline[linecolor = blue](5.5,3.5)( 8, 6)
\psline[linecolor = blue](5.5,3.5)( 3, 6)

\psline[linecolor = blue](5.5,4.03)( 3.5, 6)
\psline[linecolor = blue](5.5,4.03)( 7.5, 6)

\rput[tl](4,-.2){$\mathbf{u}$}
\rput[tl](4,1.8){$\mathbf{w}$}
\rput[tl](6,3.8){$g^\uparrow_1(\mathbf{u})$}
\rput[tl](5.7,3.3){$g_1(\mathbf{u})$}
\rput[tl](6.9,5.9){$C(g_1(\mathbf{u}))$}

\end{pspicture}
\caption{ On the event $A^1_1\cap A^2_1$, after bounded many steps $g_1(\u)$ takes step 
to the unique perturbed point in $B^+(\x_1, \delta)$ and then takes an up step to a perturbed point in 
$B^+(\x_2, \delta)$. The points $\x_1$ and $\x_2$ are represented as red circles. The cone $C(\x_1)$ is contained in the cone $C(g_1(\u))$ (both represented in blue lines) and hence remain unexplored. $C(\x_1)$ contains the region $B^+(\x_2, \delta)$ as well. We can also observe that there are ample spaces within the cone $C(\x_1)$
outside $B(\x_1, r)$ (represented as a lightgray shaded region) to place perturbations of the 
remaining points in $S^\text{same}_n \cup S^\text{up}_n$.}
\end{center}

\end{figure}

Finally we end this section with the proof of Lemma  \ref{lem:GivenUpStep_NextUpStepBound}. 

\noindent \textbf{Proof of Lemma \ref{lem:GivenUpStep_NextUpStepBound}}
Given that the $n$-th step $\langle g_{n-1}(\u), g_{n}(\u) \rangle$ is an up step, 
$g_n(\u)$ belongs to the region $B^+(g^\downarrow_n(\u), \delta)$. By Lemma \ref{lemma:UnexploredCone}, 
we also have that $C( g^\downarrow_n(\u) + \delta e_d) \cap H_n
= \emptyset $. As the set $S^{\text{up}}_n$ is non-empty always and $B^+(g^\uparrow_n(\u), \delta) \subset 
C( g^\downarrow_n(\u) + \delta e_d) $, using the unexplored cone it is not difficult to create a favourable configuration whose probability is uniformly bounded from below and on this event,
 $g_n(\u)$ must take an up step to $B^+(g^\uparrow_n(\u), \delta)$.  This completes the proof.
\qed

\begin{remark}
For the joint exploration process of $2$ DSF 
paths, the proof is essentially similar and we only give a brief sketch here.
 Fix any $n \in \N$ and given ${\cal F}_n$, 
we need to show that in geometric many steps the process takes an up step.
Rest of the argument is exactly the same. Now, if we have $\lfloor g_n(\u)(d) \rfloor = \lfloor g_n(\v)(d) \rfloor$, then the argument for an up step within geometric many steps is same as that of Lemma \ref{lem:SpecialSingleUpStep}. If not then we need to wait for the lower vertex to come at the same level and the argument of Lemma \ref{lem:SpecialSingleUpStep} also gives us that 
in at most geometric many steps the lower level vertex will catch the other one.
There is one important point that while working with the joint process, at most one element 
in $S^\text{up}_n$ could be used up whereas for the marginal process all the elements in 
$S^\text{up}_n$ remain unexplored. Since at most one element in $S^\text{up}_n$ could be used
up, it does not create any problem and the argument still goes through.      
\end{remark}

\section{Exploration process at renewal steps}
\label{sec:RW_Martingale}
In this section we first show that for the marginal process $\{g_n(\u) : n \geq 0\}$ dealing with a 
single DSF trajectory, the sequence of renewal steps gives rise to a random walk process with i.i.d. 
increments. Later for the joint process we will show that the difference between the two DSF oaths 
observed at the renewal steps gives a Markov chain. 
 
\subsection{Random walk for the marginal path}
\label{sec:RWalk}
For $\x \in \R^d$, the notation $\overline{\x} := (\x(1), \cdots, \x(d-1)) \in \R^{d-1}$
denotes the projection of $\x$ on the first $d-1$ co-ordinates.
For the marginal process $\{g_n (\u) : n \geq 0\}$ we define 
\begin{align}
\label{def:SinglePtRwalk}
Y_{j+1} = Y_{j+1}(\u) := \overline{\widehat{g_{\tau_{j+1}}}}(\u) - \overline{\widehat{g_{\tau_{j}}}}(\u) \in \Z^{d-1} \text{ for }j \geq 0.
\end{align}

Now we focus on the implications of our renewal step construction for  a 
single DSF trajectory starting from $\u$. Recall that for any $j \geq 1$ the
lattice point $\widehat{g_{\tau_j}}(\u)$ is ${\cal G}_{\tau_j}$ measurable as well.
Our next proposition explains the renewal structure observed at these random steps. 
\begin{proposition}  
\label{prop:SinglePt_Rwalk}
$\{ Y_{j+1} : j \geq 1\}$ is a sequence of i.i.d. $\Z^{d-1}$ valued random variables
whose distribution does not depend on the starting point $\u$. 
\end{proposition}
\noindent \textbf{ Proof:} The used up lattice point $\widehat{g_{\tau_j}}(\u) \in \Gamma_{\tau_j} $ is ${\cal G}_{\tau_j}$ measurable. Given ${\cal G}_{\tau_j}$, occurrence of the event 
$A(g^\uparrow_{\tau_j - 1}(\u))$ ensures that when we \textit{restart} the process from the lattice point $\widehat{g_{\tau_j}}(\u)$, then we must have 
$$
h(\widehat{g_{\tau_j}}(\u)) = g_{\tau_j}(\u),
$$
and thereafter it follows the same trajectory. 
By definition of an `up step' it follows that 
the newly created history region during 
an up step is well controlled. We note that at the $\tau_j$-th step, the last $m_d$ many
steps  are all `up' steps including the fact that the last step is a special one. Hence  
 by Lemma \ref{lem:HistoryHeightBound} it follows that the earlier explored region is no longer part of $H_{\tau_j}$. Further the $\sigma$-field ${\cal G}_{\tau_j}$ does not 
have any information about the perturbed points $\w + U_\w$ for 
$\w(d) \geq \widehat{g_{\tau_j}}(\u)(d)$ and $\w \notin N(\widehat{g_{\tau_j}}(\u))$.
Hence for any such $\w$, the r.v $U_\w$ is uniformly distributed over $[-1,+1]^d$ independent of the $\sigma$-field ${\cal G}_{\tau_j}$. More importantly, given ${\cal G}_{\tau_j}$ the 
point $g_{\tau_j}(\u)$ is uniformly 
distributed over the region $B^+(\widehat{g_{\tau_j}}(\u), \delta)$. 
This essentially tells us that given ${\cal G}_{\tau_j}$,
 when we restart the process from $\widehat{g_{\tau_j}}(\u)$, the history set
  $H_{\tau_j}$ can be taken to be empty. 
  
  It is important to observe that for the restarted process we have certain information about the 
perturbed point $\w^\prime = \w + U_\w$ for $\w \in N(\widehat{g_{\tau_j}}(\u))$, viz.,     
  conditional to ${\cal G}_{\tau_j}$, the point $\w^\prime$
is uniformly distributed over $B^+(\w, \delta)$. With this information 
about the points $\w + U_\w$ for all $\w \in N(\widehat{g_{\tau_j}}(\u))$
we restart the process from $\widehat{g_{\tau_j}}(\u)$.
 Since our model is translation invariant under translations by lattice points,
the distribution of $Y_j$'s are identical for $j \geq 2$. 
Suppose we start a DSF path $\{g_n(\mathbf{0}) : n \geq 0\}$  from the origin with 
 history set $H_0 = \emptyset$ and the information that for all $\w \in  N(\mathbf{0})$, the 
 associated perturbed point is uniformly distributed over $B^+(\w, \delta)$ and let $\tau_1$
 denote the first renewal step for such a process. Define the increment random variable $W$ as 
$$
W := \overline{\widehat{g_{\tau_1}}}(\mathbf{0}) - \overline{\mathbf{0}}.
$$
The earlier discussions give us that given ${\cal G}_{\tau_j}$, 
the increment r.v. $Y_{j+1}$ has the same distribution as $W$.
Fix $ m \geq 2 $ and Borel subsets $ B_2, \dotsc, B_m $ of $ \Z^{d-1}$. Let 
$ I_{j} ( B_j ) $ be the indicator random variable of the event $ \{ Y_{j} \in B_{j } \} $. Then, we have 
\begin{align*}
& \P (  Y_{j} \in B_{j } \text{ for } j =  2, \dotsc, m )
= \E (  \prod_{ j = 2}^{m } I_{j} ( B_{j} )) \\
&  = \E \Bigl( \E  \bigl( \prod_{ j = 2}^{m } I_{j} ( B_{j} ) 
\mid {\cal G}_{m-1} \bigl) \Bigr) = \E \Bigl( \prod_{ j = 2}^{m-1 } I_{j} ( B_{j} ) 
\E  \bigl(  I_{m} ( B_{m} ) \mid {\cal G}_{m-1} \bigl) \Bigr)\\
& = \P( W \in B_m)\E \Bigl( \prod_{ j = 2}^{m-1 } I_{j} ( B_{j} )  \Bigr) = \prod_{j=2}^m
\P( W \in B_j).
\end{align*}
 Note that $Y_1$ does not have the same distribution as that of $W$. Because though
 the initial history set is empty, we don't have any information about the perturbed points $\w + U_\w$
 for $\w \in N(\u)$. This proves that $\{ Y_j : j \geq 2\}$ gives a sequence of i.i.d.
 $\Z^{d-1}$ valued random increments.
\qed

Exactly the same argument as in the above proposition gives us the following corollary:  
\begin{corollary}
For $\{g_n(\u) : n \geq 0\}$, the set $\{ (Y_{j+1}, \tau_{j+1}-\tau_{j}) : j \geq 1\}$ gives
 a sequence of i.i.d. $(\Z^{d-1}\times \N)$ valued random vectors
whose distribution does not depend on the starting point $\u$.
\end{corollary}

Below we mention some properties of moments of the increments observed at renewal steps. 
For the single DSF path process $\{g_n(\u) ; n \geq 0\}$, 
the same argument of Lemma \ref{lemma:UnexploredCone} gives us that each step the increment $||g_{n+1}(\u) - g_{n}(\u)||_1$ is bounded (by $m_d$) and hence by Proposition \ref{prop:BetaExpTail} it follows that 
the increment random variable $|Y_{j+1}(i) - Y_j(i)| $ has moments of all orders
for $1 \leq i \leq d-1$.  Below we list out some further properties of these random variables which follows from the symmetry of the set $\cup_{\w \in N(\widehat{g_{\tau_j}}(\u))}B^+(\w , \delta)$
 considered at renewal step $\tau_j$.
  
\begin{corollary}
\label{cor:FirstCoOrdRW}
\begin{itemize}
\item[(i)] By reflection symmetry of the model, about any of the first $(d-1)$ coordinates, 
we have that the increment random variable $(Y_2 - Y_1)(j)$ is symmetric for each $1\leq j \leq d-1$.
Further the rotational symmetry of the model in the first $d-1$ coordinates implies that  
the marginal distributions $(Y_2 - Y_1)(j)$ are the same for $1\leq j \leq d-1$. In other words, 
$$
\P((Y_2 - Y_1)(j) = +m) = \P((Y_2 - Y_1)(l) = -m) \text{ for all } m \geq 1 \text{ and } 1 \leq j , l \leq d-1.  
$$ 
\item[(ii)] Consider $ j, l \in \{1, \cdots, d-1\}$ with $j \neq l$. By reflection symmetry along the 
$j$-th coordinate, with other coordinates being fixed, we observe that the joint distribution of 
$(Y_2 - Y_1)(j)(Y_2 - Y_1)(l)$ remains unchanged. This implies that $\E[(Y_2 - Y_1)(j)(Y_2 - Y_1)(l)] = 
\E[-(Y_2 - Y_1)(j)(Y_2 - Y_1)(l)]$ and hence we have $\E[(Y_2 - Y_1)(j)(Y_2 - Y_1)(l)] = 0$.
The same argument holds to obtain that $\E[((Y_2 - Y_1)(j))^{m_1}((Y_2 - Y_1)(l))^{m_2}] = 0$ for $m_1, m_2 \geq 1$ with at least one of them being odd. 
\end{itemize}
\end{corollary}

Hence, Corollary \ref{cor:FirstCoOrdRW} gives us that the diffusively 
scaled DSF path converges to the Brownian motion.

\subsection{Joint process at the renewal steps}
\label{sec:MarkovMartingale}

Next we explain the Markov properties observed at the renewal steps for the joint 
exploration process starting from $\u$ and $\v$. 
Set $Z_0 =  Z_0(\u, \v) = \overline{\u} - \overline{\v}$ and for $j \geq 1$ we define  
\begin{align}
\label{def:ZProcess}
Z_{j} = Z_j(\u, \v) := \overline{\widehat{g_{\tau_j}}}(\u) - \overline{\widehat{g_{\tau_j}}}(\v).
\end{align}
The following proposition describes the Markov property observed at these random steps. 
\begin{proposition} 
\label{prop:TwoPt_Markovwalk} 
\label{def:TwoPt_Markovwalk}
$\{ Z_{j+1} : j \geq 1\}$ is a $\Z^{d-1}$ valued time homogeneous 
Markov chain with absorbing state $\mathbf{0} \in \Z^{d-1}$. 
\end{proposition}
\noindent \textbf{Proof: } 
We first show that the process $\{\bigl ( \widehat{g_{\tau_j}}(\u), \widehat{g_{\tau_j}}(\v) \bigr) 
: j \geq 1\}$ is a $\Z^d$ valued Markov chain.  The argument is similar to that of the earlier proposition. Given 
${\cal G}_{\tau_j}$, the restarted process starting from the lattice points $\widehat{g_{\tau_j}}(\u)$ and $\widehat{g_{\tau_j}}(\v)$ with empty history set and information that for any $\w \in N(\widehat{g_{\tau_j}}(\u))\cup N(\widehat{g_{\tau_j}}(\v)$ the associated perturbed point $\w + U_\w$  is uniformly distributed over the region $B^+(\w, \delta)$.
 Also the $\sigma$-field ${\cal G}_{\tau_j}$ does not have any information 
about the random variables $U_\w$ for $\w \in \Z^d \setminus (N(\widehat{g_{\tau_j}}(\u))\cup N(\widehat{g_{\tau_j}}(\u))$
with $\w(d) \geq \widehat{g_{\tau_j}}(\u)(d)$. Hence using  a random mapping representation,
 we can show that  the process $\{\bigl ( \widehat{g_{\tau_j}}(\u), \widehat{g_{\tau_j}}(\v) \bigr) 
: j \geq 1\}$ is a $\Z^d$ valued Markov chain.

Further, by definition of a renewal step, for any $j \geq 1$ we have that 
$\widehat{g_{\tau_j}}(\u)(d) = \widehat{g_{\tau_j}}(\v)(d)$. 
Since our model is translation invariant under 
translations by lattice points only, given 
${\cal G}_{\tau_j}$ the conditional distribution of $\bigl ( \widehat{g_{\tau_{j+1}}}(\u), \widehat{g_{\tau_{j+1}}}(\v) \bigr)$ depends only on the difference  $\widehat{g_{\tau_j}}(\u) - \widehat{g_{\tau_j}}(\v)$. Finally, the proof follows from the observation that
 $(\widehat{g_{\tau_j}}(\u) - \widehat{g_{\tau_j}}(\v))(d) = 0$ for all $j \geq 1$.
 \qed

In the next section we explore properties of 
the process $\{Z_j : j \geq 1\}$ in more detail which will allow us 
to conclude that the coalescing time of any two DSF paths is an a.s. finite r.v.

\section{Proof of Theorem \ref{thm:TreeDSF}}
\label{sec:Trees}
In this section we prove Theorem \ref{thm:TreeDSF}. For $d=2,3$, we need to 
show that for $\x, \y \in V$, the DSF paths $\pi^\x$ and $\pi^\y$ coincide eventually, i.e., 
there exists $t_0 < \infty$ such that  $\pi^\x(s) = \pi^\y(s)$ for all $s \geq t_0$.
Following \cite{RSS16}, we first argue that it is enough to show that
\begin{align}
\label{eq:TreeSameLevel}
\pi^{\u} \text{ and }\pi^{\v} \text{ coincide eventually for all }\u, \v \in \Z^d
\text{ with }\u(d) = \v(d).
\end{align}
This follows from the simple observation that 
\begin{align*}
\P \bigl [  \bigcap_{\x , \y \in V} \{& \text{ there exist }\u, \v \in \Z^d \text{ with }\u(d) = \v(d)
\text{ such that }\\
& h(\u) = h^{l_1}(\x), h(\v)= h^{l_2}(\x) \text{ for some }l_1, l_2\}\bigr] = 1.
\end{align*}
Since the DSF paths starting from  $\u$ and $\v$ reaches special up step in finitely many steps a.s.,
the above equality follows. Now to  show that for $\u, \v \in \Z^d$ with $\u(d) = \v(d)$,
we consider the $\Z^{d-1}$ valued Markov chain $Z_j (\u, \v) : j \geq 1$ and 
show that this Markov chain
hits  $\mathbf{0} \in \Z^{d-1}$ almost surely.

\subsection{The case $d=2$}
\label{sec:d2}

In this section we prove Theorem \ref{thm:TreeDSF} for $d=2$ by showing that for $\u, \v \in \Z^d$
with $\u(d) = \v(d)$, the DSF paths $\pi^{\u}, \pi^{\v}$ coincide eventually with 
probability $1$. To do that, in the following section
 we prove a stronger result by showing that the tail of the 
coalescing time decays as in case of coalescing time of two independent random walks (with finite variances). This estimate is crucial to show convergence to the Brownian web. 
Because of (\ref{eq:TreeSameLevel}), this completes the proof of Theorem \ref{thm:TreeDSF} for $d=2$. 
In this context it is useful to mention that we find it is hard to show that for $d=2$, 
the Markov chain $\{Z_j(\u, \v) : j \geq 2 \}$ is a martingale and hence could not apply 
the technique prescribed by Coletti et. al. in \cite{CFD09} to achieve coalescing time tail decay.

\subsubsection{Coalescing time tail estimate for DSF paths for $d=2$}
\label{sec:CoalescingTimeTail}
Let $\u,\v \in \Z^2$ be chosen such that $\u(1)<\v(1)$ and $\u(2)=\v(2)$. 
Without loss of generality we can take $\u(2) = 0$.
The coalescence time of the two DSF paths $\pi^{\u}$ and $\pi^{\v}$ starting from $\u$ and $\v$,
 is given by:
\begin{equation}
\label{CoalTime:u1u2}
T(\u,\v) := \inf \{t\geq \u(2) : \, \pi^{\u}(t) = \pi^{\v}(t) \}.
\end{equation}
We prove the following proposition on tail decay of $T(\u, \v)$.

\begin{prop}
\label{prop:DSF_CoalescingTimeTail}
For $\u,\v \in \Z^2$ with $\u(1)<\v(1)$ and $\u(2)=\v(2)$, 
there exists a constant $C_0 > 0$, which does not depend on $\u, \v$ such that, 
for any $t > 0$, 
$$
\P(T(\u,\v) > t) \leq \frac{C_0(\v(1)- \u(1)) }{\sqrt{t}}.
$$
\end{prop}
We first observe that the above proposition tells us that $T(\u, \v)$ is almost surely finite
and because of (\ref{eq:TreeSameLevel}), it further proves that the $2$-dimensional 
perturbed DSF is connected with probability $1$. 

In order to get the required estimate, we will
apply a robust technique that was developed in \cite{CSST19}. 
We first quote Corollary 5.6 from \cite{CSST19} which essentially states 
that for a process which behaves like a symmetric random walk 
far from the origin and satisfy certain moment bounds, 
similar tail estimate holds. Symmetric random walk like behaviour far away from the 
origin is reflected through the property that increment of the 
process considered below, has null expectation on an event with high probability.

 \begin{corol}
\label{corol:LaplaceCoalescingTimeTailNew3}
Let $\{Y_t : t \geq 0\}$ be a $\{{\cal G}_t : t \geq 0\}$ adapted stochastic process taking values in $\R_+$. Let $\nu^Y := \inf\{t \geq 1 : Y_t = 0\}$ be the first hitting time to $0$. Suppose for any $t \geq 0$ there exist positive constants $ M_0, C_0, C_1, C_2, C_3 $ such that:
\begin{itemize}
\item[(i)] There exists an event $F_t$ such that, on the event $\{Y_t> M_0\}$, we have $\P(F^c_t \mid {\cal G}_t) \leq C_0 / Y^3_t$ and
\begin{equation*}
\E \big[ (Y_{t+1}-Y_t) \mathbf{1}_{F_t} \mid {\cal G}_t\big] = 0 ~.
\end{equation*}
\item[(ii)] For any $t\geq 0$, on the event $\{Y_t \leq M_0\}$,
\begin{equation*}
\E \big[ (Y_{t+1}-Y_t)  \mid {\cal G}_t\big] \leq C_1 ~.
\end{equation*}
\item[(iii)] For any $t\geq 0$ and $ m > 0 $, there exists $c_m > 0 $ such that, on the event $\{Y_t \in (0, m]\}$,
\begin{equation*}
\P \big( Y_{t+1} = 0 \mid {\cal G}_t\big) \geq c_m ~.
\end{equation*}
\item[(iv)] For any $t \geq 0$, on the event $\{Y_t> M_0\}$, we have
\begin{equation*}
\E \bigl[ (Y_{t+1} - Y_t)^2 \mid {\cal  G}_t \bigr] \geq C_2 \; \text{ and } \; \E \bigl[ |Y_{t+1} - Y_t|^3 \mid {\cal  G}_t \bigr] \leq C_3 ~.
\end{equation*}
\end{itemize}
Then, $\nu^Y<\infty $ almost surely. Further, there exist positive constants $C_4, C_5$ such that for any $y>0$ and any integer $n$,
\begin{equation*}
\P( \nu^Y > n \mid Y_0 = y) \leq \frac{C_4 + C_5 y}{\sqrt{n}} ~.
\end{equation*}
\end{corol}

In order to apply the above corollary for DSF paths, 
we recall the process $\{Z_j : j \geq 2\}$ (see Definition (\ref{def:ZProcess})).
It useful to observe that the DSF paths are non-crossing a.s. which means that 
we have 
$$
\pi^\v (t) \geq \pi^\u (t) \text{ for all }t \geq \v(2). 
$$
Hence for the given choice of $\u, \v $ with $\u(1) \leq \v(1)$ the process $\{Z_j : j \geq 2\}$ 
becomes non-negative with the only absorbing state at $0$. 
We first obtain a suitable estimate on the number of renewal steps required 
by the process $\{Z_j : j \geq 2 \}$ to hit $0$ denoted by $\nu=\nu(\u,\v)$ :
\begin{equation}
\label{def:nu(u1,u2)}
\nu := \inf\{ j \geq 2 : Z_j = 0\} .
\end{equation}

The following corollary, states that the non-negative process $\{ Z_j : j \geq 2\}$
satisfies the conditions of Corollary \ref{corol:LaplaceCoalescingTimeTailNew3} and therefore 
have the required tail estimate in terms of number of renewal steps. This indeed
completes the proof of Theorem \ref{thm:TreeDSF} for $d=2$.
In order to complete the proof of Proposition \ref{prop:DSF_CoalescingTimeTail}, we further need 
to show that the coalescing time $T(\u, \v)$ decays in a similar manner.  
We note that the argument of Lemma \ref{lem:HistoryHeightBound} gives us that at each step,
increment of each path is bounded by $m_d$ (in fact by $m_d - 4$).   

We define the width random variable for $j \geq 0$
\begin{align}
\label{def:Widthrv}
W_{j+1} := 2m_d(\tau_{j+1} -\tau_j),
\end{align} 
and it is clear that given ${\cal G}_{\tau_j}$, the rectangles, $\widehat{g_{\tau_j}}(\u) + [-W_{j+1}, W_{j+1}] \times [0, W_{j+1}]$ and $\widehat{g_{\tau_j}}(\v) + [-W_{j+1}, W_{j+1}] \times [0, W_{j+1}]$,
centred at $\widehat{g_{\tau_j}}(\u)$ and $\widehat{g_{\tau_j}}(\v)$ respectively contain the region explored between $j$-th renewal and $j+1$-th renewal. Since each increment is bounded, Proposition \ref{prop:BetaExpTail}
ensures that for any $j\geq 1$ the random variable $W_j$ has exponentially decaying tail. Formally
there exist $C_0, C_1 > 0$ depending only on the number of DSF trajectories, $d$, and $\delta$ such that 
for all $n \geq 1$
\begin{equation}   
\label{eq:ExpDecayWidth}
\P(W_{j+1} \geq n \mid {\cal G}_{\tau_j}) \leq C_0 \exp{(- C_1 n)}.
\end{equation}
The next result says that, far from the origin, the Markov chain $\{Z_j : j \geq 2\}$
 behaves like a symmetric random walk satisfying certain moment bounds.

\begin{corol}
\label{corol:ZprocessRwalkProperties}
For the given choice of $\u, \v \in \Z^2$, 
there exist positive constants $M_0, C_0, C_1, C_2$ and $C_3$ such that:
\begin{itemize}
\item[(i)] For any $ j \geq 1$, let us set the event $E_j := \{W_{j +1} < Z_j/4 \}$.
 Then, on the event $\{Z_j \geq M_0\}$, we have $\P(E^c_j \mid {\cal G}_{\tau_j} ) 
 \leq C_3/(Z_j)^3$ and
\begin{equation*}
\E \big[ (Z_{j +1} - Z_{j}) \mathbf{1}_{E_j} \mid {\cal G}_{\tau_j} \big] = 0 .
\end{equation*}
\item[(ii)] For any $j \geq 1$, on the event $\{Z_j \leq M_0\}$,
\begin{equation*}
\E \big[ (Z_{j +1} - Z_{j})  \mid {\cal G}_{\tau_j} \big] \leq C_0 .
\end{equation*}
\item[(iii)] For any $j \geq 1$ and $m>0$, 
there exists $c_m>0$ such that, on the event $\{Z_j \in (0, m]\}$,
\begin{equation*}
\P \big( Z_{j+1} = 0 \mid {\cal G}_{\tau_j} \big) \geq c_m .
\end{equation*}
\item[(iv)] For any $j \geq 1$, on the event $\{Z_j > M_0\}$,
\begin{equation*}
\E \bigl[ ( Z_{j + 1} - Z_{j} )^2 \mid {\cal  G}_{\tau_j} \bigr] \geq C_1 \; \text{ and } \;
\E \bigl[ |Z_{j+1} - Z_{j} |^3 \mid {\cal  G}_{\tau_j}  \bigr] \leq C_2 .
\end{equation*}
\end{itemize}
\end{corol}

% We comment that the choice of $M_0 > 2$ ensures that the two semi-balls 
% $B^+(g^{\uparrow}_{\beta_\ell} (\u_1), 1)$ and $B^+(g^{\uparrow}_{\beta_\ell} 
% (\u_2), 1)$ are disjoint indeed and hence the tricky situation depicted 
% in Figure \ref{fig:Aj} can not happen.

\begin{proof} We first consider  part (i). We already commented that the 
region explored by the DSF paths starting from $\u$ and $\v$ in between 
the $j$-th and ${j + 1}$-th  (joint) renewal step is enclosed within the rectangles 
$(\widehat{g_{\tau_j}}(\u) + [- W_{j+1}, W_{j+1}]\times [0,  W_{j+1}]
) $ and $(\widehat{g_{\tau_j}}(\v) + [- W_{j+1} ,  W_{j+1}]\times [0,  W_{j+1}]) $.
Now take $M_0$ sufficiently large and let us consider the trajectories of $\pi^{\u}$ 
and $\pi^{\v}$ in between the $j$-th and $( j + 1)$-th (joint) renewal steps.
The choice of $M_0$ ensures that the probability on the event $E_j^c$ is sufficiently small. 
We recall that the original paths agree with the restarted paths
starting from $\widehat{g_{\tau_j}}(\u)$ and $\widehat{g_{\tau_j}}(\v)$ 
and the restarted paths use perturbed points only in $\mathbb{H}^+(\widehat{ g_{\tau_j}}(\u)(2)$. 
On the event $E_j$, these two  rectangles, $(\widehat{g_{\tau_j}}(\u) + [- W_{j+1}, W_{j+1}]\times [0,  W_{j+1}]
) $ and $(\widehat{g_{\tau_j}}(\v) + [- W_{j+1} ,  W_{j+1}]\times [0,  W_{j+1}]) $, are disjoint. 
This allows us to resample 
the point process within these two rectangles satisfying the renewal condition so that 
the resampled point process still satisfies the (joint) renewal condition and 
the joint distribution of the restarted trajectories does not change. 

We construct a new point process in the following way:
\begin{itemize}
\item[(1)] The realizations of the perturbed points in the rectangles 
$$R_1 := \widehat{g_{\tau_j}}(\u) + [- M_0 / 4, M_0 / 4] \times 
[ 0, M_0 / 4] \text{ and  }R_2 := \widehat{g_{\tau_j}}(\v) + [ - M_0 / 4, M_0 / 4] 
\times [ 0, M_0 / 4] $$ are interchanged. 
\item[(2)] The realization of the perturbed point process
 outside these two rectangles is kept as it is.
\end{itemize}
Since $Z_j = | \widehat{g_{\tau_j}}(\u)(1) - \widehat{g_{\tau_j}}(\v)(1)| \geq M_0$, a sufficiently large quantity, the new point process that we got through interchange of points in those rectangular regions, 
also satisfy renewal condition. Further for the ``new'' restarted paths, constructed 
using the new point process, the number of steps until the next 
renewal step and the size of the corresponding renewal block remain exactly the same. 
In fact, the increments of the two DSF paths between the $j$-th and 
$(j + 1)$-th renewal steps have been interchanged. This means that 
the increment $I_{j + 1} = Z_{j+1} - Z_j$ has become $- I_{j + 1}$. 
This completes the proof of part (1).

Part (ii) follows readily from the fact that
\begin{align*}
\E[(Z_{j + 1} - Z_j) \mid {\cal G}_{\tau_j} ] \leq \E[|Z_{j + 1} - Z_j| \mid {\cal G}_{\tau_j} ]
 \leq \E(W_{j+1}) \mid {\cal G}_{\tau_j}) < \infty,
\end{align*}
where $T$ is a non-negative random variable with exponential tail such that for all $j \geq 0$,
the conditional distribution of $W_{j + 1} \mid {\cal G}_{\tau_j}$
is uniformly stochastically dominated by $T$.

For part (iii) we recall the fact that the $\sigma$-field ${\cal G}_{\tau_j}$, using the unexplored lattice points we can create a suitable configuration with it's probability uniformly 
bounded away from zero such that 
that the  conditional probability $\P(Z_{j + 1} = 0 \mid {\cal G}_{\tau_j})$ is strictly positive .

For part (iv) we observe that 
$$
 \E[|Z_{j + 1} - Z_j|^3 \mid {\cal G}_{\tau_j} ] \leq \E[(W_{j+1})^3 \mid {\cal G}_{\tau_j}] 
 \leq \E[T^3] < \infty,
$$
where $T$ is a uniformly dominating stochastic r.v.  
This completes the proof. 
\end{proof}

 The above corollary shows that the process $\{Z_j : j \geq 2\}$ satisfies the 
 conditions of Corollary \ref{corol:LaplaceCoalescingTimeTailNew3}, and hence 
 we have the following tail estimate on $\nu$.
For all integer $n \in \N$, there exists a positive constant $C_0$ 
which does not depend on $|\u(1) - \v(1)|$ such that,
\begin{equation}
\label{eq:CoalStepEst}
\P(\nu > n | Z_0 = (\v(1) - \u(1))) \leq \frac{C_0(\v(1)-\u(1))}{\sqrt{n}}.
\end{equation}

Since, the number of steps between two consecutive renewal steps decay exponentially, as shown in 
(\ref{eq:ExpDecayWidth}), using the above lemma we prove Proposition \ref{prop:DSF_CoalescingTimeTail}.

\noindent{\textbf{ Proof of Proposition \ref{prop:DSF_CoalescingTimeTail}:}} 
For our DSF model,  it is clear that 
$T(\u, \u) \leq \tau_\nu $.
Hence we have that,
\begin{equation}
\label{CoalescingTimetail:Ineq1}
\P(T(\u, \v) > t) \leq \P \Big( \sum_{j = 1}^{\nu} 
( \tau_j - \tau_{j-1}) > t \Big)
\leq \P \Big( \sum_{j = 1}^{\lfloor ct \rfloor} (\tau_j - \tau_{j-1}) > t \Big)
 + \P \big( \nu > \lfloor ct\rfloor \big).
\end{equation}
 Recall that the r.v.'s $(\tau_j - \tau_{j-1})|{\cal G}_{\tau_{j-1}}$ 
 are uniformly stochastically dominated with a random variable $T$ 
 with exponential tail. 
 Hence, it is not difficult to obtain
$$
\P \Big( \sum_{j = 1}^{\lfloor c't \rfloor} (\tau_{j}-\tau_{j-1}) > t \Big) \leq C_0 e^{-C_1 t}
$$
for a constant $c'>0$ small enough. To sum up, we have:
$$
\P(T(\u, \v) > t) \leq \P(\nu > \lfloor c't\rfloor) + C_0 e^{-C_1 t}
$$
from which we conclude using (\ref{eq:CoalStepEst}).
\qed
      
\subsection{The case $d=3$}
\label{sec:3d}

In this section we prove Theorem \ref{thm:TreeDSF} for $d=3$.
To do that we first describe simultaneous renewal of two independent DSF paths and 
we use it to approximate the joint distribution of DSF paths at (joint) renewal steps when the
paths are far apart. 
This section is motivated from \cite{RSS16} and we only give a brief sketch here. For details we refer the reader to Sections 3.2 and 3.3 of \cite{RSS16}. 

In order to construct two independent DSF paths, we start with two independent collections 
$\{U^a_\w : \w \in \Z^3\}$ and $\{U^b_\w : \w \in \Z^3\}$ of i.i.d. random variables where each 
random variable is uniformly distributed over $B(\w) = \w + [-1,1]^d$. These collections give 
two independent copies of perturbed point process given as $V^a := \{\w + U^a_\w : \w \in \Z^3\}$ 
and $V^b := \{\w + U^b_\w : \w \in \Z^3\}$. The process $\{g^{\text{Ind}}_n(\u) : n \geq 0\}$, starting 
from $\u$, uses the perturbed points in $V^a$ and the process $\{g^{\text{Ind}}_n(\v) : n \geq 0\}$, starting from $\v$, uses the perturbed points in $V^b$.
For $j \geq 0$, let $\tau^{\text{Ind}}_j(\u)$ and $\tau^{\text{Ind}}_j(\v)$ denote the $j$-th 
marginal renewal steps for the marginal processes  $\{g^{\text{Ind}}_n(\u) : n \geq 0\}$ and 
 $\{g^{\text{Ind}}_n(\v) : n \geq 0\}$ respectively.
It follows that the two sequences, $\{(\tau^{\text{Ind}}_{j+1}(\u) - \tau^{\text{Ind}}_{j}(\u)) 
: j \geq 1\}$ and $\{(\tau^{\text{Ind}}_{j+1}(\v) - \tau^{\text{Ind}}_{j}(\v)) 
: j \geq 1\}$ form independent collections of i.i.d. renewal times with exponentially decaying tails. 

Now we define the sequence of simultaneous renewal steps for the two independent paths.
We set $J_0 , J^\prime_0:= 0$.  For $m \in \N$ let 
\begin{align*}
J_{m+1} & := \inf\{ j > J_m :  \tau^{\text{Ind}}_{j}(\u) =  \tau^{\text{Ind}}_{j^\prime}(\v)
\text{ for some } j^\prime >J^\prime_m\} \\
J^\prime_{m+1} & := \inf\{ j^\prime > J^\prime_m :  \tau^{\text{Ind}}_{j^\prime}(\v) =  
\tau^{\text{Ind}}_{j}(\v) \text{ for some }j > J_m\}.
\end{align*}
It follows that for all $m \geq 0$, the r.v.'s $J_m$ and $J^\prime_m$ are finite a.s.
Then for $m \geq 0$
$$
\tau^{\text{Ind}}v_m(\u, \v) := \tau^{\text{Ind}}_{J_m}(\u) = 
\tau^{\text{Ind}}_{J^\prime_m}(\v), 
$$
gives the sequence simultaneous renewals for two independent paths.   

It follows that $\{\tau^{\text{Ind}}_{m+1}(\u, \v) - \tau^{\text{Ind}}_m(\u, \v) : m \geq 1\}$
forms an i.i.d. collection of random variables. Further by Proposition 3.3 of \cite{RSS16} we have 
that there exists positive constants $C_0, C_1$ such that for all $n \in \N$ and for all $m \geq 0$
 we have
\begin{equation}
\label{eq:IndRenewalExptailDecay}
\P(\tau^{\text{Ind}}_{m+1}(\u, \v) - \tau^{\text{Ind}}_m(\u, \v) \geq n \mid {\cal G}^\text{Ind}_m(\u, \v)) \leq C_0 \exp{(-C_1 n)},
\end{equation} 
where ${\cal G}^\text{Ind}_m(\u, \v) := \sigma \bigl({\cal G}_{\tau^{\text{Ind}}_m(\u)}(\u), {\cal G}_{\tau^{\text{Ind}}_m(\v)}(\v) \bigr )$.
Next we observe both the independent processes from the restarted points 
at the simultaneous renewal steps, i.e., 
$\{\u^\text{Ind}_{m+1} :=  \widehat{g^{\text{Ind}}_{\tau^{\text{Ind}}_{m+1}(\u, \v)}}(\u) : m \geq 0\}$ and 
$\{\v^\text{Ind}_{m+1} := \widehat{g^{\text{Ind}}_{\tau^{\text{Ind}}_{m+1}(\u, \v)}}(\v): m \geq 0\}$.
We observe that both the lattice points have the same $d$-th coordinate. 

Now it follows that for $d=3$, the processes 
$\{\overline{\u^\text{Ind}_{m+1}} : m \geq 0\}$ and 
$\{\overline{\v^\text{Ind}_{m+1}} : m \geq 0\}$  
form $\Z^2$ valued independent random walks with i.i.d. increments
\begin{align*}
\psi^{\u}_{m+1} & := \overline{\u^\text{Ind}_{m+1}} - \overline{\u^\text{Ind}_{m}}\text{ and }\\
\psi^{\v}_{m+1} & := \overline{\v^\text{Ind}_{m+1}} - \overline{\v^\text{Ind}_{m}},
\end{align*}
respectively. 
Clearly, both $\psi^{\u}_{m+1}$ and $\psi^{\v}_{m+1}$ have moments of all orders. 
The same argument as in Corollary \ref{cor:FirstCoOrdRW} gives us the following properties
 of moments of these $\Z^2$-valued random vectors. We observe that the independence of the two families ensure that 
we may apply rotation operator independently for both the families $\{U^a_\w : \w \in \Z^d\}$ and $\{U^b_\w : \w \in \Z^d\}$ independently. This allows us to derive the following important properties of the increment distributions: 
\begin{itemize}
\item[(i)] We apply the rotation operator independently to \textit{one} family and the same argument as in Corollary \ref{cor:FirstCoOrdRW} give us that the marginal distribution of each coordinate of $\psi^{(u)}_2$ as well as $\psi^{(v)}_2$
is symmetric. Further, they are all the same. More precisely, 
$$
\P(\psi^{\u}_2(l) = +r) = \P(\psi^{\u}_2(l) = -r) = \P(\psi^{\v}_2(j) = +r) = \P(\psi^{\v}_2(j) = -r)
\text{ for } 1 \leq i,j \leq 2. 
$$  
\item[(ii)] The same technique as in the proof of Proposition 3.4 of \cite{RSS16} gives us that $\E[(\psi^{\u}_2(l))^{m_1}(\psi^{\v}_2(l))^{m_2}]$ depends only on $m_1$ and $m_2$ and becomes zero if at least one of $m_1$ and $m_2$ is odd.
\end{itemize}

Next we describe a coupling procedure which will allow us to compare 
the independent DSF paths at simultaneous renewal steps with joint DSF paths 
at (joint) renewal steps  when the paths are far apart. We consider another 
collection $\{U^c_\w : \w \in \Z^3\}$ of i.i.d. random variables independent of the 
collections $\{U^a_\w : \w \in \Z^3\}$ and $\{U^a_\w : \w \in \Z^3\}$, such that $U^c_\w$ is uniformly 
distributed over $[-1, 1]^d$. We use these collections to construct a perturbed point process 
to construct the joint paths starting from $\u$ and $\v$. 

Set $d_{\text{min}} := ||\overline{\u} - \overline{\v}||_1 $. 
Fix $r < d_{\text{min}}/ 3$ and construct the new collection of i.i.d. random variables $\{\tilde{U}_\w : \w \in \Z^3\}$ as follows:
\begin{align*}
\tilde{U}_\w := 
\begin{cases}
U^a_\w & \text{ if } ||\overline{\u} - \overline{\w}||_1 < r\\
U^b_\w & \text{ if } ||\overline{\v} - \overline{\w}||_1 < r\\
U^c_\w & \text{ otherwise.}
\end{cases}
\end{align*}
Using the collection  $\{\tilde{U}_\w : \w \in \Z^3\}$ we obtain a new perturbed point process 
$\tilde{V} := \{\v + \tilde{U}_\w : \v \in \Z^3\}$ (which has the same distribution as $V$)
and use it to construct the joint process $\{(g_n(\u), g_n(\v)): n \geq 0\}$ starting from $\u$
and $\v$ till their first joint renewal step $\tau_{1}(\u, \v)$.

We observe that on the event $A^{\text{Good}}_r := \{ W_1  < r \}$, where $W_1$ is defined as in (\ref{def:Widthrv}), we have $\tau_{1}(\u, \v) = \tau^{\text{Ind}}_{1}(\u, \v)$.
 Not only that, the trajectory of the independent DSF paths
  coincide with that of the joint paths till $\tau_{1}(\u, \v) = 
  \tau^{\text{Ind}}_{1}(\u, \v)$. In particular, we have, 
\begin{align*}
\P[(\overline{\widehat{g_{\tau_1(\u, \v))}}}(\u), \overline{\widehat{g_{\tau_1(\u, \v))}}}(\v)) = 
(\overline{\u} + \psi^{\u}_1, \overline{\v} + \psi^{\v}_1) ]\geq 
\P(A^{\text{Good}}_r) \geq 1 - C_0 \exp{(-C_1 r)},
\end{align*}
where the last inequality follows from (\ref{eq:ExpDecayWidth}).
The Markov property allows us to extend this coupling for each subsequent renewal steps 
where for the $j$-th renewal, the value of $d_{min}$ is updated as 
$||\overline{\widehat{g_{\tau_1(\u, \v))}}(\u)}  -  
\overline{\widehat{g_{\tau_1(\u, \v))}}(\v)}||_1$.   
    
We follow the method used in \cite{RSS16} to prove that the perturbed DSF is connected a.s. 
for $d=3$. We recall that the auxiliary process obtained from restarted 
paths at (joint) renewal steps denoted by $\{Z_j (\u, \v) : j \geq 1 \}$ (see (\ref{def:ZProcess})) forms a $\Z^2$ valued Markov chain. 
Because of (\ref{eq:TreeSameLevel}), it suffices to show that this Markov chain hits $(0,0)$
a.s. and to do that as in  \cite{RSS16}, we apply Foster's criterion (see \cite{A03}, Proposition 
5.3 of Chapter 1). We change the transition probability of $Z_j$ from the state $(0,0)$
in any reasonable way so that this state no longer remains an absorbing state and the resulting 
Markov chain becomes irreducible. With slight abuse of notation, we continue to denote the modified
chain by $\{Z_j : j \geq 2\}$ and show that this chain is recurrent.

Next, exactly the argument as in Proposition 4.2  of \cite{RSS16} gives us that, when the
paths are far apart, $Z_j(\u, \v)$ is well approximated by the independent process 
in expectation. More precisely for $m \geq 2$ we have
\begin{align}
\label{eq:JtRenewalApproxIndrenewal}
\bigl |\E[(||Z_{j+1}||^2_2 - ||\v||^2_2)^m \mid Z_{j} = \v ] - 
\E[(||((0,0) + \psi^{(0,0)}_1) - (\v + \psi^{\v}_1)||^2_2 - ||\v||^2_2)^m ]
  \bigr | \leq C^{(m)}_0\exp{(-C^{(m)}_1 ||\v||_2)},
\end{align}    
where $C^{(m)}, C^{(m)}_1$ are positive constants depending on $m$.

Now, we consider $f: \Z^2 \mapsto [0,\infty)$ defined by $f(\v) = \sqrt{\log(1 + ||\v||^2_2)}$. 
Clearly $f(\v) \to \infty$ as $||\v|| \to \infty$. Using Taylor's expansion of the 
function $h(t) = \sqrt{\log(1 + t)}$ and observing that the fourth derivative of $h$ is always negative,
we have
\begin{align}
\label{eq:FosterCalc1}
& \E[f(Z_1) - f(Z_0) \mid Z_0 = \v]\nonumber\\
& = \E[h(||Z_1||^2_2) - h(||Z_0||^2_2) \mid Z_0 = \v]\nonumber\\
& \leq \sum^{3}_{m=1}\frac{h^{(m)}(||\v||^2_2)}{m\!} \E[(||Z_1||^2_2 - ||\v||^2_2)^m \mid Z_0 = \v]\nonumber\\
& \leq \sum^{3}_{m=1}\frac{h^{(m)}(||\v||^2_2)}{m\!} \Bigl \{
\E \Bigl [ (||((0,0) + \psi^{(0,0)}_1) - (\v + \psi^{\v}_1)||^2_2 - ||\v||^2_2)^m \Bigr ] + 
C^{(m)}_0\exp{(-C^{(m)}_1 ||\v||_2)}\Bigr \},
\end{align} 
where $h^{(m)}$ represents $m$-th derivative of $h$ and the last inequality follows from (\ref{eq:JtRenewalApproxIndrenewal}).

Now in order to calculate (\ref{eq:FosterCalc1}) we use properties of $\psi = \psi^{(0,0)}_1$ that we observed earlier and obtain 
\begin{align*}
\E \Bigl [ (||((0,0) + \psi^{(0,0)}_1) - (\v + \psi^{\v}_1)||^2_2 - ||\v||^2_2) \Bigr ] & = 4\E[\psi^2] = \alpha (\text{say});\\
\E \Bigl [ (||((0,0) + \psi^{(0,0)}_1) - (\v + \psi^{\v}_1)||^2_2 - ||\v||^2_2)^2 \Bigr ] & \geq 8\E[\psi^2]||\v||^2_2 = 2\alpha||\v||^2_2 ;\\
\E \Bigl [ (||((0,0) + \psi^{(0,0)}_1) - (\v + \psi^{\v}_1)||^2_2 - ||\v||^2_2) \Bigr ] & = O(||\v||^2_2).
\end{align*}
Putting the above values of moments and plugging the expressions for $h^{(m)}$ in (\ref{eq:FosterCalc1})
we have that the first sum  in (\ref{eq:FosterCalc1}) is bounded by 
$\alpha||\v||^2_2/ [8 (1 + ||\v||^2_2)^2 (\log(1 + ||\v||^2_2))^3/2]$
whereas the second sum is bounded by $C_2\exp{(-C_3||\v||_2)}$ for a proper choice of $C_2, C_3 > 0$.
Hence we obtain that 
$$
\E[f(Z_{j + 1}) - f(Z_{j}) \mid Z_{j} = \v] < 0,
$$
for all $||\v||_2$ large enough. 
This implies that the modified Markov chain $\{Z_j : j \geq 1\}$ is recurrent and 
completes the proof of Theorem \ref{thm:TreeDSF}
for $d=3$.

\section{Convergence to the Brownian web}
\label{sec:cvBW}

This section is devoted to the proof of Theorem \ref{thm:BW},
 i.e,  that for $d=2$ the scaled perturbed DSF converges to the 
Brownian web (BW). In fact we prove a stronger version of the theorem in the sense that
we construct a dual process and show that under diffusive scaling the original process together 
with the dual process jointly converge to the BW and its dual. 
Towards this we will apply a robust technique that was developed in \cite{CSST19}
to study convergence to the BW for non-crossing path models.

We first start with introducing the dual BW $\widehat{{\cal W}}$.
As in case of forward paths, one can consider a similar metric space of collection of backward paths 
denoted by $(\widehat{\Pi}, d_{\widehat{\Pi}})$. 
The notation $(\widehat{{\mathcal H}}, d_{\widehat{{\mathcal H}}})$
denotes the corresponding Polish space of compact collections of backward paths with the induced 
Hausdorff metric. The BW and its dual denoted 
by $({\mathcal W},\widehat{{\mathcal W}})$ is a $({\mathcal H}\times \widehat{{\mathcal H}}, {\mathcal B}_{{\mathcal H}}\times {\mathcal B}_{\widehat{{\mathcal H}}})$-valued random variable such that:
\begin{itemize}
\item[$(i)$] $\widehat{{\mathcal W}}$ is distributed as $-{\mathcal W}$, the BW rotated $180^0$ about the origin;
\item[$(ii)$] ${\mathcal W}$ and $\widehat{{\mathcal W}}$ uniquely determine each other
in the sense that the paths of ${\cal W}$ a.s. do not cross with (backward) paths in 
$\widehat{{\cal W}}$. See  \cite[Theorem 2.4]{SSS17}. The interaction between the paths in ${\mathcal W}$ and $\widehat{\mathcal W}$ is that of Skorohod reflection (see \cite{STW00}).
\end{itemize}

Now it is time to specify a dual graph $\widehat{G}$ to the DSF $G$. 
The construction of the dual graph is not unique and we follow the construction of the 
dual graph from \cite{CSST19} (see Figure \ref{fig:dual}). 
We start by constructing a dual vertex set $\widehat{V}$.
 For any $(x,t) \in \R^2$,
 let $(x,t)_r \in V$ be the unique perturbed  point such that
\begin{itemize}
\item[$\bullet$] $(x,t)_r(2)<t$, $h((x,t)_r)(2)\geq t$ and $\pi^{(x,t)_r}(t)>x$ where $\pi^{(x,t)_r}$ denotes the path in ${\cal X}$ starting from $(x,t)_r$;
\item[$\bullet$] there is no path $\pi\in {\cal X}$ with $\sigma_\pi<t$ and $\pi(t)\in (x,\pi^{(x,t)_r}(t))$.
\end{itemize}
Hence, $\pi^{(x,t)_r}$ is the nearest path in ${\cal X}$ to the right of $(x,t)$ starting strictly before time $t$. It is useful to observe that $\pi^{(x,t)_r}$ is defined for any $(x,t)\in\R^2$. Similarly, $\pi^{(x,t)_l}$ denotes the nearest path to the left of $(x,t)$ which starts strictly before time $t$. Now, for each $(x,t)\in V$, the nearest left and right dual vertices are respectively defined as
$$
\widehat{r}_{(x,t)} := \bigl(( x + \pi^{(x,t)_r}(t))/2, t \bigr) \; \text{ and } \; \widehat{l}_{(x,t)} := \bigl(( x + \pi^{(x,t)_l}(t))/2, t \bigr) ~.
$$
Then, the dual vertex set $\widehat{V}$ is given by $\widehat{V}:=\{\widehat{r}_{(x,t)}, \widehat{l}_{(x,t)} : (x,t) \in V \}$.

Next, we need to define the dual ancestor $\widehat{h}(y,s)$ of $(y,s)\in\widehat{V}$ as the unique vertex in $\widehat{V}$ given by
\begin{equation*}
\widehat{h}(y,s) :=
\begin{cases}
\widehat{l}_{(y,s)_r} & \text{ if }(y,s)_r(2) >  (y,s)_l(2)  \\
\widehat{r}_{(y,s)_l} & \text{ otherwise.}
\end{cases}
\end{equation*}
The dual edge set $\widehat{E}$ is then given by $\widehat{E}:=\{\langle (y,s), \widehat{h}(y,s) \rangle : (y,s)\in \widehat{V}\}$. Clearly, each dual vertex has exactly one outgoing edge which goes in the downward direction. Hence, the dual graph $\widehat{G}:= (\widehat{V},\widehat{E})$ does not contain any cycle or loop. This forest $\widehat{G}$ is entirely determined from $G$ without any extra randomness. 

\begin{figure}[!ht]
\begin{center}
\includegraphics[width=7.5cm,height=6cm]{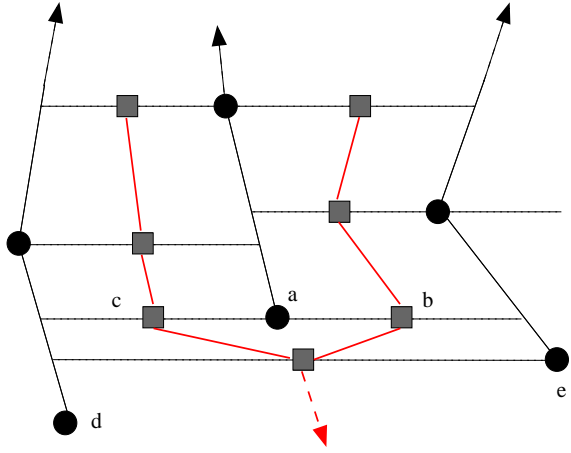}
\caption{Here is a picture of the DSF in upward direction and its dual forest
in downward direction. Vertices of the DSF are black circles whereas dual vertices are grey squares. In particular, the vertex $a$ produces left dual neighbour $c$ and right dual neighbour $b$. Both 
these two dual vertices, $b$ and $c$ take dual step to the same (dual) vertex.}
\label{fig:dual}
\end{center}
\end{figure}

The dual (or backward) path $\widehat{\pi}^{(y,s)}\in\widehat{\Pi}$ starting at $(y,s)$ is
constructed  by linearly joining the successive $\widehat{h}(\cdot)$ steps. 
Thus, $\widehat{{\cal X}}:=\{\widehat{\pi}^{(y,s)} : (y,s) \in \widehat{V}\}$ denotes the collection of all dual paths obtained from $\widehat{G}$.

%\begin{figure}[!ht]
%\begin{center}
%\psfrag{a}{\small{$\x$}}
%\psfrag{b}{\small{$\widehat{r}_{\x}$}}
%\psfrag{c}{\small{$\widehat{l}_{\x}$}}
%\psfrag{d}{\small{$\x_l$}}
%\psfrag{e}{\small{$\x_r$}}
%\includegraphics[width=7.5cm,height=6cm]{dual.eps}
%\caption{{\small \textit{Here is a picture of the DSF $\mathfrak{F}$ (in upward direction) and its dual forest $\widehat{\mathfrak{F}}$ (in downward direction). Vertices of the DSF are black circles whereas dual vertices are grey squares. In particular, the vertex $\x$ produces two dual vertices $\widehat{l}_{\x}$ and $\widehat{r}_{\x}$. On this picture, $(\widehat{r}_{\x})_r=\x_r$ and $(\widehat{r}_{\x})_l=\x_l$ with $\x_r(2)>\x_l(2)$: this implies that $\widehat{h}(\widehat{r}_{\x})=\widehat{l}_{\x_r}$. The same is true for $\widehat{l}_{\x}$.}}}\label{fig:dual}
%\end{center}
%\end{figure}

Let us recall that ${\cal X}_n={\cal X}_n(\gamma,\sigma)$ for $\gamma,\sigma>0$ and $n\geq 1$, is the collection of $n$-th order diffusively scaled paths where for a path $\pi$ with starting time $
\sigma_{\pi}$, the scaled path $\pi_n(\gamma,\sigma) : [\sigma_{\widehat{\pi}}/n^2\gamma, \infty ] \to [-\infty, \infty]$ is given by
\begin{equation}
\label{defi:ScaledPath}
\pi_n(\gamma,\sigma) (t) := \pi(n^2\gamma t)/n \sigma ~.
\end{equation}
In the same way, we define $\widehat{{\cal X}}_n=\widehat{{\cal X}}_n(\gamma,\sigma)$ as the collection of diffusively scaled dual paths. 
For any dual path $\widehat{\pi}$ with starting time $
\sigma_{\widehat{\pi}}$, the scaled dual path $\widehat{\pi}_n(\gamma,\sigma) : [-\infty , \sigma_{\widehat{\pi}}/n^2\gamma] \to [-\infty, \infty]$ is given by
\begin{equation}
\label{defi:ScaledDualPath}
\widehat{\pi}_n(\gamma,\sigma) (t) := \widehat{\pi}(n^2\gamma t)/n \sigma ~.
\end{equation}
For each $n\geq 1$, the closure of $\widehat{{\cal X}}_n$ in $(\widehat{\Pi},d_{\widehat{\Pi}})$, still denoted by $\widehat{{\cal X}}_n$, is a $(\widehat{{\cal H}}, {\cal B}_{\widehat{{\cal H}}})$-valued random variable.\\

Now we are ready to state our result:
\begin{theorem}
\label{thm:BW_Joint}
There exist $\sigma >0$ and $\gamma>0$ such that the sequence
$$
\big\{ \big( \overline{{\cal X}}_n(\gamma,\sigma) , \overline{\widehat{\cal X}}_n(\gamma,\sigma) \big) : \, n \geq 1 \big\}
$$
converges in distribution to $({\cal W}, \widehat{{\cal W}})$ as $({\cal H}\times \widehat{{\cal H}},{\cal B}_{{\cal H}\times \widehat{{\cal H}}})$-valued random variables as $n\rightarrow\infty$.
\end{theorem}

Recall that the perturbed DSF paths are non-crossing and 
the convergence criteria to the BW  for
 non-crossing path models are provided in  \cite{FINR04}.
  The reader may refer to \cite{SSS17} for a very complete overview on the topic.
Let $\Xi\subset\Pi$. For $t>0$ and $t_0,a,b\in\R$ with $a<b$, consider the counting random variable $\eta_\Xi(t_0,t;a,b)$ defined as
\begin{equation}
\label{def:EtaPtSet}
\eta_{\Xi}(t_0,t;a,b) := \# \big\{ \pi(t_0+t) :\, \pi \in \Xi , \, \sigma_{\pi}\leq t_0 \text{ and } \pi(t_0) \in [a,b] \big\}
\end{equation}
which considers all paths in $\Xi$, born before $t_0$, that intersect $[a,b]$ at time $t_0$ and counts the number of different positions these paths occupy at time $t_0+t$.  Theorem 2.2 of \cite{FINR04} lays out the following convergence criteria .

\begin{theorem}[Theorem 2.2 of \cite{FINR04}]
\label{thm:BwebConvergenceNoncrossing1}
Let $\{\Xi_n : n \in \N\}$ be a sequence of $({\mathcal H},B_{{\mathcal H}})$ valued random variables with non-crossing paths. Assume that the following conditions hold:
\begin{itemize}
\item[$(I_1)$] Fix a deterministic countable dense set ${\mathcal D}$  of $\R^2$. For each $\x \in {\mathcal D}$, there exists $\pi_n^{\x} \in \Xi_n$ such that for any finite set of points $\x^1, \dotsc, \x^k \in {\mathcal D}$, as $n \to \infty$, we have
$(\pi^{\x^1}_n, \dotsc, \pi^{\x^k}_n)$ converges in distribution to $(W^{\x^1}, \dotsc, W^{\x^k} )$, where $(W^{\x^1}, \dotsc, W^{\x^k} )$ denotes coalescing Brownian motions starting from the points $\x_1, \ldots, \x_k$.
\item[$(B_1)$] For all $t>0$, $\limsup_{n\rightarrow \infty}\sup_{(a,t_0)\in\R^2}\P(\eta_{\Xi_n}(t_0,t;a,a+\epsilon)\geq 2)\rightarrow 0$  as $\epsilon\downarrow 0$.
\item[$(B_2)$] For all $t>0$, $\frac{1}{\epsilon}\limsup_{n\rightarrow\infty}\sup_{(a,t_0)\in \R^2}\P(\eta_{\Xi_n}(t_0,t;a,a+\epsilon)\geq 3)\rightarrow 0$ as $\epsilon\downarrow 0$.
\end{itemize}
Then $\Xi_n$ converges in distribution to the standard Brownian web ${\mathcal W}$ as $n \to \infty$.
\end{theorem}

Let us first mention that for a sequence of $({\cal H}, {\cal B}_{{\cal H}})$-valued random variables $\{\Xi_n: n \in \N\}$ with non-crossing paths, Criterion $(I_1)$ implies tightness (see Proposition B.2 in the Appendix of \cite{FINR04} or Proposition 6.4 in \cite{SSS17}) and hence subsequential limit(s) always exists. Moreover, Criterion $(B_1)$ has in fact been shown to be redundant with $(I_1)$ for models with non-crossing paths (see Theorem 6.5 of \cite{SSS17}). Actually Condition $(I_1)$
implies that subsequential limit $\Xi$ contains coalescing Brownian motions 
starting from all rational vectors and hence contain a copy of the standard Brownian  web (${\mathcal W}$). Criterion $(B_2)$ ensures that the limiting random variable does not have extra paths and 
hence the limit must be ${\mathcal W}$.

Criterion $(B_2)$ is often verified by applying an FKG type correlation inequality
 together with an estimate on the distribution of the coalescence time between two paths. However, 
FKG is a strong property and very difficult to obtain for models  with complicate dependencies.
For a drainage network model with long range interactions Coletti \textit{et. al.} showed 
that FKG inequality holds partially and proved convergence to the Brownian web \cite{CFD09}.
We will follow a more robust technique developed in \cite{CSST19} and applicable for non-crossing path models. We apply Theorem 32 of \cite{CSST19} to 
obtain joint convergence for the DSF and its dual to the Brownian web and its dual.
\begin{theorem}
\label{thm:JtConvBwebGenPath}
Let $\{(\Xi_n, \widehat{\Xi}_n) : n\geq 1\}$ be a sequence of $({\cal H}\times\widehat{{\cal H}}, {\cal B}_{{\cal H}\times\widehat{{\cal H}}})$-valued random variables with non-crossing paths only, satisfying the following assumptions:
\begin{itemize}
\item[(i)] For each $n \geq 1$, paths in $\Xi_n$ do not cross (backward) paths in $\widehat{\Xi}_n$ almost surely: there does not exist any $\pi\in\Xi_n$, $\widehat{\pi}\in\widehat{\Xi}_n$ and $t_1,t_2\in (\sigma_{\pi},\sigma_{\widehat{\pi}})$ such that $(\widehat{\pi}(t_1) - \pi(t_1))(\widehat{\pi}(t_2) - \pi(t_2))<0$ almost surely.
\item[(ii)] $\{\Xi_n : n \in \N\}$ satisfies $(I_1)$.
\item[(iii)]  $\{(\widehat{\pi}_n(\sigma_{\widehat{\pi}_n}),\sigma_{\widehat{\pi}_n}) : \widehat{\pi}_n \in \widehat{\Xi}_n \}$, the collection  of starting points of all the backward paths in $\widehat{\Xi}_n$, as $n \to \infty$, becomes dense in $\R^2$.
\item[(iv)] For any sub sequential limit $({\cal Z},\widehat{{\cal Z}})$ of $\{(\Xi_n,\widehat{\Xi}_n) : n \in \N\}$, paths of ${\cal Z}$ do not spend positive Lebesgue measure time together with paths of $\widehat{{\cal Z}}$, i.e.,  almost surely there is no $\pi\in{\cal Z}$ and $\widehat{\pi}\in\widehat{{\cal Z}}$ such that $\int_{\sigma_\pi}^{\sigma_{\widehat{\pi}}} \mathbf{1}_{\pi(t)=\widehat{\pi}(t)} dt > 0$.
\end{itemize}
Then $(\Xi_n,\widehat{\Xi}_n)$ converges in distribution $({\cal W}, \widehat{{\cal W}})$ as $n \to \infty$.
\end{theorem}

It is useful to mention here that there are several other approaches to replace Criterion $(B_2)$. 
Long before, Criterion $(E)$ was proposed by Newman et al \cite{NRS05} which is applicable even for  models with crossing paths. \cite{SSS17} provided  a new criterion in Theorem 6.6 replacing $(B_2)$, called the \textit{wedge condition}.  Theorem \ref{thm:JtConvBwebGenPath} 
 appears as a slight generalization of Theorem 6.6 of \cite{SSS17} by considering the joint convergence and it replaces the wedge condition by the fact that no limiting primal and dual paths can spend positive Lebesgue time together.  In the next section we show that the conditions of Theorem \ref{thm:JtConvBwebGenPath} hold for the diffusively scaled DSF and its dual 
 $\{({\cal X}_n,\widehat{{\cal X}}_n) : n \in \N \}$. Finally we make the following remark 
 regarding the existence of bi-infinite path for the perturbed DSF for $d=2$. 
 We mention that following the arguments arguments of Section 4 of \cite{GRS04}, which is based 
 on a Burton-Keane argument, one can show that there is no bi-infinite path for the 
 perturbed DSF a.s. for any $d\geq 2$.
 \begin{remark}
 \label{rem:Bi-infinited2}
 From the construction of the dual graph it is evident that the DSF has a bi-infinite path if and only if the dual graph is not connected. If there are scaled dual paths which do not coalesce but converge to coalescing Brownian motions then there must be
 scaled forward paths entrapped between these scaled dual paths. Further, the joint convergence to the double Brownian web $(\WW, \widehat{\WW})$ forces that there must be a limiting forward Brownian path approximating this sequence of entrapped forward scaled paths. Further this limiting forward Brownian path must spend positive Lebesgue measure time together with a backward Brownian path which
  contradicts the properties of $(\WW, \widehat{\WW})$.
\end{remark}

\subsection{Verification of conditions of Theorem \ref{thm:JtConvBwebGenPath}}
\label{sebsec:DSFBwebFinalVerification}

In this section, we show that the sequence of diffusively scaled path families $\{({\cal X}_n,\widehat{{\cal X}}_n) : n \geq 1\}$ obtained from the DSF and its dual forest satisfies the conditions in Theorem \ref{thm:JtConvBwebGenPath}.

Conditions $(i)$ and $(iii)$ of Theorem \ref{thm:JtConvBwebGenPath} hold by construction. Indeed, paths of ${\cal X}$ do not cross (backward) paths of $\widehat{{\cal X}}$ with probability $1$ and 
the same holds for ${\cal X}_n$ and $\widehat{{\cal X}}_n$ for any $n \geq 1$.
 Moreover, the collection $\{(\widehat{\pi}_n(\sigma_{\widehat{\pi}_n}),\sigma_{\widehat{\pi}_n}) : \widehat{\pi}_n \in \widehat{\Xi}_n\}$ of all starting points of the scaled backward paths
 in $\widehat{\Xi}_n$ becomes dense in $\R^2$ as $n \to \infty$.
 
We now show that the condition $(ii)$ holds for the sequence $\{{\cal X}_n : n\geq 1\}$, i.e., Criterion $(I_1)$ of Theorem \ref{thm:BwebConvergenceNoncrossing1}. We first focus on a single path, $\pi^{\mathbf{0}}$ starting at the origin.
The main ingredient here is the construction of i.i.d. pieces through (marginal) renewal steps.    
As shown in Proposition \ref{prop:SinglePt_Rwalk}, the sequence of renewal steps
 $\{ \widehat{g_{\tau_j}}(\mathbf{0}) : j \geq 2\}$ breaks down the path
 $\pi^{\mathbf{0}}$ into independent pieces. Let us 
 scale $\pi^{\mathbf{0}}$ into $\pi^{\mathbf{0}}_n$ as in (\ref{defi:ScaledPath}) with
\begin{align}
\label{def:SigmaLambda}
\sigma^2 := \text{Var} \bigl( Y_2 = \widehat{g_{\tau_2}}(\mathbf{0})(1) -  
\widehat{g_{\tau_1}}(\mathbf{0})(1) \bigr)  \; \text{ and } \; 
\gamma := \E \bigl( \widehat{g_{\tau_2}}(\mathbf{0})(2) -  
\widehat{g_{\tau_1}}(\mathbf{0})(2)\bigr).
\end{align}
From now on, the diffusively scaled sequence $\{{\cal X}_n : n\geq 1\}$ is considered w.r.t. these parameters, but for ease of writing, we drop $(\gamma,\sigma)$ from our notation. Proposition  \ref{prop:SinglePt_Rwalk} together with Corollary \ref{cor:FirstCoOrdRW}  allow us to apply 
Donsker's invariance principle to show that $\pi^{\mathbf{0}}_n$ converges in distribution in $(\Pi,d_{\Pi})$ to $B^{\mathbf{0}}$ a standard Brownian motion started at $\mathbf{0}$ (for a similar argument see Proposition 5.2 in \cite{RSS16}).

While working with  multiple paths the essential idea is that the 
multiple paths of the DSF, when they are far away, can be approximated as independent DSF paths. 
Using Proposition \ref{prop:DSF_CoalescingTimeTail} we can show that 
when two DSF paths are close enough, they coalesce quickly. This strategy was 
first developed by Ferrari et al in \cite{FFW05} to deal with dependent paths with bounded range interactions. In \cite{CFD09}, Coletti \text{et al.} first extended this technique 
for dependent paths with long range interactions. Several other papers studied convergence to the Brownian web 
for dependent paths with long range interactions (\cite{SS13},  \cite{CV14}, \cite{RSS16}, \cite{VZ17}, \cite{CSST19}). 
In this paper we follow a robust tool developed in \cite{RSS16}  and \cite{CSST19} to deal with long range interactions in non-Markovian set up. 
Since the proof here is the same to that of \cite{CSST19} we do not provide 
the details here and for more details we refer the reader to  Section 5.1 of \cite{RSS16}  and Section 6.2.1 of \cite{CSST19}.

To show condition $(iv)$, we mainly follow Section 6.2.2 of \cite{CSST19} 
and  the coalescence time estimate given in Proposition \ref{prop:DSF_CoalescingTimeTail} 
serves as a key ingredient. Let $({\cal Z}, \widehat{\cal Z})$ be any subsequential limit of $\{({\cal X}_n, \widehat{\cal X}_n): n \geq 1\}$. By Skorohod's representation theorem we may assume that the convergence happens almost surely. With slight abuse of notation we continue to denote that subsequence 
by $\{({\cal X}_n, \widehat{\cal X}_n): n\geq 1\}$.

We have to prove that, with probability $1$, paths in ${\cal Z}$ do not spend positive Lebesgue measure time together with the dual paths in $\widehat{\cal Z}$. This means that for any $\delta>0$ and any integer $m\geq 1$, the probability of the event
\begin{equation*}
A(\delta, m) := \left\lbrace
\begin{array}{c}
\text{$\exists$ paths $\pi\in {\cal Z}, \widehat{\pi}\in \widehat{\cal Z}$ and $t_0\in\R$ s.t. $-m<\sigma_{\pi}<t_0<t_0 + \delta<\sigma_{\widehat{\pi}}<m$} \\
\text{and $-m<\pi(t) = \widehat{\pi}(t)<m$ for all $t\in [t_0, t_0+\delta]$}
\end{array}
\right\rbrace
\end{equation*}
has to be $0$. We recall that ${\cal Z}$ contains coalescing Brownian paths starting from every point in $\Q^2$ which is dense in $\R^2$. The non-crossing nature of paths in ${\cal Z}$ ensures that the above event $A(\delta, m)$ is measurable w.r.t. the countable Brownian paths in ${\cal Z}$ starting from all points in  $\Q^2$ .

To show that $\P(A(\delta,m))=0$, we introduce a generic event $B^{\epsilon}_n(\delta, m)$ defined as follows. Given an integer $m\geq 1$ and $\delta,\epsilon>0$,
\begin{equation*}
B^{\epsilon}_n(\delta, m) := \left\lbrace
\begin{array}{c}
\text{$\exists$ paths $\pi_1^n,\pi_2^n,\pi_3^n \in {\cal X}_n$ s.t. $\sigma_{\pi_1^n},\sigma_{\pi_2^n}\leq 0$, $\sigma_{\pi_3^n} \leq \delta$ and $\pi_1^n(0),\pi_1^n(\delta)\in [-m,m]$} \\
\text{with $|\pi_1^n(0)-\pi_2^n(0)|<\epsilon$ but $\pi_1^n(\delta) \not= \pi_2^n(\delta)$} \\
\text{and with $|\pi_1^n(\delta)-\pi_3^n(\delta)|<\epsilon$ but $\pi_1^n(2\delta) \not= \pi_3^n(2\delta)$} 
\end{array}
\right\rbrace ~.
\end{equation*}
The event $B^{\epsilon}_n(\delta, m)$ means that there exists a path $\pi_1^n$ localized in $[-m,m]$ at time $0$ as well as at time $\delta$ which is approached (within distance $\epsilon$) by two paths $\pi_2^n$ and $\pi_3^n$ respectively at times $0$ and $\delta$ while still being different from them respectively at time $\delta$ and $2\delta$. 
We comment that for our model the measurability of the event $B^{\epsilon}_n(\delta, m)$ is not a problem as the paths in ${\cal X}_n$ are countable. 

It was shown in Section 6.2.2 of \cite{CSST19} that to show $\P(A(\delta, m) = 0)$ it suffices 
to prove the following lemma.
\begin{lemma}
\label{lem:coalescenceDSF}
For any integer $m\geq 1$, real numbers $\epsilon,\delta>0$, there exists a constant $C_0(\delta,m)>0$ (only depending on $\delta$ and $m$) such that for all large $n$,
\begin{equation*}
\P( B^{\epsilon}_n (\delta, m) ) \leq C_0(\delta, m) \, \epsilon ~.
\end{equation*}
\end{lemma}

For the proof of Lemma \ref{lem:coalescenceDSF} we essentially follow \cite{CSST19} but the discrete nature of the perturbed point process makes the proof slightly easier. It 
is useful to mention that still we have to deal with the non-Markovian nature of DSF paths.

\begin{proof}[Proof of Lemma \ref{lem:coalescenceDSF}]

We recall that the DSF paths are non-crossing. For the paths $\pi^n_1, \pi^n_2 \in {\cal X}_n$ we assume that $\pi^n_1(0) < \pi^n_2(0) $. Now, for the DSF paths starting from the lattice points $(\lfloor \pi^n_1(0) n \sigma \rfloor - 1, 0)$ and $(\lfloor \pi^n_2(0) n \sigma \rfloor + 1, 0)$, we still have that  
$$
\pi^{(\lfloor \pi^n_1(0) n \sigma \rfloor - 1, 0)}(\lfloor n^2 \gamma \delta \rfloor) \neq \pi^{(\lfloor \pi^n_2(0) n \sigma \rfloor + 1, 0)}(\lfloor n^2 \gamma \delta \rfloor).
$$
Because of this observation, we can say that the event  $ B^{\epsilon}_n $ is contained in the following event 
$ D^{\epsilon}_n $ involving unscaled paths where: 
\begin{align*}
D^{\epsilon}_n  & := \bigl\{ \text{there exist }x, y, z \in \Z 
\text{ such that }  x \in [ - m n \sigma  -1,  m n \sigma + 1], 
| x - y | < n \epsilon \sigma \text{ and } \\ 
& \pi^{(x, - \delta)}(\lfloor n^2 \gamma \delta \rfloor) 
 \neq \pi^{(y, -\delta)}(\lfloor n^2 \gamma \delta \rfloor),
  | \pi^{(x,-\delta)}(\lfloor n^2 \gamma \delta \rfloor) - z | < n \epsilon \sigma,
 \pi^{(x,-\delta)}(2\lfloor n^2 \gamma \delta \rfloor) \neq 
\pi^{(z,-\delta)}(2\lfloor n^2 \gamma \delta \rfloor)\bigr\}.
\end{align*}
For $ \omega \in D^{\epsilon}_n$, suppose $ x , y $ are as in the definition above and assume that 
$ x < y $. Set $ l  = \max \{ x + j : \pi^{(x, 0)}(\lfloor n^2 \gamma \delta \rfloor)  = 
\pi^{(x+j, 0)}(\lfloor n^2 \gamma \delta \rfloor) \}
$.  Clearly, $ -  m n \sigma  - 1 \leq  x \leq l < y \leq  ( m+\epsilon) n \sigma + 1$
 and $ \pi^{(x, 0)}(\lfloor n^2 \gamma \delta \rfloor)
 = \pi^{(l, 0)}(\lfloor n^2 \gamma \delta \rfloor)
  < \pi^{(l+1,0)}(\lfloor n^2 \gamma \delta \rfloor)
  \leq \pi^{(y, 0)}(\lfloor n\delta \rfloor) $.
    Assume that $\lfloor \pi^{(x, 0)}(\lfloor n^2 \gamma \delta \rfloor) \rfloor = k  $ for some $ k \in \Z $. 
Then, $z$ in the definition above satisfies $ z \in (k - n \epsilon \sigma, k + n \epsilon \sigma)$
and by non-crossing property of paths we must have that 
$$
\pi^{( k -  \lfloor n \epsilon \sigma  \rfloor- 1,\lfloor n^2 \gamma \delta \rfloor )}(2\lfloor n^2 \gamma \delta \rfloor) \neq 
\pi^{(k +  \lfloor n \epsilon \sigma  \rfloor+ 1, \lfloor n^2 \gamma \delta \rfloor  )}(2\lfloor n^2 \gamma \delta \rfloor)
$$
 Thus, we must have $ \omega \in H^{(L)} (n, \delta, \epsilon) $ where  for $ l \in \Z $, 
\begin{align*}
  H_{l, k}^{(L)} (n, \delta, \epsilon) := & \bigl\{ \pi^{(l, 0)}(\lfloor n^2 \gamma \delta \rfloor) = k \neq \pi^{(l+1,0)}(\lfloor n^2 \gamma \delta \rfloor) \text{ and } \\
& \pi^{( k -  \lfloor n \epsilon \sigma  \rfloor- 1,\lfloor n^2 \gamma \delta \rfloor )}(2\lfloor n^2 \gamma \delta \rfloor) \neq 
\pi^{(k +  \lfloor n \epsilon \sigma  \rfloor+ 1, \lfloor n^2 \gamma \delta \rfloor  )}(2\lfloor n^2 \gamma \delta \rfloor)  \bigr\}; \\
H^{(L)} (n, \delta, \epsilon) := &  \cup_{ l =  - \lfloor  2 m n \sigma \rfloor  }
^{ \lfloor  2 m n \sigma \rfloor} \cup_{k \in \Z } H_{l, k}^{(L)} (n, \delta, \epsilon)  .  
\end{align*}

Similarly for $ \omega \in D^{\epsilon}_n$ such that  $ x > y $, set $ r  = \min \{ x - j : \pi^{(x, 0)}(\lfloor n^2 \gamma \delta \rfloor)   =  \pi^{(x-j, 0)}(\lfloor n^2 \gamma \delta \rfloor)   \} $. 
As earlier, $ \omega \in H^{(R)} (n, \delta, \epsilon) $ where 
for $ r \in \Z $, 
\begin{align*}
H_{r, k}^{(R)} (n, \delta, \epsilon) := & \bigl\{ \pi^{(r, 0)}(\lfloor n^2 \gamma \delta \rfloor)
= k \neq \pi^{(r-1, 0)}(\lfloor n^2 \gamma \delta \rfloor) \text{ and } \\
& \pi^{( k -  \lfloor n \epsilon \sigma  \rfloor- 1,\lfloor n^2 \gamma \delta \rfloor )}(2\lfloor n^2 \gamma \delta \rfloor) \neq 
\pi^{(k +  \lfloor n \epsilon \sigma  \rfloor+ 1, \lfloor n^2 \gamma \delta \rfloor)}(2\lfloor n^2 \gamma \delta \rfloor)   \bigr\}; \\
H^{(R)} (n, \delta, \epsilon) := & \cup_{ r =  - \lfloor 2 m n \sigma \rfloor  }
^{ \lfloor  2 m n \sigma \rfloor}  \cup_{k \in \Z } H_{r, k}^{(R)} (n, \delta, \epsilon) .  
\end{align*}

Thus, $ D^{\epsilon}_n \subseteq H^{(L)} (n, \delta, \epsilon) \cup H^{(R)} (n, \delta, \epsilon) $.
We fix $0 < 2\beta < \alpha < 1$. 
For $k \in \Z$, we define the event  $F_{n}(k)$ as
\begin{align*}
F_{n}(k) := \big\{ & k - \lfloor n \epsilon \sigma - n^{\alpha} \rfloor \leq \pi^{( k -  \lfloor n \epsilon \sigma  \rfloor- 1,\lfloor n^2 \gamma  \delta \rfloor)}(\lfloor n^2 \gamma \delta  + n^{\beta} \rfloor) \\
& \quad \leq \pi^{( k +  \lfloor n \epsilon \sigma  \rfloor + 1,\lfloor n^2 \gamma \delta \rfloor -\delta)}
(\lfloor n^2 \gamma \delta  + n^{\beta} \rfloor) \leq 
k + \lfloor n \epsilon \sigma - n^{\alpha} \rfloor \big\} ~.
\end{align*}
The event $F_{n}(k)$ asks that the two paths starting at $( k -  \lfloor n \epsilon \sigma  \rfloor- 1,\lfloor n^2 \gamma \delta \rfloor -\delta)$ and $( k +  \lfloor n \epsilon \sigma  \rfloor + 1,\lfloor n^2 \gamma \delta \rfloor -\delta)$ do not fluctuate too much till the  time $\lfloor n^2 \gamma \delta  + n^{\beta} \rfloor$. By Lemma \ref{lemma:UnexploredCone} we have that 
the DSF paths starting from the line $y = 0$ 
do not explore anything in the positive half-plane  $\mathbb{H}^+(\lfloor n^2 \gamma \delta \rfloor + m_d)$ 
till crossing the line $y = \lfloor n\delta \rfloor$ and hence for large $n$, do not explore anything  in the half-plane  
$\mathbb{H}^+(\lfloor n^2 \gamma \delta \rfloor + n^\beta)$. We observe 
 that on the event $F_n(k)^{c}$, at least one of the two paths starting from $( k -  \lfloor n \epsilon \sigma  \rfloor- 1,\lfloor n^2 \gamma \delta \rfloor )$ and $( k +  \lfloor n \epsilon \sigma  \rfloor + 1,\lfloor n^2 \gamma \delta \rfloor)$ 
 admits fluctuations larger than $n^{\alpha}$ on the time interval $[\lfloor n^2 \gamma \delta \rfloor,
 \lfloor n^2 \gamma \delta + n^{\beta}\rfloor]$. Following the same arguments as in Proposition 35 
 of \cite{CSST19}, we have that this event has a probability smaller than $C_0 \exp{(-C_1n^{(\alpha-\beta)/2)})}$. 
This gives us that  uniformly in $k$,
 the probability of the event $(F_{n}(k))^c $ decays to $0$ sub-exponentially.
 Hence in order to study the asymptotic behaviour of the probabilities it 
 suffices to focus on the event $D^{\epsilon}_n \cap F_n(k)$. 
 
 This motivates us to consider, 
\begin{align*}
  H_{l, k}^{(L),1} (n, \delta, \epsilon) := & \bigl\{ \pi^{(l, 0)}(\lfloor n^2 \gamma \delta \rfloor) = k \neq \pi^{(l+1, 0)}(\lfloor n^2 \gamma \delta \rfloor) \text{ and } \\
& \pi^{( k -  \lfloor n \epsilon \sigma  - n^{\alpha} \rfloor,\lfloor n^2 \gamma \delta + n^\beta \rfloor)}(2\lfloor n^2 \gamma \delta \rfloor) \neq 
\pi^{(k +  \lfloor n \epsilon \sigma + n^\alpha \rfloor, \lfloor n^2 \gamma \delta + n^\beta \rfloor )}(2\lfloor n^2 \gamma \delta \rfloor)  \bigr\}; \\
H^{(L),1} (n, \delta, \epsilon) := &  \cup_{ l =  - \lfloor  2 m n \sigma \rfloor  }
^{ \lfloor  2 m n \sigma \rfloor} \cup_{k \in \Z } H_{l, k}^{(L),1} (n, \delta, \epsilon)  .  
\end{align*}
Similarly the events $ H_{l, k}^{(R),1} (n, \delta, \epsilon)$ and $H^{(R),1} (n, \delta, \epsilon)$
are defined. Now we have 
$ D^{\epsilon}_n \cap F_n(k)
\subseteq H^{(L),1} (n, \delta, \epsilon) \cup H^{(R),1} (n, \delta, \epsilon) $.

We note that for all large $n$, the events $ \{ \pi^{(l, 0)}(\lfloor n^2 \gamma \delta \rfloor) = k \neq \pi^{(l+1, 0)}(\lfloor n^2 \gamma \delta \rfloor) \} $ and $ \{ \pi^{( k -  \lfloor n \epsilon \sigma  - n^{\alpha} \rfloor,\lfloor n^2 \gamma \delta + n^\beta \rfloor )}(2\lfloor n^2 \gamma \delta \rfloor) \neq 
\pi^{(k +  \lfloor n \epsilon \sigma + n^\alpha \rfloor, \lfloor n^2 \gamma \delta + n^\beta \rfloor )}(2\lfloor n^2 \gamma \delta \rfloor)  \} $ as the latter event depends only on 
perturbed points in $\Gamma(\lfloor n^2 \gamma \delta + n^\beta \rfloor)$.  
Proposition \ref{prop:DSF_CoalescingTimeTail} gives us that for all large $n$,
\begin{align*}
\label{eqn: GRScoalesceTimeTail}
 & \P \{ \pi^{( k -  \lfloor n \epsilon \sigma  - n^{\alpha} \rfloor,\lfloor n^2 \gamma \delta + n^\beta \rfloor )}(2\lfloor n^2 \gamma \delta \rfloor) \neq 
\pi^{(k +  \lfloor n \epsilon \sigma + n^\alpha \rfloor, \lfloor n^2 \gamma \delta + n^\beta \rfloor)}(2\lfloor n^2 \gamma \delta \rfloor) \} \\
& \leq  \frac{C_2 (2\lfloor n \sigma \epsilon \rfloor) }{\sqrt{ \lfloor n^2 \gamma  \delta 
- n^\beta \rfloor}}
  \leq C_3 (\delta) \epsilon 
 \end{align*}
where $ C_2, C_3 (\delta)  > 0 $ are constants. Hence, 
\begin{align*}
 \P ( H_{l, k}^{(L),1} (n, \delta, \epsilon)) \leq C_3 (\delta) \epsilon  ~  \P  \{  \pi^{(l,0)}(\lfloor n\delta \rfloor) = k \neq \pi^{(l+1, 0)}(\lfloor n\delta \rfloor)  \} . 
\end{align*}

Now, the events $ \{ \pi^{(l, 0)}(\lfloor n\delta \rfloor) = k \neq \pi^{(l+1, 0)}(\lfloor n\delta \rfloor)  \} $ are disjoint for distinct values of $ k $. Hence, 
\begin{align*}
\P   \bigl( \cup_{k \in \Z } H_{l, k}^{(L),1} (n, \delta, \epsilon) 
\bigr) & \leq \sum_{ k \in \Z} \P  \bigl( H_{l, k}^{(L),1} (n, \delta, \epsilon) \bigr) \\
& \leq C_3 (\delta) \epsilon  \sum_{ k \in \Z}  \P  \{  \pi^{(l, -\delta)}(\lfloor n^2 \gamma \delta \rfloor) = k \neq \pi^{(l+1, -\delta)}(\lfloor n^2 \gamma \delta \rfloor) \}  \\
& = C_3 (\delta) \epsilon ~  \P  \{  \pi^{(l, -\delta)}(\lfloor n^2 \gamma \delta \rfloor) \neq \pi^{(l+1, -\delta)}(\lfloor n^2 \gamma \delta \rfloor) \}.
\end{align*}
The above argument also holds for $  \cup_{k \in \Z } H_{r, k}^{(R),1} (n, \delta, \epsilon)$. 
Thus, combining the above terms and applying Proposition \ref{prop:DSF_CoalescingTimeTail}
\begin{align*}
\P ( D^{\epsilon}_n \cap F_n(k) ) & \leq  \P   \bigl( H^{(L),1} (n, \delta, \epsilon)  \bigr)  +  \P   \bigl( H^{(R),1} (n, \delta, \epsilon)  \bigr) \\
 & \leq  \sum_{ l =  - \lfloor 2 m n \sigma \rfloor  }^{ \lfloor  2 m n \sigma \rfloor} 
\P   \bigl( \cup_{k \in \Z } H_{l, k}^{(L),1} (n, \delta, \epsilon) \bigr) +  
\sum_{ r =  - \lfloor 2 m n \sigma \rfloor  }^{ \lfloor  2 m n \sigma \rfloor} 
\P   \bigl( \cup_{k \in \Z } H_{r, k}^{(R),1} (n, \delta, \epsilon) \bigr)   \\
& \leq 16 m \sqrt{n} \sigma  C_3 ( \delta ) \epsilon C_2 / \sqrt{\lfloor n^2 \gamma \delta \rfloor } \leq C_1 ( \delta, m ) \epsilon 
\end{align*}
for a proper choice of $ C_1 ( \delta, m) $. 
This completes the proof.
\end{proof}

{\bf Acknowledgements:}
We are grateful to the referees for their careful reading of the manuscript and their illuminating comments and suggestions. S.G. was supported in part by the MOE grant R-146-000-250-133.

\end{document}